\documentclass[fleqn,reqno,11pt,a4paper,final]{amsart}

\usepackage[a4paper,left=20mm,right=20mm,top=20mm,bottom=20mm,marginpar=20mm]{geometry}
\usepackage[english]{babel}
\usepackage{amsmath}
\usepackage{amssymb}
\usepackage{amsthm}
\usepackage{amscd}
\usepackage[utf8]{inputenc}
\usepackage{color}
\usepackage[english=american]{csquotes}
\usepackage[final]{graphicx}
\usepackage{hyperref}
\usepackage{calc}
\usepackage{mathptmx}
\usepackage{mathtools}
\usepackage{textcmds} 
\usepackage{tikz}
\usepackage{graphicx}
\usepackage{float}
\usepackage{stix}
\usepackage{enumitem}
\usepackage{dsfont}

\usepackage[backend=biber,bibencoding=utf8,firstinits,maxnames=4,style=alphabetic,isbn=false, doi=false, eprint= true, url= false,date=year]{biblatex}
\usepackage[english=american]{csquotes}
\bibliography{literature.bib}

\definecolor{luh-dark-blue}{rgb}{0.0, 0.313, 0.608}

\graphicspath{{../Pictures/}}

\numberwithin{equation}{section}

\newtheoremstyle{thmlemcorr}{10pt}{10pt}{\itshape}{}{\bfseries}{.}{10pt}{{\thmname{#1}\thmnumber{ #2}\thmnote{ (#3)}}}
\newtheoremstyle{thmlemcorr*}{10pt}{10pt}{\itshape}{}{\bfseries}{.}\newline{{\thmname{#1}\thmnumber{ #2}\thmnote{ (#3)}}}
\newtheoremstyle{remexample}{10pt}{10pt}{}{}{\bfseries}{.}{10pt}{{\thmname{#1}\thmnumber{ #2}\thmnote{ (#3)}}}
\newtheoremstyle{ass}{10pt}{10pt}{}{}{\bfseries}{.}{10pt}{{\thmname{#1}\thmnumber{ A#2}\thmnote{ (#3)}}}

\theoremstyle{thmlemcorr}
\newtheorem{theorem}{Theorem}
\numberwithin{theorem}{section}
\newtheorem{lemma}[theorem]{Lemma}
\newtheorem{corollary}[theorem]{Corollary}
\newtheorem{proposition}[theorem]{Proposition}

\newtheorem{definition}[theorem]{Definition}

\theoremstyle{thmlemcorr*}
\newtheorem*{theorem*}{Theorem}
\newtheorem{lemma*}[theorem]{Lemma}
\newtheorem{corollary*}[theorem]{Corollary}
\newtheorem{proposition*}[theorem]{Proposition}
\newtheorem{problem*}[theorem]{Problem}
\newtheorem{conjecture*}[theorem]{Conjecture}
\newtheorem{definition*}[theorem]{Definition}
\newtheorem{assumption*}[theorem]{Assumption}

\theoremstyle{remexample}
\newtheorem{remark}[theorem]{Remark}

\theoremstyle{ass}

\newcommand{\flux}[1]{\partial_x^3 u^{#1} - G_{#1}^{\prime\prime}(u^{#1})\partial_x u^{#1}}

\newcommand{\Fcal}{\mathcal{F}}

\DeclareMathOperator{\esssup}{ess\,sup}

\DeclareMathOperator{\dist}{dist}

\DeclareMathOperator{\supp}{supp}

\newcommand{\dd}{\;\mathrm{d}}

\newcommand{\N}{\mathbb{N}}
\newcommand{\R}{\mathbb{R}}

\newcommand{\Z}{\mathbb{Z}}

\newcommand{\loc}{\mathrm{loc}}

\newcommand{\eps}{\varepsilon}

\DeclareMathOperator*{\essinf}{ess\,inf}
\DeclareMathOperator*{\argmin}{arg\, min}

\def\div{\mathrm{div\,}}

\def\XXint#1#2#3{{\setbox0=\hbox{$#1{#2#3}{\int}$}
\vcenter{\hbox{$#2#3$}}\kern-.5\wd0}}

\renewcommand{\eps}{\varepsilon}
\renewcommand{\phi}{\varphi}

\begin{document}


\title[]{Non-Newtonian thin-film equations: global existence of solutions, gradient-flow structure and guaranteed lift-off}

\author{Peter Gladbach}
\address{\textit{Peter Gladbach:} Institute of Applied Mathematics, University of Bonn, Endenicher Allee~60, 53115 Bonn, Germany}
\email{gladbach@iam.uni-bonn.de}

\author{Jonas Jansen}
\address{\textit{Jonas Jansen:}  Centre for Mathematical Sciences, Lund University, P.O. Box 118, 221 00 Lund, Sweden}
\email{jonas.jansen@math.lth.se}

\author{Christina Lienstromberg}
\address{\textit{Christina Lienstromberg:}  Institute of Analysis, Dynamics and Modeling, University of Stuttgart, Pfaffenwaldring~57, 70569 Stuttgart, Germany}
\email{christina.lienstromberg@iadm.uni-stuttgart.de}

\begin{abstract}
We study the gradient-flow structure of a non-Newtonian thin film equation with power-law rheology. The equation is quasilinear, of fourth order and doubly-degenerate parabolic.
By adding a singular potential to the natural Dirichlet energy, we introduce a modified version of the thin-film equation.
Then, we set up a minimising-movement scheme that converges to global positive weak solutions to the modified problem. These solutions satisfy an energy-dissipation equality and follow a gradient flow. In the limit of a vanishing singularity of the potential, we obtain global non-negative weak solutions to the power-law thin-film equation
\begin{equation*}
    \partial_t u + \partial_x\bigl(m(u) |\partial_x^3 u - G^{\prime\prime}(u) \partial_x u|^{\alpha-1} \bigl(\partial_x^3 u - G^{\prime\prime}(u) \partial_x u\bigr)\bigr) = 0
\end{equation*}
with potential \(G\) in the shear-thinning (\(\alpha > 1\)), Newtonian (\(\alpha = 1\)) and shear-thickening case (\(0 <\alpha < 1\)). The latter satisfy an energy-dissipation inequality. Finally, we derive dissipation bounds in the case \(G\equiv 0\) which imply that solutions emerging from initial values with low energy lift up uniformly in finite time.
\end{abstract}
\vspace{4pt}

\maketitle








\section{Introduction}


The present paper is concerned with the quasilinear doubly-degenerate parabolic evolution equation 
\begin{equation}\label{eq:PDE}
\begin{cases}
    \partial_t u + \partial_x\bigl(m(u) |\partial_x^3 u - G^{\prime\prime}(u) \partial_x u|^{\alpha-1} \bigl(\partial_x^3 u - G^{\prime\prime}(u) \partial_x u\bigr)\bigr) = 0, & t>0,\ x \in \Omega, 
    \\
    \partial_x u =m(u) |\partial_x^3 u - G^{\prime\prime}(u) \partial_x u|^{\alpha-1} \bigl(\partial_x^3 u - G^{\prime\prime}(u) \partial_x u\bigr) = 0, & t>0,\ x \in \partial\Omega, \\
    u(0,x) = u_0, & x\in \Omega,
\end{cases}
\end{equation}
describing for instance the dynamics of the height $u=u(t, x)$ of a non-Newtonian incompressible thin liquid film on a solid bottom under presence of a potential $G$. Here, $\Omega \subset \R$ is a bounded interval with boundary $\partial\Omega$ and we assume the fluid film to be homogeneous in the horizontal $y$-direction, such that the problem reduces to one spatial dimension; see Figure \ref{fig:thin-film} for a sketch of the setting. At the lateral boundaries we prescribe a zero-contact-angle condition and a no-flux condition. Moreover, $m$ denotes the so-called mobility function and $\alpha > 0$ is the flow-behaviour exponent which describes the rheological properties of the fluid.

\begin{center}
\begin{figure}[h]
\begin{tikzpicture}[domain=0:3*pi, xscale=1.2, yscale=1.2] 
\draw[ultra thick, smooth, variable=\x, luh-dark-blue!20] plot (\x,{(0.6*cos(\x r)+0.6)}); 
\fill[luh-dark-blue!20] plot[domain=0:3*pi] (\x,{0}) -- plot[domain=3*pi:0] (\x,{(0.6*cos(\x r)+0.6)});
\draw[very thick,<->] (3*pi+0.4,0) node[right] {$x$} -- (0,0) -- (0,2) node[above] {$z$};
\draw[very thick,dashed,->] (0,0) -- (-0.6,-0.6) node[below] {$y$};
\draw[-] (0,-0.3) -- (0.3, 0);
\draw[-] (0.5,-0.3) -- +(0.3, 0.3);
\draw[-] (1,-0.3) -- +(0.3, 0.3);
\draw[-] (1.5,-0.3) -- +(0.3, 0.3);
\draw[-] (2,-0.3) -- +(0.3, 0.3); 
\draw[-] (2.5,-0.3) -- +(0.3, 0.3);
\draw[-] (3,-0.3) -- +(0.3, 0.3);
\draw[-] (3.5,-0.3) -- +(0.3, 0.3);
\draw[-] (4,-0.3) -- +(0.3, 0.3);
\draw[-] (4.5,-0.3) -- +(0.3, 0.3);
\draw[-] (5,-0.3) -- +(0.3, 0.3);
\draw[-] (5.5,-0.3) -- +(0.3, 0.3);
\draw[-] (6,-0.3) -- +(0.3, 0.3);
\draw[-] (6.5,-0.3) -- +(0.3, 0.3);
\draw[-] (7,-0.3) -- +(0.3, 0.3);
\draw[-] (7.5,-0.3) -- +(0.3, 0.3);
\draw[-] (8,-0.3) -- +(0.3, 0.3);
\draw[-] (8.5,-0.3) -- +(0.3, 0.3);
\draw[-] (9,-0.3) -- +(0.3, 0.3);
\end{tikzpicture}   
\caption{Cross section of fluid film on impermeable solid bottom.}
\label{fig:thin-film}
\end{figure}
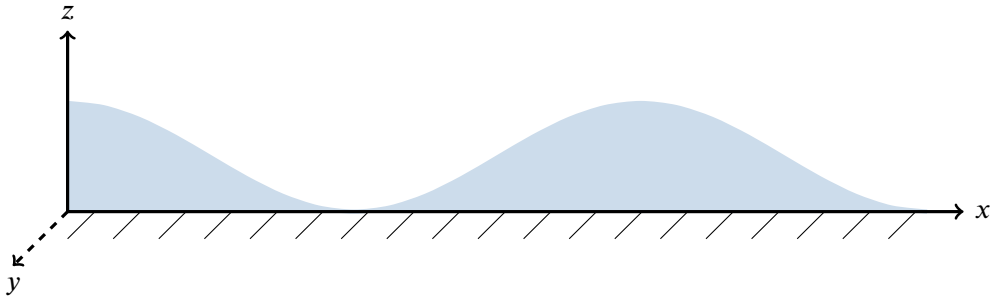 
\end{center}

Problem \eqref{eq:PDE} is a fourth-order quasilinear equation which is doubly-degenerate parabolic in the sense that we lose uniform parabolicity if either $m(u)=0$ or $\partial_x^3u = G''(u) \partial_x u$.  

The aim of the paper is to prove existence of non-negative global weak solutions and to study qualitative and quantitative properties via gradient-flow methods. 
In contrast to previous works, our approach relies on the introduction of an additive singular potential which guarantees positivity of approximate solutions. Furthermore, we prove that solutions that do not oscillate too much, lift up, as soon as they are close to zero. This property implies in addition the stability of steady states. 

More precisely, we introduce a modified version of \eqref{eq:PDE} by replacing the potential $G$ by a potential $G_\sigma$ which is singular for vanishing film heights. Then, we construct a minimising-movement scheme for the modified equation which yields global existence of positive weak solutions emerging from positive initial film heights. In particular, these solutions satisfy an energy-dissipation identity. 
Using standard energy methods, we obtain existence of global non-negative weak solutions in the limit of a vanishing modification parameter. The latter still satisfy an energy-dissipation inequality.
The modification used in the present paper differs from previous approaches which rely on removing the degeneracy or modifying the mobility to force approximate solutions to preserve positivity via entropy methods (see e.g. \cite{bernis_higher_1990, bertozzi_lubrication_1996,ansini_doubly_2004,lienstromberg_analysis_2022,lienstromberg_long-time_2022} for different regularisation schemes).

\subsection{Physical background and main assumptions}

Problem \eqref{eq:PDE} is derived from the non-Newtonian Navier--Stokes system by lubrication approximation in the limit of a vanishing film height \cite{greenspan_motion_1978, ockendon_viscous_1995, giacomelli_rigorous_2003, gunther_justification_2008}.

We assume the fluid to be a non-Newtonian power-law fluid, i.e. the constitutive law for the strain-dependent viscosity is given by
\begin{equation*}
    \mu(|\epsilon|) = \mu_0 |\epsilon|^{\frac{1}{\alpha}-1},
\end{equation*}
where $\epsilon$ is the strain rate, $\mu_0 > 0$ is the characteristic viscosity of the fluid and $\alpha > 0$ denotes the so-called flow-behaviour exponent. Note that for $\alpha=1$ the viscosity is constant and hence the fluid is Newtonian. For $0 < \alpha < 1$ the fluid is shear-thickening, while flow-behaviour exponents $\alpha > 1$ describe a shear-thinning rheological behaviour.

The function $m$ in \eqref{eq:PDE} is the so-called mobility function. It is a locally Lipschitz continuous function
\begin{equation*}
    m\colon\R \longrightarrow [0,\infty)
    \quad \text{such that} \quad  
    m(s) = 0, \quad s \leq 0, \quad \text{and} \quad m(s) > 0, \quad s>0.
\end{equation*}
Physically, the mobility is explicitly given once one specifies slippage conditions at the solid-fluid interface. Typically, it is of the form \(m(u)= u^n, n >0\). For instance, if we prescribe a no-slip condition, we obtain \(m(u)=u^{\alpha+2}\), which in the Newtonian case becomes $m(u) = u^3$ . In the case of the Navier-slip condition \cite{navier_memoire_1823}, we obtain \(m(u) =  \lambda u^{\alpha+1} + u^{\alpha+2}\), where \(\lambda > 0\) is the characteristic slip length. Note that, in both cases, the mobility is superlinear regardless of the flow-behaviour exponent \(\alpha >0\).

In contrast to the `standard' thin-film problems, we include in this paper a potential $G$ which might incorporate effects due to gravitational, surface or intermolecular forces, see \cite{de_gennes_wetting_1985} for more details on the modelling. Mathematically, the function $G\colon \R \to [0,\infty]$ is assumed to have the following properties:
\begin{enumerate}[label=($G$\arabic*)]
    \item\label{it:G1} $G \in C^2((0,\infty))$;
    \item\label{it:G2} $G$ is convex on $(0,\infty)$;
    \item\label{it:G3} $G(s) \geq 0,\ s \in \R$;
    \item\label{it:G4} $G$ is continuous in $s=0$.
\end{enumerate}

\begin{remark} \label{rem:assummptions_G}
Condition \ref{it:G4} can be replaced by the following
\begin{enumerate}[label=($G4$\alph*)]
    \item\label{it:G4a} $G(s) = 0,\ s \leq 0$, and $G(s) \gtrsim s^{-2}$, as $s \searrow 0$.
\end{enumerate}
In contrast to \ref{it:G4}, assumption \ref{it:G4a} ensures that functions $u \in H^1(\Omega)$ with finite energy $E[u] < \infty$ are strictly positive. Note that both conditions exclude mildly singular potentials of the form
\begin{equation*}
    G(s) = C s^\beta \mathbf{1}_{\{s > 0\}}, \quad \beta \in (-2,0).
\end{equation*}
For potentials of this form the problem becomes a free-boundary problem for the contact point. Steady states and travelling-wave solutions of the corresponding problem are studied in \cite{durastanti_spreading_2022,durastanti_thin-film_2022}.
\end{remark}

The natural energy for \eqref{eq:PDE} is given by a sum of the Dirichlet energy $\int_{\Omega} \tfrac12 |\partial_x u|^2\dd x$ and the potential energy $\int_\Omega G(u) \dd x$, that is by
\begin{equation*}
    E[u] = \int_{\Omega} \tfrac{1}{2} |\partial_x u|^2 \dd x + G(u) \dd x.
\end{equation*}
The Dirichlet energy approximates the length of the surface to first order and encapsulates the capillary forces on the surface. The potential energy may for instance reflect gravitational or intermolecular effects.
Moreover, sufficiently smooth (positive) solutions to \eqref{eq:PDE} dissipate energy:
\begin{equation}\label{eq:EDE_intro}
    E[u](t) 
    + \int_0^t \int_\Omega |\partial_x^3 u - G^{\prime\prime}(u) \partial_x u|^{\alpha+1} \dd x \dd t
    =
    E[u_0].
\end{equation}
This corresponds to testing the equation with $\phi = \partial_x^2 u - G'(u)$. It turns out that \eqref{eq:EDE_intro} plays an important role in two respects. 
First, this balance of viscous dissipation and the rate of change of the energy is the reason why \eqref{eq:PDE} admits a gradient-flow structure, cf. Section \ref{sec:gradient-flow_intro}.
Moreover, \eqref{eq:EDE_intro} provides us with the a-priori bounds needed to construct weak solutions.

Finally, solutions to \eqref{eq:PDE} conserve their mass in the sense that
\begin{equation*}
    \bar{u} 
    \coloneq 
    \fint_{\Omega} u_0(x) \dd x
    =
    \fint_{\Omega} u(t,x) \dd x, \quad t\geq 0.
\end{equation*}
This may be seen by testing the equation with the constant function $\phi\equiv 1$.  

Note that in the case $G\equiv 0$, equation \eqref{eq:PDE} reduces to the well-known doubly-degenerate thin-film equation
\begin{equation}\label{eq: thin-film}
\begin{cases}
    \partial_t u + \partial_x\bigl(u^{n} |\partial_x^3 u|^{\alpha-1} \partial_x^3 u \bigr) = 0, & t>0,\ x \in \Omega, 
    \\
    \partial_x u = u^{n} |\partial_x^3 u|^{\alpha-1} \partial_x^3 u = 0, & t>0,\ x \in \partial\Omega, \\
    u(0,x) = u_0, & x\in \Omega,
\end{cases}
\end{equation}
which is studied in \cite{king_two_2001,king_spreading_2001,ansini_doubly_2004}.

\subsection{Main results of the paper}

We briefly summarise the main results and techniques of this paper. Recall that our analysis is restricted to the one-dimensional problem. We prove that:
\begin{itemize}
    \item for positive solutions $u$, the thin-film equation \eqref{eq:PDE} is a gradient flow in $H^1(\Omega)$;
    \item there exists a global non-negative weak solutions to \eqref{eq:PDE} for all flow-behaviour exponents $\alpha > 0$;
    \item for mobilities of the form $m(u) = u^n$ with $n \in (0,2\alpha+2)$ and $G\equiv 0$, solutions with initial values $u_0 \in H^1(\Omega)$ with sufficiently small energy 
    $$E[u_0] < \tfrac{9}{2} \bar{u}_0^2$$
    lift up after a fixed positive time $t_0 > 0$, independent of $E[u_0]$. In particular, these solutions converge to the constant positive steady state $\bar{u}_0$ with explicit decay rates.
\end{itemize}

\medskip
\noindent\textbf{\textsc{Gradient-flow structure. }} 
We introduce a modified potential $G_\sigma\colon \R \to [0,\infty]$ with the following properties

\begin{enumerate}[label=($G_{\sigma}$\arabic*)]
    \item\label{it:Gsigma1} $G_\sigma \in C^2((0,\infty))$;
    \item\label{it:Gsigma2} $G_\sigma$ is convex on $\R_+$;
    \item\label{it:Gsigma3} $G_\sigma(s) \geq 0, \quad s \in \R$;
    \item\label{it:Gsigma4} $G_\sigma(s) = +\infty, \quad s < 0, \quad G_\sigma(s) \gtrsim s^{-2}, \quad s \leq \sigma, \quad \text{and} \quad G_\sigma(s) = G(s), \quad s \geq 2\sigma$,
\end{enumerate}
where $\sigma \in (0,1)$. The corresponding modified energy functional is given by
\begin{equation*}
    E^\sigma[u] = \int_{\Omega} \tfrac12 |\partial_x u|^2 + G_\sigma(u) \dd x.
\end{equation*}
For this energy functional and discretising in time, we iteratively construct a minimising-movement scheme 
\begin{equation*}
    u^{\sigma,h}(t+h) = \argmin_u E^\sigma[u] + \tfrac{1}{h} \text{dissipation}[u^{\sigma,h}(t),u]
\end{equation*}
with time step $h>0$. The corresponding Euler--Lagrange equation for the minimiser $u^{\sigma,h}$ is a time-discrete version of equation \eqref{eq:PDE} with potential $G_\sigma$. Sending $h \searrow 0$, we extract an accumulation point $u^\sigma$ that solves the modified power-law thin-film equation 
\begin{equation} \label{eq:PDE_mod_intro} \tag{$P_\sigma$}
\begin{cases}
    \partial_t u^\sigma + \partial_x\bigl(m(u^\sigma) |\partial_x^3 u^\sigma - G_\sigma^{\prime\prime}(u^\sigma) \partial_x u^\sigma|^{\alpha-1} \bigl(\partial_x^3 u^\sigma - G_\sigma^{\prime\prime}(u^\sigma) \partial_x u^\sigma\bigr)\bigr) = 0, & t>0,\ x \in \Omega, 
    \\
    \partial_x u^\sigma = m(u^\sigma) |\partial_x^3 u^\sigma - G_\sigma^{\prime\prime}(u^\sigma) \partial_x u^\sigma|^{\alpha-1} \bigl(\partial_x^3 u^\sigma - G^{\prime\prime}(u^\sigma) \partial_x u^\sigma\bigr) = 0, & t>0,\ x \in \partial\Omega, \\
    u^\sigma(0,x) = u_0, & x\in \Omega,
\end{cases}
\end{equation}
and is strictly positive for all times $t \geq 0$, due to condition \ref{it:Gsigma4}. Moreover, for all times $t,s \geq 0$ the accumulation points satisfy an energy-dissipation equality
\begin{equation*}
    E^{\sigma}[u^\sigma](t) 
    +
    \tfrac{\alpha}{\alpha+1} \int_s^t \int_{\Omega} \frac{|j^{\sigma}|^{\frac{\alpha+1}{\alpha}}}{m(u^{\sigma})^{\frac{1}{\alpha}}} \dd x \dd \tau
    +
    \tfrac{1}{\alpha+1}\int_s^t \int_\Omega 
    m(u^\sigma)
    |\partial_x^3 u^\sigma - G_\sigma''(u^\sigma) \partial_x u^\sigma|^{\alpha+1}
    \dd x \dd \tau
    =
    E^{\sigma}[u^\sigma](s).
\end{equation*}
In fact, we characterise all positive weak solutions to the modified equation $\eqref{eq:PDE}_{\sigma}$ via this identity.
A detailed derivation of the heuristic for the construction of the minimising-movement scheme is provided in Section \ref{sec:gradient-flow_intro}.

\medskip
\noindent\textbf{\textsc{Existence of global non-negative weak solutions. }} 
Using the above energy-dissipation equality for the modified problem \eqref{eq:PDE_mod_intro}, we derive uniform a-priori bounds for the solutions $u^\sigma$. These bounds allow to identify an accumulation point $u$ which is a solution to the original equation \eqref{eq:PDE} and  satisfies
\begin{equation*}
    E[u](t)
    +
    \int_0^t \int_\Omega 
    m(u)
    |\partial_x^3 u - G''(u) \partial_x u|^{\alpha+1}
    \dd x \dd s
    \leq
    E[u_0].
\end{equation*}
The main issue in the proof is the identification of the nonlinear limit flux, which relies on a localised version of Minty's trick.

\medskip
\noindent\textbf{\textsc{Lift-off. }} For this purpose, we restrict our analysis to the case $G\equiv 0$. Moreover, for ease of calculation, we use w.l.o.g. $\Omega=(0,1)$. 
Constant functions $\bar{u}$ are clearly stationary solutions to \eqref{eq: thin-film}. However, the parabola
\begin{equation*}
    v\colon \Omega \to \R,
    \quad
    v(x) = 1-x^2
\end{equation*}
is formally also a stationary solution to \(\eqref{eq: thin-film}_1\) with higher energy $E[v] > E[\bar{u}] =0$. Note that \(v\) does not solve \(\partial_x v=0\) on \(\partial\Omega\), but it appears as a minimiser of the energy among all functions with fixed mass and a root.

In the case $2(\alpha+1)>n$ we show that for any initial value $u_0\in H^1(\Omega)$ with the same mass $\bar{u}_0=\bar{v}$ and $E[u_0]<E[v]$, we have $\min_{x\in \bar{\Omega}} u(t,x) > \tfrac{\bar{u}_0}{3}$ for all $t\geq t_0(\alpha,n)$, where $t_0$ does not depend on $u_0$. This shows that $v$ is at most an unstable equilibrium of $\eqref{eq: thin-film}_1$.


\subsection{Formal derivation of the gradient-flow structure} \label{sec:gradient-flow_intro}

In this subsection we provide the heuristic for the gradient-flow structure of the power-law thin-film equation \eqref{eq:PDE} with potential $G$ in $H^1(\Omega)$.
It is well-known that the energy 
\begin{equation*}
    E[u] = \int_{\Omega} \tfrac12 |\partial_x u|^2 + G(u) \dd x
\end{equation*}
is a Lyapunov functional for \eqref{eq:PDE}, i.e. $E$ is non-increasing in time along sufficiently regular solutions to \eqref{eq:PDE}.

Note that conceptually the following arguments work in general dimension $d \geq 1$. Since the rigorous analysis is restricted to $d=1$, the following heuristic is presented for the one-dimensional case.

Due to the conservation of mass property $\bar{u} = \bar{u}_0$,
$u=u(t,x)$ solves the continuity equation
\begin{equation*}
    \begin{cases}
        \partial_t u + \partial_x j = 0, & x \in \Omega,
        \\
        j=0, & x \in \partial\Omega,
    \end{cases}
\end{equation*}
for a certain flux $j$.
Formally, the time-derivative of the energy is given by 
\begin{equation*}
    \begin{split}
    \frac{d}{dt} E[u](t)
    & =
    \int_\Omega \bigl(-\partial_x^2 u  + G'(u)\bigr) \partial_t u \dd x + \int_{\partial \Omega} \partial_x u \partial_t u \nu \dd \mathcal{H}^{0}
    \\
    & =
    \langle DE[u],-\partial_x j\rangle - \int_{\partial \Omega} \partial_x u \partial_x j \nu \dd \mathcal{H}^{0}\\
    & = -\int_{\Omega} (\partial_x^3 u - G''(u)\partial_x u) j \dd x - \int_{\partial \Omega} \partial_x u \partial_x j \nu \dd \mathcal{H}^{0}
    ,
    \end{split}
\end{equation*}
where $DE[u] = -\partial_x^2 u + G'(u)$ is the chemical potential.
In the steepest-descent gradient-flow formulation, $j(t,\cdot)$ is chosen for every time $t>0$ as the minimiser of the functional
\begin{equation*}
    \int_{\Omega} \tfrac{1}{p} \frac{|j|^p}{m(u)^q}\dd x 
    + 
    \langle DE[u],-\partial_x j\rangle - \int_{\partial \Omega} \partial_x u \partial_x j \nu \dd \mathcal{H}^{0},
\end{equation*}
where the first summand is the so-called dissipation potential, see for example \cite{jordan_variational_1998,ambrosio_gradient_2008}. A minimiser exists if and only if \(\partial_x u = 0\) on \(\partial \Omega\). In this case, we find that \(j(t,\cdot)\) satisfies
\begin{equation*}
    \frac{|j|^{p-2}j}{m(u)^q} 
    = 
    -\partial_x DE[u] 
    =
    \partial_x^3 u - G^{\prime\prime}(u) \partial_x u,
\end{equation*}
and hence
\begin{equation*}
    j 
    = 
    m(u)^{\frac{q}{p-1}} |\partial_x^3 u - G^{\prime\prime}(u) \partial_x u|^{\frac{1}{p-1}-1} \bigl(\partial_x^3 u - G^{\prime\prime}(u) \partial_x u\bigr).
\end{equation*}
Choosing \(p=\frac{\alpha+1}{\alpha}\) and \(q=\frac{1}{\alpha}\) and using the continuity equation, we find that \(u\) solves \eqref{eq:PDE}.

In order to construct a steepest-descent solution, we discretise time and use a minimising-movement scheme. For a time step \(h>0\) and a given state \(u^*=u(t)\) at time \(t\), the state of the steepest-descent at time \(t+h\) is given by the pair \((u(t+h),j(t))\) which minimises the functional
\begin{equation*}
    E[u] + h \tfrac{\alpha}{\alpha+1} \int_{\Omega}\frac{|j|^{\frac{\alpha+1}{\alpha}}}{m(u^*)^{\frac{1}{\alpha}}} \dd x 
\end{equation*}
over all pairs \((u,j)\) that satisfy the time-discrete continuity equation
\begin{equation*}
    \begin{cases}
        \frac{u-u^*}{h} + \partial_x j = 0, & x\in \Omega,\\
        j = 0, & x\in \partial\Omega.
    \end{cases}
\end{equation*}
Note that the mobility is sampled from the previous time so as to simplify the minimisation. Unlike the steepest-descent problem, the time-discrete problem is well-posed. The limit points, as \(h\searrow 0\), of the time-discrete problem are solutions to the steepest-descent and weak solutions to \eqref{eq:PDE}, as long as they remain positive or under certain assumptions on \(G\), cf. Remark \ref{rem:assummptions_G}. In this case, in fact we characterise all weak solutions to \eqref{eq:PDE} via the energy-dissipation equality
\begin{equation*}
    E[u](t)
    + \tfrac{\alpha}{\alpha+1} \int_{s}^{t} \int_{\Omega} \frac{|j|^{\frac{\alpha+1}{\alpha}}}{m(u)^\frac{1}{\alpha}} \dd x \dd \tau 
    + \tfrac{1}{\alpha+1} \int_s^t \int_{\Omega} m(u) |\partial_x^3 u - G''(u)\partial_x u|^{\alpha+1} \dd x \dd \tau 
    = E[u](s)
\end{equation*}
for almost all times \(t\geq s \geq 0\). Sufficiently smooth solutions to the energy-dissipation equality are solutions to the steepest-descent problem, and vice versa. To see this, calculate
\begin{equation*}
    \begin{split}
    E[u](t) - E[u](s) & = \int_s^t \frac{d}{d\tau} E[u](\tau) \dd \tau \\
    & = \int_{s}^{t} \langle \partial_x DE[u](\tau), j(\tau)\rangle \dd \tau \\
    & =  \int_{s}^{t} \langle -\partial_x^3 u(\tau) + G''(u(\tau))\partial_x u(\tau), j(\tau)\rangle \dd \tau \\
    &\geq 
    - \tfrac{\alpha}{\alpha+1} 
    \int_{s}^{t} \int_{\Omega} \frac{|j|^{\frac{\alpha+1}{\alpha}}}{m(u)^\frac{1}{\alpha}} \dd x \dd \tau - \tfrac{1}{\alpha+1} \int_s^t \int_{\Omega} m(u) |\partial_x^3 u - G''(u)\partial_x u|^{\alpha+1} \dd x \dd \tau
    \end{split}
\end{equation*}
by Young's inequality. Equality can be achieved if and only if
\begin{equation*}
    j = m(u) |\partial_x^3 u - G''(u)\partial_x u|^{\alpha-1}\bigl(\partial_x^3 u - G''(u)\partial_x u\bigr)
\end{equation*}
for almost all times.

In \cite{lisini_cahnhilliard_2012}, the authors consider a similar problem for concave mobilities \(m(s) = s^n\) and \(n\in \bigl[\tfrac{1}{2},1\bigr]\). In that case, they show that there is an underlying metric framework for the steepest descent. We consider more general mobilities which include the relevant physical cases which are typically superlinear. In this more general framework, there is no underlying Wasserstein-type metric space, since the action functional
\begin{equation*}
    \int_{0}^{1} \int_{\Omega} \frac{|j|^{\frac{\alpha+1}{\alpha}}}{m(u)^{\frac{1}{\alpha}}} \dd x \dd t
\end{equation*}
is nonconvex. The induced metric
\begin{equation*}
\begin{split}
    d_m^{\frac{\alpha+1}{\alpha}}(u_0, u_1) &\coloneq 
    \inf\biggl\{ \int_0^1 \int_\Omega \frac{|j|^{\frac{\alpha+1}{\alpha}}}{m(u)^{\frac{1}{\alpha}}}\dd x\dd t;\, \partial_t u + \div j = 0,\,
    j\cdot n = 0\text{ on } \partial\Omega,
    \\ 
    &  \qquad \qquad u\geq 0,\, u(0,x) = u_0(x),\, u(1,x) = u_1(x)\biggr\}
\end{split}
\end{equation*}
is identically zero, as we will demonstrate in Proposition \ref{prop:metric-degenerate} in Appendix \ref{app:BenamouBrenier}.

\subsection{Related results} We briefly comment on results related to our problem which are available in the literature.

\medskip
\noindent\textbf{\textsc{Thin-film equation with potential. }}
In the case $\alpha=1$ the problem \eqref{eq:PDE} reduces to the Newtonian thin-film problem
\begin{equation} \label{eq:PDE_Newtonian}
\begin{cases}
    \partial_t u + \partial_x\bigl(m(u) \bigl(\partial_x^3 u - G^{\prime\prime}(u) \partial_x u\bigr)\bigr) = 0, & t>0,\ x \in \Omega, 
    \\
    \partial_x u = m(u) \bigl(\partial_x^3 u - G^{\prime\prime}(u) \partial_x u\bigr) = 0, & t>0,\ x \in \partial\Omega, \\
    u(0,x) = u_0, & x\in \Omega,
\end{cases}
\end{equation}
with potential. Equation \eqref{eq:PDE_Newtonian} has been introduced in \cite{de_gennes_wetting_1985} with a potential describing the interplay of surface forces, gravitational forces and intermolecular forces such as Van-der-Waals forces. In \cite{durastanti_spreading_2022} the  authors study existence, uniqueness and contact-point behaviour of stationary solutions to the Newtonian thin-film equation \eqref{eq:PDE_Newtonian} with mildly singular potential. 
In \cite{durastanti_thin-film_2022} the same authors are concerned with travelling-wave solutions. In particular, the authors propose to study the Newtonian thin-film equation with mildly singular potential as a remedy for the so-called no-slip paradox \cite{huh_hydrodynamic_1971,dussan_v_motion_1974}.

\medskip
\noindent\textbf{\textsc{Gradient-flow structure. }}
In the Newtonian case $\alpha=1$ and with zero potential $G\equiv 0$ the problem \eqref{eq:PDE_Newtonian} in turn reduces to the well-known thin-film equation
\begin{equation} \label{eq:Newtonian_PDE_intro}
\begin{cases}
    \partial_t u + \partial_x\bigl(m(u) \partial_x^3 u\bigr) = 0, & t>0,\ x \in \Omega, 
    \\
    \partial_x u = m(u) \partial_x^3 u = 0, & t>0,\ x \in \partial\Omega, \\
    u(0,x) = u_0, & x\in \Omega.
\end{cases}
\end{equation}
It is well-known (see \cite{almgren_singularity_1996,otto_lubrication_1998,giacomelli_rigorous_2003}, where the latter two assume that \(d=1\)) that for Newtonian fluids in the setting of Hele-Shaw flows, given by flow-behaviour exponent \(\alpha=1\) and mobility exponent \(n=1\), the thin-film equation is a gradient flow with respect to the Dirichlet energy
\begin{equation*}
    E[u](t) = \int_{\Omega} \tfrac{1}{2} |\partial_x u(t)|^2 \dd x.
\end{equation*}
A numerical gradient-flow scheme -- discrete both in time and space -- for the Newtonian thin-film equation with general mobility exponents in one and two space dimensions is studied in \cite{grun_nonnegativity_2000}. Numerical schemes for more advanced geometries are studied for instance in \cite{rumpf_numerical_2013} and \cite{vantzos_functional_2017}.

In \cite{lisini_cahnhilliard_2012}, the Newtonian thin-film equation with mobility exponents \(n\in (0,1]\) in dimension \(d=1\) is studied as a gradient flow in weighted Wasserstein spaces. All these results have in common that the dissipation potential turns out to be jointly convex in the film height and the flux, as we will see later.

In the non-Newtonian case $\alpha > 1$ and with zero potential $G\equiv 0$ the problem \eqref{eq:PDE} is the power-law thin-film equation 
\begin{equation} \label{eq:PL_intro}
\begin{cases}
    \partial_t u + \partial_x\bigl(m(u) |\partial_x^3 u|^{\alpha-1} \partial_x u\bigr) = 0, & t>0,\ x \in \Omega, 
    \\
    \partial_x u = m(u) |\partial_x^3 u|^{\alpha-1} \partial_x^3 u = 0, & t>0,\ x \in \partial\Omega, \\
    u(0,x) = u_0, & x\in \Omega.
\end{cases}
\end{equation}
To the best of our knowledge there are no results on the gradient-flow structure on \eqref{eq:PL_intro} available.

\medskip
\noindent\textbf{\textsc{Weak solutions and stability for thin-film problems. }}
Weak solutions to Newtonian thin-film equation \eqref{eq:Newtonian_PDE_intro} have been extensively studied in the literature since the seminal work \cite{bernis_higher_1990}. Since in our paper we are neither concerned with the Newtonian case nor with higher dimensions, we mention only the further contributions \cite{beretta_nonnegative_1995,bertozzi_lubrication_1996,bertsch_thin_1998}.

Existence of global non-negative weak solutions to the doubly nonlinear power-law thin-film equation \eqref{eq:PL_intro} is treated in \cite{ansini_doubly_2004} for flow-behaviour exponents in the shear-thinning regime, where the authors do also (besides other qualitative properties) prove convergence to the mean as time tends to infinity. For positive initial data, short-times existence of positive weak solutions and convergence rates for all flow-behaviour exponents are provided in \cite{jansen_long-time_2022}.

Lift-off properties for the Newtonian thin-film equation in the partial-wetting regime are studied in \cite{cuesta_self-similar_2018}, where the authors show that self-similar solutions emerging from initial values with isolated touch-down point lift up. 

We briefly mention that the methods used in our work in order to treat the power-law thin-film equation with potential, does also apply to the so-called Ellis thin-film equation, see for instance \cite{weidner_contactline_1994,ansini_shear-thinning_2002,lienstromberg_local_2020,jansen_long-time_2022}.


\subsection{Outline of the paper}

The structure of the present paper is as follows: in Section \ref{sec:mms} we derive the minimising-movement scheme for the modified problem \eqref{eq:PDE_mod_intro}. 
Moreover, we study the properties of the interpolations of the time-discrete flow. Using the De Giorgi technique, we derive an optimal discrete energy-dissipation equality and use this to prove a-priori bounds which yield convergence to a limit.

While from the results of Section \ref{sec:mms} it can already be deduced that positive solutions to the power-law thin-film equation have a gradient-flow structure, in Section \ref{sec:limit} we study the limit \(\sigma\to 0\). Using the energy-dissipation equality, we derive a-priori bounds. Combining these with a localised version of Minty's trick, we obtain convergence to global non-negative weak solutions to \eqref{eq:PDE} that satisfy an energy-dissipation inequality.

In Section \ref{sec:lifting} we study lifting properties in the case of potential \(G\equiv 0\). We use dissipation bounds to show that solutions with low initial energy lift up after a fixed positive time.

Finally, in Appendix \ref{app:BenamouBrenier} we demonstrate that a metric framework for gradient flows degenerates in the case of superlinear mobilities.


\subsection{Notation}

As above, $\Omega \subset \R$ denotes an open and bounded interval with boundary $\partial\Omega$. For \(k\in \N\) and \(p\in [1,\infty]\) we use the notation \(W^k_p(\Omega)\) for the standard Sobolev space with norm
\begin{equation*}
    \|v\|_{W^k_p(\Omega)} = \left(\sum_{0\leq|\alpha|\leq k} 
    \|\partial_{\alpha} v\|_{L_p(\Omega)}^p\right)^{1/p}.
\end{equation*}

Let \(\Omega\subset \R\) an interval. To account for the Neumann-type boundary conditions of the solutions to the power-law thin-film equation, we further introduce the notation
\begin{equation*}
    W^{k}_{p,B}(\Omega)
    =
    \begin{cases}
        \bigl\{v \in W^{k}_p(\Omega);\, 
        v_x = v_{xxx} = 0 \text{ on } \partial\Omega\bigr\}, 
        &k=4,
        \\[1ex]
        \bigl\{v \in W^{k}_p(\Omega);\, 
        v_x = 0 \text{ on } \partial\Omega\bigr\}, 
        & k \in \{2,3\},
        \\[1ex] 
        W^{k}_p(\Omega), & k \in \{0,1\}.
    \end{cases}
\end{equation*}
The spaces $W^{k}_{p,B}(\Omega)$, \(k\in \{0,1,\ldots,4\}\), are closed
linear subspaces of $W^{k}_{p}(\Omega)$.


\section{Minimising-movement scheme for the modified thin-film equation}\label{sec:mms}

In this section we define a minimising-movement scheme for the power-law thin-film equation
\begin{equation}\label{eq:PDE_mod}\tag{$P_\sigma$}
\begin{cases}
    \partial_t u^\sigma + \partial_x\bigl(m(u^\sigma) |\partial_x^3 u^\sigma - G_\sigma^{\prime\prime}(u^\sigma) \partial_x u^\sigma|^{\alpha-1} \bigl(\partial_x^3 u^\sigma - G_\sigma^{\prime\prime}(u^\sigma) \partial_x u^\sigma\bigr)\bigr) = 0, & t>0,\ x \in \Omega, 
    \\
    \partial_x u^\sigma = m(u^\sigma) |\partial_x^3 u^\sigma - G_\sigma^{\prime\prime}(u^\sigma) \partial_x u^\sigma|^{\alpha-1} \bigl(\partial_x^3 u^\sigma - G^{\prime\prime}(u^\sigma) \partial_x u^\sigma\bigr) = 0, & t>0,\ x \in \partial\Omega, \\
    u^\sigma(0,x) = u_0, & x\in \Omega,
\end{cases}
\end{equation}
with a modified potential $G_\sigma,\, \sigma \in (0,1)$, having the properties \ref{it:Gsigma1}--\ref{it:Gsigma4}. We show that, for positive initial values $u_0 \in H^1(\Omega),\, u_0 > 0$, problem \eqref{eq:PDE_mod} has a gradient-flow structure and admits global positive weak solutions.
As above, the mobility $m$ is a locally Lipschitz continuous function 
\begin{equation*}
    m\colon\R \to [0,\infty)
    \quad \text{such that} \quad  
    m(s) = 0, \quad s \leq 0,
    \quad \text{and} \quad
    m(s) > 0, \quad s > 0.
\end{equation*}
The energy functional corresponding to the modified problem \eqref{eq:PDE_mod} is given by
\begin{equation*}
    E^\sigma[u] 
    = 
    \int_{\Omega} \tfrac12 |\partial_x u|^2 + G_\sigma(u) \dd x.
\end{equation*}

For better readability, we introduce the notation
    \[\Psi(s) = |s|^{\alpha-1} s, \quad \alpha >0,\ s\in \R.\]


By definition of $G_\sigma$ we obtain that any $v \in H^1(\Omega)$ with finite energy $E^\sigma[v]$ is necessarily positive.

\begin{lemma}\label{lem:positivity_sigma}
Let $\sigma \in (0,1)$ and let $G_\sigma\colon \R \to \R$ satisfy \ref{it:Gsigma1}--\ref{it:Gsigma4}. If $v \in H^1(\Omega)$ satisfies $E^\sigma[v] \leq C$ for some constant $C >0$, then there exists a constant $c_{\sigma,C} > 0$ such that
\begin{equation*}
    \min_{x \in \bar{\Omega}} v \geq c_{\sigma,C} > 0.
\end{equation*}
\end{lemma}

\begin{proof}
We first observe that finite energy implies positivity. Indeed, since $H^1(\Omega) \hookrightarrow C^\frac{1}{2}(\bar{\Omega})$, every $v \in H^1(\Omega)$ with $v(x_0) = 0$ for some $x_0 \in \bar{\Omega}$ satisfies 
\begin{equation*}
    E^\sigma[v] 
    \geq 
    \int_{B(x_0)} |v(x)|^{-2} \dd x
    =
    \int_{B(x_0)} |v(x)-v(x_0)|^{-2} \dd x 
    \geq 
    c \int_{B(x_0)} |x-x_0|^{-1} \dd x
    = \infty,
\end{equation*}
where $B(x_0)$ is a small neighbourhood of $x_0 \in \bar{\Omega}$.
Assume that the assertion is false and fix $C > 0$. Then there exist sequences $(v_k)_k \subset H^1(\Omega)$ with $E^\sigma[v_k] \leq C$ and $(x_k)_k \subset \bar{\Omega}$ with 
\begin{equation*}
    v_k(x_k) \leq \tfrac{1}{k},
    \quad 
    k \in \N.
\end{equation*}
The assumption $E^\sigma[v_k] \leq C$ implies that there is a function $\tilde{v} \in H^1(\Omega)$ and a subsequence (not relabelled) with
\begin{equation*}
    v_k\ \xrightharpoonup{\phantom{\text{wi}}} \tilde{v}  
    \quad \text{in } H^1(\Omega).
\end{equation*}
Since the embedding $H^1(\Omega) \subset C(\bar{\Omega})$ is compact, we obtain (maybe again up to a subsequence)
\begin{equation}\label{eq:unif_conv_positivity}
    v_k \longrightarrow \tilde{v} 
    \quad \text{in } C(\bar{\Omega}).
\end{equation}
Lower semicontinuity of the norm and Fatou's lemma imply
\begin{equation*}
    E^\sigma[\tilde{v}]
    \leq
    \liminf_{k \to \infty} E^\sigma[v_k]
    \leq 
    C.
\end{equation*}
Since $\Omega$ is bounded, there exists a corresponding subsequence $(x_k)_k$ with $x_k \to \tilde{x}$. Due to the uniform convergence \eqref{eq:unif_conv_positivity} we finally obtain
\begin{equation*}
    \tilde{v}(\tilde{x})
    =
    \lim_{k \to \infty} v_k(x_k)
    =
    0,
\end{equation*}
contradicting that $E^\sigma[\tilde{v}] \leq C$.
\end{proof}


\subsection{Existence, uniqueness and regularity for a single time step}

We start setting up one time step of the minimising-movement scheme. We fix a time-step size \(h>0\). If at a time \(t\geq 0\) we are in the state \(u^*\), we define the next iteration, that is the approximation at time \(t+h\), to be the minimiser of the functional
\begin{equation} \label{eq:functional_mms}
    \Fcal_{u^\ast}^{h,\sigma}[u,j] 
    =  
    E^\sigma[u] + h\tfrac{\alpha}{\alpha+1}\int_{\Omega} \frac{|j|^{\frac{\alpha+1}{\alpha}}}{m(u^\ast)^{\frac{1}{\alpha}}}\dd x.
\end{equation}
The minimisation runs over all pairs \((u,j) \in H^1(\Omega) \times L_{\frac{\alpha+1}{\alpha}}(\Omega)\) that satisfy
\begin{equation}
	\begin{cases}
		\partial_x j + \frac{u - u^\ast}{h} = 0, & x \in \Omega, \\
		j= 0, & x \in \partial\Omega.
	\end{cases}
\end{equation}
Note that in the second integral on the right-hand side of \eqref{eq:functional_mms} we use $m(u^\ast)$ in order to avoid a lack of convexity. Heuristically this is no major change since solutions turn out to be continuous in time, and hence \(m(u(t+h))\) and \(m(u(t))\) are very close.

Before setting up the minimising-movement scheme, we prove existence, uniqueness and regularity properties of minimisers of the functional $\Fcal_{u^\ast}^{h,\sigma}$.

\begin{definition}\label{def:flow-equation}
	Let $u^\ast \in H^1(\Omega)$ and \(h>0\).
	We say that a pair $(u,j) \in H^1(\Omega)\times L_{\frac{\alpha+1}{\alpha}}(\Omega)$ solves the \emph{flow equation}
	\begin{equation} \label{eq:flow_equation}
		\begin{cases}
			\frac{u - u^\ast}{h} + \partial_x j = 0, & x\in  \Omega, \\
			j= 0, & x \in \partial\Omega,
		\end{cases}
	\end{equation}
	if the equation
	\begin{equation}\label{eq:weak-flow-equation}
		-\int_\Omega j\cdot \partial_x \phi\dd x + \tfrac{1}{h}\int_\Omega (u - u^\ast) \phi\dd x = 0
	\end{equation}
	is satisfied for all $\phi \in W^1_{\alpha+1}(\Omega)$.
\end{definition}


Now, for fixed $u^\ast \in H^1(\Omega)$ with finite energy $E^\sigma[u^\ast] < \infty$, we define the minimisation problem
\begin{equation} \label{eq:mp_sigma}
    \begin{cases}
     \text{find } (v,k) = \argmin_{(u,j)} \Fcal^{h,\sigma}_{u^\ast}[u,j] & \\
     \text{s.t. } (u,j) \text{ solves the flow equation \eqref{eq:flow_equation} in the sense of Definition \ref{def:flow-equation}.}
    \end{cases}
\end{equation}
Recall that the finite-energy assumption $E^\sigma[u^\ast] < \infty$ guarantees that $u^\ast$ and thus $m(u^\ast)$ are strictly positive.


\begin{remark}[Conservation of mass and Poincar\'e inequality] \label{rem:cons_mass_poincare}
	For every $u^\ast \in H^1(\Omega)$, a solution $(u,j) \in H^1(\Omega)\times L_{\frac{\alpha+1}{\alpha}}(\Omega;\R^d)$ to the flow equation \eqref{eq:flow_equation} conserves its mass in the sense that
	\begin{equation*}
	    \bar{u} = \fint_\Omega u\dd x = \fint_\Omega u^\ast\dd x.
	\end{equation*}
	Indeed, this follows immediately by choosing \(\phi \equiv 1\) in \eqref{eq:weak-flow-equation}.
	In particular, $u$ satisfies the Poincar\'e inequality 
	\begin{equation} \label{eq:poincare}
	    \|u-\bar{u}\|_{L_2(\Omega)} \leq C \|\partial_x u\|_{L_2(\Omega)} ,
	\end{equation}
	where $C > 0$ is a positive constant that depends only on $\Omega$ and $\bar{u} = \bar{u}^\ast$. 
\end{remark}


The following proposition guarantees existence and uniqueness of minimisers of the functional \(\Fcal_{u^\ast}^{h,\sigma}\) for a given initial datum \(u^* \in H^1(\Omega)\). 
Recall that $\Psi(s) = |s|^{\alpha-1} s$ for $s \in \R$ and $\alpha > 0$.

\begin{proposition}\label{prop:existence_minimiser}
	Let $u^\ast \in H^1(\Omega)$ with finite energy $E^\sigma[u^\ast] < \infty$ and fix \(h,\ \sigma >0\).
	There exists a unique solution $(u^{h,\sigma},j^{h,\sigma}) \in H^1(\Omega)\times L_{\frac{\alpha+1}{\alpha}}(\Omega)$ to the minimisation problem \eqref{eq:mp_sigma}.
	\\
	The minimiser has the additional regularity $(u^{h,\sigma},j^{h,\sigma}) \in W^3_{\alpha+1,B}(\Omega)\times L_{\frac{\alpha+1}{\alpha}}(\Omega)$ and solves the Euler--Lagrange equation
	\begin{equation}\label{eq:elliptic_bvp}
		\begin{cases}
			\frac{u^{h,\sigma} - u^\ast}{h} + \partial_x j^{h,\sigma} = 0, & x \in \Omega, 
			\\
			j^{h,\sigma} = m(u^*) \Psi\bigl(\partial_x^3 u^{h,\sigma} - G^{\prime\prime}(u^{h,\sigma}) \partial_x u^{h,\sigma}\bigr), & x \in \Omega, 
			\\
			\partial_x u^{h,\sigma} = j^{h,\sigma} = 0, & x\in \partial\Omega.
		\end{cases}
	\end{equation}  
\end{proposition}


In other words, Proposition \ref{prop:existence_minimiser} characterises a minimiser of \eqref{eq:mp_sigma} as a solution $(u^{h,\sigma},j^{h,\sigma}) \in W^3_{\alpha+1,B}(\Omega)\times L_{\frac{\alpha+1}{\alpha}}(\Omega)$ to the degenerate-elliptic boundary-value problem
\begin{equation*}
	\begin{cases}
		\frac{u-u^\ast}{h} + \partial_x \bigl(m(u^\ast) \Psi\bigl(\partial_x^3 u^{h,\sigma} - G^{\prime\prime}(u^{h,\sigma}) \partial_x u^{h,\sigma}\bigr)\bigr) = 0, & x \in \Omega,  
		\\
		\partial_x u^{h,\sigma} = m(u^\ast) \Psi \bigl(\partial_x^3 u^{h,\sigma} - G^{\prime\prime}(u^{h,\sigma}) \partial_x u^{h,\sigma}\bigr) = 0, & x \in \partial\Omega,
	\end{cases}
\end{equation*}  
where we use that $m(u^\ast)$ is strictly positive.

\begin{proof}
    \noindent \textbf{(i) Existence.}
	The minimisation problem \eqref{eq:mp_sigma} is strictly convex with a linear constraint. Existence of a unique minimiser follows from the direct method of the calculus of variations. Indeed, \(H^1(\Omega)\times L_{\frac{\alpha+1}{\alpha}}(\Omega)\) is a reflexive Banach space and the set of all pairs $(u,j) \in H^1(\Omega)\times L_{\frac{\alpha+1}{\alpha}}(\Omega)$ satisfying the flow equation \eqref{eq:flow_equation} is a closed, non-empty and affine subspace, containing the point $(u^\ast,0)$. Thus, since the functional $\Fcal^{h,\sigma}_{u^\ast}$ in \eqref{eq:functional_mms} is non-negative, there exists a minimising sequence. In view of the Poincar\'e inequality \eqref{eq:poincare} and the strict positivity of $m(u^\ast)$, we may extract a subsequence that converges weakly in $H^1(\Omega)\times L_{\frac{\alpha+1}{\alpha}}(\Omega)$. 
	The strict convexity of $\Fcal^{h,\sigma}_{u^\ast}$ in \((u,j)\in H^1(\Omega)\times L_{\frac{\alpha+1}{\alpha}}(\Omega)\) yields weak lower semicontinuity and therewith the existence of a minimiser $(u^{h,\sigma},j^{h,\sigma})\in H^1(\Omega)\times L_{\frac{\alpha+1}{\alpha}}(\Omega)$ solving the flow equation \eqref{eq:flow_equation}. 
	
	Note that since $E^\sigma[u^{h,\sigma}] \leq E^\sigma[u^\ast] < \infty$, each minimiser $u^{h,\sigma}$ is strictly positive and uniformly bounded.
	
    \noindent \textbf{(ii) Uniqueness.} In order to prove uniqueness, assume that $(u_1,j_1)$ and $(u_2,j_2)$ are two minimisers of \eqref{eq:functional_mms}. By the strict convexity of the functions 
	\begin{equation*}
	    s \longmapsto \left|s\right|^2
	    \quad \text{and} \quad
	    s\longmapsto |s|^{\frac{\alpha+1}{\alpha}},
        \quad s \in \R,
	\end{equation*}
	and the convexity of $G_\sigma$,
	we deduce that
	\begin{equation*}
	    \partial_x u_1 = \partial_x u_2 \quad \text{and} \quad j_1 = j_2 \quad \text{a.e. in } \Omega.
	\end{equation*}
	Since $\bar{u}_1 = \bar{u}_2$, it follows that $(u_1,j_1)=(u_2,j_2)$.
	
	\noindent \textbf{(iii) Euler--Lagrange equation.} In order to derive the Euler--Lagrange equation for the minimiser $(u^{h,\sigma},j^{h,\sigma}) \in H^1(\Omega)\times L_{\frac{\alpha+1}{\alpha}}(\Omega)$, we first consider  solenoidal vector fields $k \in L_{\frac{\alpha+1}{\alpha}}(\Omega)$, i.e. $k$  satisfies $\langle k,\partial_x \phi\rangle = 0$ for all $\phi \in C^1(\bar{\Omega})$. Then, we take the first variation
	\begin{equation*}
	    0
	    =
	    \frac{d}{d\eps}\Fcal^{h,\sigma}_{u^\ast}\bigl[u^{h,\sigma},j^{h,\sigma}+\eps k\bigr]_{|\eps =0} 
	    = 
	    h\int_{\Omega} \frac{|j^{h,\sigma}|^{\frac{1-\alpha}{\alpha}}j^{h,\sigma}k}{m(u^*)^{\frac{1}{\alpha}}} \dd x,
	\end{equation*}
	for all solenoidal vector fields \(k\in L_{\frac{\alpha+1}{\alpha}}(\Omega)\). Using the Helmholtz decomposition \cite{solonnikov_estimates_1977}, this implies the existence of some $\psi\in W^1_{\alpha+1}(\Omega)$ such that
	\begin{equation} \label{eq:j_psi}
	    |j^{h,\sigma}|^{\frac{1-\alpha}{\alpha}}j^{h,\sigma} 
	    = 
	    m(u^\ast)^{\frac{1}{\alpha}} \partial_x \psi.
	\end{equation}
	Now, pick a perturbation $w \in C^1(\bar{\Omega})$ with average $\bar{w} = 0$ for the first component $u^{h,\sigma}$ and choose $k \in L_\frac{\alpha+1}{\alpha}(\Omega)$ such that $k = \partial_x \Phi$, where $\Phi \in W^1_{\alpha+1}(\Omega)$ solves the Neumann problem
	\begin{equation}\label{eq:ch6-k-psi}
	\begin{cases}
	    -\partial_x^2 \Phi = \frac{w}{h}, & x \in\Omega, \\
	    \partial_x \Phi = 0, & x \in \partial\Omega.
	\end{cases}
	\end{equation}
	Since then
		\[\frac{u^{h,\sigma}+\eps w - u^*}{h} + \partial_x\bigl(j^{h,\sigma}+ \eps k\bigr) = 0,\]
	we may take the first variation
	\begin{align*}
	    0 &= 
	    \frac{d}{d\eps}\Fcal^{h,\sigma}_{u^\ast}\bigl[u^{h,\sigma} 
	    + \eps w,j^{h,\sigma}+\eps k\bigr]|_{\eps =0} 
	    \\
	    &=
	    \int_{\Omega} \partial_x u^{h,\sigma} \partial_x w \dd x 
	    +
	    \int_\Omega
	    G_\sigma^\prime(u^{h,\sigma}) w \dd x
	    + h\int_{\Omega} \partial_x\psi k \dd x \\
	    & = 
	    \int_{\Omega} \partial_x u^{h,\delta} \partial_x w \dd x 
	    +
	    \int_\Omega
	    G_\sigma^\prime(u^{h,\sigma}) w \dd x
	    + \int_{\Omega} \psi w \dd x.
	\end{align*} 
	In the last step, we have used \eqref{eq:ch6-k-psi} and the fact that \(\psi\in W^1_{\alpha+1}(\Omega)\) is a valid test function.
	Since this equation holds true for arbitrary $w \in C^1(\bar{\Omega})$ with average $\bar{w} = 0$,  we obtain in particular that
	\begin{equation} \label{eq:Laplace_u}
		G_\sigma^\prime(u^{h,\sigma}) + \psi = \partial_x^2 u^{h,\sigma} + C,
	\end{equation}
	for some constant $C$, in the sense of distributions. This shows that $u^{h,\sigma} \in H^1(\Omega)$ satisfies 
	\begin{equation*}
    	\langle \partial_x u^{h,\sigma}, \partial_x v \rangle
	    =
	    -\langle (\psi + G_\sigma^\prime(u^{h,\sigma}) - C), v \rangle,
	    \quad
	    v \in H^1(\Omega). 
	\end{equation*}
	Using \cite[Theorem 3.3]{agmon_estimates_1959} together with the fact that $G^\prime_\sigma(u^{h,\sigma}) \in W^1_{\alpha+1}(\Omega)$ yields
	\begin{equation*}
	    u^{h,\sigma} \in W^{3}_{\alpha+1,B}(\Omega) \quad \text{and} \quad \partial_x u^{h,\sigma} = 0 \quad \text{a.e. on } \partial\Omega.
	\end{equation*}
	Summarising, we find that $u^{h,\sigma} \in W^{3}_{\alpha+1,B}(\Omega)$ satisfies the boundary-value problem
	\begin{equation*}
		\begin{cases}
			\psi = \partial_x^2 u^{h,\sigma} - G_\sigma^\prime(u^{h,\sigma}) + C, & x \in \Omega, 
			\\
			\partial_x u^{h,\sigma} = 0, & x \in \partial\Omega.
		\end{cases}
	\end{equation*}
	Consequently, the minimiser $(u^{h,\sigma},j^{h,\sigma})$ satisfies
	\begin{equation*}
	    |j^{h,\sigma}|^{\frac{1-\alpha}{\alpha}}j^{h,\sigma} = m(u^\ast)^{\frac{1}{\alpha}} \bigl(\partial_x^3 u^{h,\sigma} - G^{\prime\prime}(u^{h,\sigma}) \partial_x u^{h,\sigma}\bigr).
	\end{equation*}
	Since the function \(s\mapsto |s|^{\frac{1-\alpha}{\alpha}}s=\Psi^{-1}(s)\) is invertible, this implies that the flux $j^{h,\sigma}$ is given by
	\begin{equation*}
	    j^{h,\sigma} 
	    = 
	    m(u^*) \Psi \bigl(\partial_x^3 u^{h,\sigma} - G^{\prime\prime}(u^{h,\sigma}) \partial_x u^{h,\sigma}\bigr)
	\end{equation*}
	and we obtain the Euler--Lagrange equation
	\begin{equation*}
	    \frac{u-u^\ast}{h} 
	    +
	    \partial_x\bigl(m(u^\ast) \Psi \bigl(\partial_x^3 u^{h,\sigma} - G^{\prime\prime}(u^{h,\sigma}) \partial_x u^{h,\sigma}\bigr)\bigr)=0.
	\end{equation*}
	This completes the proof.
\end{proof}

\begin{corollary}\label{cor:Laplace_u}
Under the assumptions of Proposition \ref{prop:existence_minimiser} we obtain the estimate
\begin{equation*}
    \|\partial_x^2 u^{h,\sigma}\|_{L_{\alpha+1}(\Omega)}
    \leq
    C \bigl\|\frac{j^{h,\sigma}}{m(u^\ast)}\bigr\|_{L_\frac{\alpha+1}{\alpha}(\Omega)} + C \|G^{\prime}_\sigma(u^{h,\sigma})\|_{L_\infty(\Omega)}
\end{equation*}for the second derivative of the minimiser of \eqref{eq:mp_sigma}.
\end{corollary}

\begin{proof}
In \eqref{eq:Laplace_u} in the proof of Proposition \ref{prop:existence_minimiser} we have seen that
\begin{equation*}
    \partial_x^2 u^{h,\sigma} 
    =
    \psi + G^\prime_\sigma(u^{h,\sigma}) - C,
    \quad
    \text{where}
    \quad
    C = \fint_\Omega \psi + G^\prime_\sigma(u^{h,\sigma}) \dd x.
\end{equation*}
Moreover, we know that $\psi \in W^1_{\alpha+1}(\Omega)$ with (w.l.o.g.) $\fint_\Omega \psi \dd x = 0$. Using \eqref{eq:j_psi}, we obtain 
\begin{equation*}
    \begin{split}
        \|\partial_x^2 u^{h,\sigma}\|_{L_{\alpha+1}(\Omega)}
        & \leq
        \|\psi\|_{L_{\alpha+1}(\Omega)} 
        + 
        C \|G^{\prime}_\sigma(u^{h,\sigma})\|_{L_\infty(\Omega)}
        \\
        &\leq
        C \|\partial_x \psi\|_{L_{\alpha+1}(\Omega)} 
        + 
        C \|G^{\prime}_\sigma(u^{h,\sigma})\|_{L_\infty(\Omega)}
        \\
        &\leq
        C \left\|\frac{j^{h,\sigma}}{m(u^\ast)}\right\|_{L_\frac{\alpha+1}{\alpha}(\Omega)}
        + 
        C \|G^{\prime}_\sigma(u^{h,\sigma})\|_{L_\infty(\Omega)}
    \end{split}
\end{equation*}
and the proof is complete.
\end{proof}


\subsection{Minimising-movement scheme and energy-dissipation inequality}

The Euler--Lagrange equation \eqref{eq:elliptic_bvp} corresponding to the functional \(\Fcal_{u^\ast}^{h,\sigma}\) is in fact a time-discretised version of the power-law thin-film equation with modified potential. We can exploit this and define a minimising-movement scheme with step size $h > 0$ as follows. Given an initial value $u_0 \in H^1(\Omega)$ such that $E^\sigma[u_0] < \infty$, we recursively define the sequence
\begin{equation*}
    \begin{cases}
        u^{h,\sigma}_0 \coloneq u_0,
        \\
        \bigl(u^{h,\sigma}_{k+1},j^{h,\sigma}_{k+1}\bigr) \text{ is a solution to the minimisation problem \eqref{eq:mp_sigma} with } u^\ast = u^{h,\sigma}_k.
    \end{cases}
\end{equation*}
Note that as long as we work with a fixed positive regularisation parameter $\sigma > 0$, we consider strictly positive initial values $u_0^\sigma = u_0$ and drop the superscript. Later, in the limit $\sigma \searrow 0$, we may approximate general non-negative initial values $u_0 \geq 0$ by sequences of strictly positive $u^\sigma_0 >0$.

\medskip

Comparing $\bigl(u^{h,\sigma}_{k+1},j^{h,\sigma}_{k+1}\bigr)$ with $(u_{k}^{h,\sigma}, 0)$, we obtain the weak energy-dissipation inequality
\begin{align}\label{eq:ch6-weak-energy-dissipation}
    \Fcal_{u_k^{h,\sigma}}^{h,\sigma}[u_k^{h,\sigma},0] 
    = 
    E^\sigma[u_k^{h,\sigma}]
    \geq
    E^\sigma[u_{k+1}^{h,\sigma}]
    +
    h\tfrac{\alpha}{\alpha+1}\int_{\Omega}
    \frac{|j^{h,\sigma}_{k+1}|^{\frac{\alpha+1}{\alpha}}}{ m\bigl(u_{k}^{h,\sigma}\bigr)^{\frac{1}{\alpha}}}\dd x
    = 
    \Fcal_{u_k^{h,\sigma}}^{h,\sigma}\bigl[u^{h,\sigma}_{k+1},j^{h,\sigma}_{k+1}\bigr].
\end{align}
This inequality may even be improved. Indeed, using the elementary identity 
\begin{equation*}
    \left|x\right|^2 - \left|y\right|^2 
    = 2 y\cdot(x-y) + \left|x - y\right|^2, \quad x,y \in \R,
\end{equation*}
leads to
\begin{equation} \label{eq:ei_Dirichlet}
    \begin{split}
        \int_{\Omega}\tfrac{1}{2} \left|\partial_x u_{k}^{h,\sigma}\right|^2\dd x 
        &= 
        \int_{\Omega}\tfrac{1}{2} \left|\partial_x u_{k+1}^{h,\sigma}\right|^2\dd x 
        -
        \int_\Omega \partial_x^2 u_{k+1}^{h,\sigma} \bigl(u_{k}^{h,\sigma} - u_{k+1}^{h,\sigma}\bigr)\dd x 
        +
        \int_{\Omega} \tfrac{1}{2} \left|\partial_x \bigl(u_{k}^{h,\sigma} - u_{k+1}^{h,\sigma}\bigr)\right|^2\dd x
        \\
        &\geq
        \int_{\Omega}\tfrac{1}{2} \left|\partial_x u_{k+1}^{h,\sigma}\right|^2\dd x 
        +
        h\int_{\Omega}\partial_x^3 u_{k+1}^{h,\sigma} j_{k+1}^{h,\sigma}\dd x.
    \end{split}
\end{equation}
By convexity of $G_\sigma$, we have 
\begin{equation*}
    G_\sigma(x) - G_\sigma(y) \geq G^\prime_\sigma(y) (x-y),
    \quad 
    x,y \in \R,
\end{equation*}
and hence
\begin{equation} \label{eq:ei_G_sigma}
    \begin{split}
        \int_\Omega G_\sigma(u^{h,\sigma}_k) \dd x
        & \geq 
        \int_\Omega G_\sigma(u^{h,\sigma}_{k+1}) \dd x
        +
        \int_\Omega G_\sigma^\prime(u^{h,\sigma}_{k+1}) (u^{h,\sigma}_{k} - u^{h,\sigma}_{k+1}) \dd x
        \\
        & \geq 
        \int_\Omega G_\sigma(u^{h,\sigma}_{k+1}) \dd x
        -
        h \int_\Omega G_\sigma^{\prime\prime}(u^{h,\sigma}_{k+1}) \partial_x u^{h,\sigma}_{k+1} j^{h,\sigma}_{k+1} \dd x.
    \end{split}
\end{equation}
Combining \eqref{eq:ei_Dirichlet} and \eqref{eq:ei_G_sigma} and using \eqref{eq:elliptic_bvp}, we obtain the energy-dissipation inequality 
\begin{equation}\label{eq:one-step_energy-diss_inequality}
    \begin{split}
        E^\sigma[u^{h,\sigma}_k] 
        & \geq 
        E^\sigma[u^{h,\sigma}_{k+1}]
        +
        h \int_\Omega \bigl(\partial_x^3 u^{h,\sigma}_{k+1} - G_\sigma^{\prime\prime}(u^{h,\sigma}_{k+1}) \partial_x u^{h,\sigma}_{k+1}\bigr) j^{h,\sigma}_{k+1} \dd x
        \\
        & \geq 
        E^\sigma[u^{h,\sigma}_{k+1}]
        +
        h\int_{\Omega}
        \frac{|j^{h,\sigma}_{k+1}|^{\frac{\alpha+1}{\alpha}}}{ m\bigl(u_{k}^{h,\sigma}\bigr)^{\frac{1}{\alpha}}}\dd x.
    \end{split}
\end{equation}

Now, we define the different interpolations of the minimising-movement scheme which are used in the following.
First, the piecewise constant interpolation $\bigl(\bar{u}^{h,\sigma},\bar{j}^{h,\sigma}\bigr)$ is given by
\begin{equation*}
    \begin{cases}
        \bar{u}^{h,\sigma}(t) = u_{k}^{h,\sigma},  &  t \in [kh,(k+1)h),
        \\
        \bar{j}^{h,\sigma}(t) = j_{k+1}^{h,\sigma}, & t \in (kh,(k+1)h].
    \end{cases}
\end{equation*}
Note that $\bar{u}^{h,\sigma}(t)$ takes the value of $u^{h,\sigma}$ at the current time, while $\bar{j}^{h,\sigma}$ takes the value of $j^{h,\sigma}$ after one time step in the future.
Moreover, the piecewise affine interpolation $\hat{u}^{h,\sigma}$ is given by
\begin{equation*}
    \hat{u}^{h,\sigma}((k+s)h) 
    = 
    (1-s) u_{k}^{h,\sigma} + s u_{k+1}^{h,\sigma},
    \quad k \in \N,\ s \in [0,1).
\end{equation*}
There is no need to define a piecewise affine interpolation for $j^{h,\sigma}$ as becomes clear in the following lemma.


\begin{lemma}[Discrete energy-dissipation inequality] \label{lem:Discrete_energy-dissipation}
Given an initial value $u_0 \in H^1(\Omega)$ with finite energy $E^\sigma[u_0] < \infty$, the following holds true.
	\begin{itemize}	 
		\item[(i)] The pair $\bigl(\bar{u}^{h,\sigma},\bar{j}^{h,\sigma}\bigr)$ satisfies the energy-dissipation inequality
		\begin{equation}\label{eq:energy-dissipation_inequ}
		\begin{split}
		    E^\sigma[\bar{u}^{h,\sigma}](t)
			&+
			\tfrac{\alpha}{\alpha+1}\int_s^t \int_\Omega \frac{|\bar{j}^{h,\sigma}(\tau)|^{\frac{\alpha+1}{\alpha}}}{ m\bigl(\bar{u}^{h,\sigma}(\tau)\bigr)^{\frac{1}{\alpha}}}\dd x \dd \tau   
			\\
			& 
			+ \tfrac{1}{\alpha+1}
			\int_s^t \int_\Omega 
			m\bigl(\bar{u}^{h,\sigma}(\tau)\bigr) \left|\partial_x^3 \bar{u}^{h,\sigma}(\tau+h) - G_\sigma^{\prime\prime}(\bar{u}^{h,\sigma}(\tau+h)) \partial_x \bar{u}^{h,\sigma}(\tau+h)\right|^{\alpha+1}\dd x\dd \tau
			\\
			&\leq 
			E^\sigma[\bar{u}^{h,\sigma}](s) 
		\end{split}
		\end{equation}
		for all $kh = s < t = lh$ and $k,l \in \N_0$. 
		\item[(ii)] The pair $\bigl(\hat{u}^{h,\sigma},\bar{j}^{h,\sigma}\bigr)$ solves the continuity equation
		\begin{equation} \label{eq:ch6-continuity_equation}
		    \begin{cases}
		    \partial_t \hat{u}^{h,\sigma} + \partial_x \bar{j}^{h,\sigma} = 0, 
		    & t > 0,\ x \in \Omega,
		    \\
		    \partial_x \hat{u}^{h,\sigma} 
		    = 0,
		    & t > h,\ x \in \partial\Omega,
		    \\
		    \bar{j}^{h,\sigma}
		    = 0, 
		    & t > 0,\ x \in \partial\Omega,
		    \end{cases}
		\end{equation}
		in sense of distributions, that is, the equation
		\begin{equation*}
		    \int_0^T \int_\Omega \partial_t \hat{u}^{h,\sigma} \phi
		    - \bar{j}^{h,\sigma} \partial_x \phi \dd x\dd t 
		    = 0
		\end{equation*}
		holds true for all $\phi \in C^\infty\bigl([0,T]\times\bar{\Omega}\bigr)$ and all $T > 0$. 
	\end{itemize}	 
\end{lemma}


\begin{proof}
	\textbf{(i)} The energy-dissipation inequality may be derived from \eqref{eq:one-step_energy-diss_inequality}, using that
	\begin{align*}
	    &
	    h\int_{\Omega}
	    \frac{|j^{h,\sigma}_{k+1}|^{\frac{\alpha+1}{\alpha}}}{ m\bigl(u_{k}^{h,\sigma}\bigr)^{\frac{1}{\alpha}}}\dd x
	    =
	    \tfrac{\alpha}{\alpha+1}\int_{kh}^{(k+1)h} \int_\Omega \frac{|\bar{j}^{h,\sigma}(\tau)|^{\frac{\alpha+1}{\alpha}}}{m\bigl(\bar{u}^{h,\sigma}(\tau)\bigr)^{\frac{1}{\alpha}}}\dd x \dd \tau \\
	    & +
	    \tfrac{1}{\alpha+1}\int_{kh}^{(k+1)h} \int_\Omega m\bigl(\bar{u}^{h,\sigma}(\tau)\bigr) \bigl|\partial_x^3 \bar{u}^{h,\sigma}(\tau+h) - G_\sigma^{\prime\prime}(\bar{u}^{h,\sigma}(\tau+h)) \partial_x \bar{u}^{h,\sigma}(\tau+h)\bigr|^{\alpha+1}\dd x\dd \tau.
	\end{align*}	
	Here we used that from \eqref{eq:elliptic_bvp} we know that for $\tau > 0$
	\begin{equation*}
	    \bar{j}^{h,\sigma}(\tau)
		= 
		m\bigl(\bar{u}^{h,\sigma}(\tau)\bigr)\Psi
		\bigl(\partial_x^3 \bar{u}^{h,\sigma}(\tau+h) - G_\sigma^{\prime\prime}(\bar{u}^{h,\sigma}(\tau+h)) \partial_x \bar{u}^{h,\sigma}(\tau+h)\bigr).
	\end{equation*}
	
	\noindent\textbf{(ii)} This follows immediately from 
	\begin{equation*}
	    \partial_t \hat{u}^{h,\sigma}(t) 
	    =  
	    u_{k+1}^{h,\sigma} - u_{k}^{h,\sigma}, 
	   \quad
	    t \in (kh,(k+1)h),\, x \in \Omega,
	\end{equation*}
	and the constraints
	\begin{equation*}
	\begin{cases}
	    \frac{u_{k+1}^{h,\sigma} - u_{k}^{h,\sigma}}{h} + \partial_x j_{k+1}^{h,\sigma} 
	    = 0, & k \geq 0,\, x \in \Omega,
	    \\
	    \partial_x u_{k}^{h,\sigma} = 0, & k \geq 1,\, x \in \partial\Omega, 
	    \\
	    j_{k+1}^{h,\sigma} = 0, & k \geq 0,\, x \in \partial\Omega,
	\end{cases}
	\end{equation*}
	cf. the Euler--Lagrange equation \eqref{eq:elliptic_bvp}. 
\end{proof}


In general dimension \(d > 1\) the previous result holds true on bounded Lipschitz domains with the restriction that we only obtain non-negative instead of positive solutions. To this end, we need to replace \(\partial_x^3\) by \(\nabla \Delta\) and \(\partial_x\) by \(\nabla\) and interpret the boundary conditions to hold true in normal direction and in the sense of traces. 
However, in order to evaluate the limit \(h\to 0\) for \(\bar{j}^{h,\sigma}\), we have to guarantee that \(m(\bar{u}^{h,\sigma})\) converges strongly. This requires that the sequence \(\bar{u}^{h,\sigma}\) converges uniformly to a continuous function which is only valid in dimension \(d=1\).

We first provide the necessary uniform (in $h$) a-priori estimates. 

\begin{lemma}[Uniform (in $h$) a-priori estimates] \label{lem:uniform_estimates}
Given an initial value $u_0 \in H^1(\Omega)$ with finite energy $E^\sigma[u_0] < \infty$, the following holds true. For each $\sigma > 0$ and each $h > 0$, the families $\bigl(\bar{u}^{h,\sigma}, \bar{j}^{h,\sigma}\bigr)_h$ and $\bigl(\hat{u}^{h,\sigma}\bigr)_h$ satisfy the regularity properties  
	\begin{itemize}
		\item[(i)]$\bar{u}^{h,\sigma} \in L_{\infty}\bigl([0,\infty);H^1(\Omega)\bigr)$;
		\item[(ii)] $\bar{j}^{h,\sigma} \in L_{\frac{\alpha+1}{\alpha}}\bigl([0,\infty)\times\Omega\bigr)$;
		\item[(iii)] $\partial_x^3 u^{h,\sigma}-G^{\prime\prime}_{\sigma}(u^{h,\sigma}) \partial_xu^{h,\sigma} \in L_{\alpha+1}\bigl([h,\infty)\times\Omega\bigr)$;
		\item[(iv)] $\bar{u}^{h,\sigma} \in L_{\alpha+1,\loc}\bigl([h,\infty);W^{3}_{\alpha+1,B}(\Omega)\bigr)$;
		\item[(v)] $\hat{u}^{h,\sigma} \in L_{\alpha+1,\loc}\bigl([h,\infty);W^3_{\alpha+1,B}(\Omega)\bigr) \cap L_{\infty}\bigl([0,\infty);H^1(\Omega)\bigr)$;
		\item[(vi)] $\partial_t\hat{u}^{h,\sigma} \in L_{\frac{\alpha+1}{\alpha}}\bigl([0,\infty);\bigl(W^1_{\alpha+1,B}(\Omega)\bigr)'\bigr)$,
	\end{itemize}
	with bounds that are uniform in $h$. That is, there exists a constant $C > 0$, independent of $h$ such that the families $\bigl(\bar{u}^{h,\sigma}, \bar{j}^{h,\sigma}\bigr)_h$ and $\bigl(\hat{u}^{h,\sigma}\bigr)_h$ are, for each $h > 0$, bounded by $C$ in the respective norms. 
\end{lemma}


The proof strongly relies on the energy-dissipation inequality \eqref{eq:energy-dissipation_inequ}.


\begin{proof}
\textbf{(i)} 
First, the energy-dissipation inequality \eqref{eq:energy-dissipation_inequ} immediately implies that
\begin{equation} \label{eq:H^1_estimate}
	\int_\Omega \tfrac{1}{2} |\partial_x \bar{u}^{h,\sigma}(t)|^2 dx 
	\leq 
	E^\sigma[\bar{u}^{h,\sigma}](t)
	\leq 
	E^\sigma[u_0],
	\quad t \geq 0,
\end{equation}
where the right-hand side is bounded due to the finite-energy assumption on the initial value. That is,
\begin{equation*}
    \partial_x \bar{u}^{h,\sigma} \in L_\infty\bigl([0,\infty);L_2(\Omega)\bigr).
\end{equation*}
By the Poincaré	inequality, cf. Remark \ref{rem:cons_mass_poincare}, this implies
\begin{equation*}
    \bar{u}^{h,\sigma} \in L_\infty\bigl([0,\infty);H^1(\Omega)\bigr).
\end{equation*}

\noindent \textbf{(ii)} By the Sobolev embedding, Step (i) implies in particular that $\bar{u}^{h,\sigma} \in L_\infty\bigl([0,\infty)\times \Omega\bigr)$. This provides a uniform upper bound for the mobility $m(\bar{u}^{h,\sigma})$ and (ii) follows from the energy-dissipation inequality \eqref{eq:energy-dissipation_inequ}.

\noindent \textbf{(iii)} We know from \eqref{eq:H^1_estimate} together with Lemma \ref{lem:positivity_sigma} that $\bar{u}^{h,\sigma} > c_{\sigma,u_0} > 0$ and hence $m(\bar{u}^{h,\sigma}) \geq c_{m,\sigma,u_0} > 0$. Consequently, there exists a constant $C_{\sigma,\alpha} > 0$ such that
\begin{equation*}
	\int_h^{\infty} \int_\Omega |\partial_x^3 \bar{u}^{h,\sigma}(t) - G^{\prime\prime}_\sigma(\bar{u}^{h,\sigma}(t)) \partial_x \bar{u}^{h,\sigma}(t)|^{\alpha+1}\dd t\dd x 
	\leq 
	C_{\sigma,\alpha} E^\sigma[u_0].
\end{equation*}

\noindent\textbf{(iv)} Fix \(T>h\). Using the inverse triangle inequality, step (ii), the uniform boundedness of $\bar{u}^{h,\sigma}$ of step (i) and the boundary condition $\partial_x \bar{u}^{h,\sigma}(t)=0,\, t > h$, on $\partial\Omega$ together with the Poincar\'{e} inequality, we obtain
\begin{equation*}
    \begin{split}
       \int_h^T \int_\Omega |\partial_x^3 \bar{u}^{h,\sigma}(t)|^{\alpha+1}\dd t\dd x 
	    &\leq 
	    C_{\sigma,\alpha} E^\sigma[u_0]
	    +
	    \int_h^T \int_\Omega |G^{\prime\prime}_\sigma(\bar{u}^{h,\sigma}(t)) \partial_x \bar{u}^{h,\sigma}(t)|^{\alpha+1}\dd t\dd x
	    \\
	    &\leq
	    C_{\sigma,\alpha} E^\sigma[u_0]
	    +
	    C_{u_0}
	    \int_h^T \int_\Omega |\partial_x \bar{u}^{h,\sigma}(t)|^{\alpha+1}\dd t\dd x
	    \\
	    &\leq
	    C_{\sigma,\alpha} E^\sigma[u_0]
	    +
	    C_{u_0}
	    \int_h^T \int_\Omega |\partial_x^2 \bar{u}^{h,\sigma}(t)|^{\alpha+1}\dd t\dd x.
    \end{split}
\end{equation*}
It remains to bound the last integral uniformly in $h$. To this end, we use Corollary \ref{cor:Laplace_u} and step (ii) to find that
\begin{equation*}
    \begin{split}
        \int_h^T \int_\Omega \left|\partial_x^2 \bar{u}^{h,\sigma}(t)\right|^{\alpha+1}\dd t\dd x
        &\leq
        C \int_h^T \|G^{\prime}_\sigma(\bar{u}^{h,\sigma}(t))\|_{L_\infty(\Omega)}^{\alpha+1} \dd t
        +
        C \int_h^T \int_\Omega \frac{|\bar{j}^{h,\sigma}(t)|^\frac{\alpha+1}{\alpha}}{m(\bar{u}^{h,\sigma}(t))^\frac{\alpha+1}{\alpha}} \dd x \dd t
        \\
        &\leq
        C_{u_0} T + C \int_h^T \int_\Omega \frac{|\bar{j}^{h,\sigma}(t)|^\frac{\alpha+1}{\alpha}}{m(\bar{u}^{h,\sigma}(t))^\frac{\alpha+1}{\alpha}} \dd x \dd t.
    \end{split}
\end{equation*}
Now we conclude as in step (ii). This proves that $\partial_x^3 \bar{u}^{h,\sigma} \in L_{\alpha+1}\bigl([h,T];L_{\alpha+1}(\Omega)\bigr)$, and therefore
\begin{equation*}
	\bar{u}^{h,\sigma} \in L_{\alpha+1}\bigl([h,T];W^3_{\alpha+1,B}(\Omega)\bigr)
\end{equation*}
for every \(T>h\) with a uniform bound only dependent on the initial energy $E^\sigma[u_0]$ and $T$.

\noindent\textbf{(v)} By definition of the piecewise affine interpolation $\hat{u}^{h,\sigma}$ we have that
\begin{equation*}
	\|\hat{u}^{h,\sigma}(t)\|_{W^3_{\alpha+1,B}(\Omega)}^{\alpha+1} 
	\leq C 
	\left(\|\bar{u}^{h,\sigma}(t)\|_{W^3_{\alpha+1,B}(\Omega)}^{\alpha+1} 
	+ 
	\|\bar{u}^{h,\sigma}(t+h)\|_{W^3_{\alpha+1,B}(\Omega)}^{\alpha+1}\right), \quad t > 0.
\end{equation*}
Integration with respect to $t$ and the a-priori bound obtained in (iv) yield the desired bound. The same argument applies to the \(H^1\)-bound.
	
\noindent\textbf{(vi)} The regularity of the flux $\bar{j}^{h,\sigma}$ obtained in (iii) implies in particular that $\partial_x \bar{j}^{h,\sigma} \in L_{\frac{\alpha+1}{\alpha}}\bigl([0,\infty);\bigl(W^1_{\alpha+1,B}(\Omega)\bigr)'\bigr)$. Therefore, the continuity equation $\partial_t \hat{u}^{h,\sigma} = - \partial_x \bar{j}^{h,\sigma}$, obtained in Lemma \ref{lem:Discrete_energy-dissipation} (ii), yields \begin{equation*}
	\partial_t \hat{u}^{h,\sigma} \in L_{\frac{\alpha+1}{\alpha}}\bigl([0,\infty);\bigl(W^1_{\alpha+1,B}(\Omega)\bigr)'\bigr),
\end{equation*}
where the corresponding uniform bound follows from (ii).
\end{proof}


Using the previous lemma, an application of the Aubin--Lions--Simon lemma yields  the following convergence result.

\begin{proposition}[The limit $h \searrow 0$] \label{prop:convergences}
    Fix $\sigma > 0$ and let $u_0 \in H^1(\Omega)$ with finite energy $E^\sigma[u_0] < \infty$.
	There exist subsequences  $\bigl(\bar{u}^{h,\sigma}, \bar{j}^{h,\sigma}\bigr)_h$ and $\bigl(\hat{u}^{h,\sigma}\bigr)_h$ (not relabelled) and limit functions 
\begin{equation*}
    \bigl(u^\sigma,j^\sigma\bigr) 
    \in 
	L_{\alpha+1,\loc}\bigl((0,\infty);W^3_{\alpha+1,B}(\Omega)\bigr) 
	\times 
	L_\frac{\alpha+1}{\alpha}\bigl((0,\infty)\times\Omega\bigr)
	\quad
	\text{with}
	\quad
	\partial_t u^\sigma 
	\in L_{\frac{\alpha+1}{\alpha}}\bigl([0,T];\bigl(W^1_{\alpha+1,B}(\Omega)\bigr)'\bigr)
\end{equation*}
	with \(u^{\sigma}\in C_b\bigl([0,\infty);H^1(\Omega)\bigr)\) such that
	\begin{equation*}
	\begin{cases}
	    \bar{u}^{h,\sigma} \xrightharpoonup{\phantom{\text{wii}}} u^\sigma & \text{weakly in } L_{\alpha+1,\loc}\bigl((\eps,\infty);W^3_{\alpha+1,B}(\Omega)\bigr); 
	    \\
	    \bar{u}^{h,\sigma} \xrightharpoonup{\phantom{\text{i}}*\phantom{\text{i}}} u^\sigma & \text{weakly-star in } L_{\infty}\bigl((0,\infty);H^1(\Omega)\bigr); 
	    \\
	    \bar{j}^{h,\sigma}\, \xrightharpoonup{\phantom{\text{wii}}} j^\sigma & \text{weakly in } L_{\frac{\alpha+1}{\alpha}}\bigl((0,\infty)\times\Omega\bigr);
	    \\
	    \hat{u}^{h,\sigma} \xrightharpoonup{\phantom{\text{wii}}} u^\sigma & \text{weakly in } L_{\alpha+1,\loc}\bigl((\eps,\infty);W^3_{\alpha+1,B}(\Omega)\bigr);  
	    \\
	    \partial_t\hat{u}^{h,\sigma} \xrightharpoonup{\phantom{\text{wii}}} \partial_t u^\sigma & \text{weakly in }  L_{\frac{\alpha+1}{\alpha}}\bigl((0,\infty);\bigl(W^1_{\alpha+1,B}(\Omega)\bigr)'\bigr);  
	    \\
	    \hat{u}^{h,\sigma} \longrightarrow u^\sigma & \text{strongly in } C_{loc}\bigl([0,\infty);C^{\nu}(\bar{\Omega})\bigr)
	\end{cases}
	\end{equation*} 
	for all $\eps >0$ and all \(0<\nu <\frac{1}{2}\). Furthermore, it holds that \(u^{\sigma}\in C([0,T];H^1(\Omega))\) and there is a constant $c_{\sigma,u_0} > 0$ such that
	\begin{equation*}
	    u^\sigma \geq c_{\sigma,u_0} > 0,
	    \quad
	    t \geq 0,\, x \in \bar{\Omega}.
	\end{equation*}
\end{proposition}

Note that the constant $c_{\sigma,u_0}$ in general tends to zero as $\sigma \searrow 0$, as long as the potential $G$ is \lq not very singular\rq.

\begin{proof}
\textbf{(i) Compactness. }
Combining the uniform bounds proved in step 1, an application of the Eberlein--\v{S}mulian theorem provides the existence of the subsequences and weak accumulation points
	\begin{equation*}
		\begin{cases}
			\bar{u}^{h,\sigma} \xrightharpoonup{\phantom{\text{wii}}} u^\sigma & \text{weakly in } L_{\alpha+1,\loc}\bigl((\eps,\infty);W^3_{\alpha+1,B}(\Omega)\bigr); 
	    \\
	    \bar{u}^{h,\sigma} \xrightharpoonup{\phantom{\text{i}}*\phantom{\text{i}}} u^\sigma & \text{weakly-star in } L_{\infty}\bigl((0,\infty);H^1(\Omega)\bigr); 
	    \\
	    \bar{j}^{h,\sigma}\, \xrightharpoonup{\phantom{\text{wii}}} j^\sigma & \text{weakly in } L_{\frac{\alpha+1}{\alpha}}\bigl((0,\infty)\times\Omega\bigr);
	    \\
	    \hat{u}^{h,\sigma} \xrightharpoonup{\phantom{\text{wii}}} u^\sigma & \text{weakly in } L_{\alpha+1,\loc}\bigl((\eps,\infty);W^3_{\alpha+1,B}(\Omega)\bigr);  
	    \\
	    \partial_t\hat{u}^{h,\sigma} \xrightharpoonup{\phantom{\text{wii}}} \partial_t u^\sigma & \text{weakly in }  L_{\frac{\alpha+1}{\alpha}}\bigl((0,\infty);\bigl(W^1_{\alpha+1,B}(\Omega)\bigr)'\bigr).
		\end{cases}
	\end{equation*}
By the Aubin--Lions--Simon lemma \cite{simon_compact_1986}, we find that \(\hat{u}^{h,\sigma}\) converges strongly (up to a subsequence) in \(C([\eps,T];C^{\nu}(\bar{\Omega}))\) for every \(0<\eps < T<\infty\) and \(0<\nu< \frac{1}{2}\). Moreover, by \cite[Remark 3.4]{bernis_existence_1988}, every function \(u^{\sigma}\in L_{\alpha+1,\loc}\bigl((0,\infty);W^3_{\alpha+1,B}(\Omega)\bigr)\) with \(\partial_t \hat{u}^{\sigma}\in L_{\frac{\alpha+1}{\alpha}}\bigl((0,\infty);\bigl(W^1_{\alpha+1,B}(\Omega)\bigr)'\bigr)\) satisfies \(u^{\sigma}\in C([0,\infty);H^1(\Omega))\) and we also conclude \(u^{\sigma}(0)=u_0\).

\noindent\textbf{(ii) Positivity of the limit function. } 
We know that $u^{h,\sigma} \geq c_{\sigma,u_0} > 0$. By the uniform convergence obtained in (i), we conclude that also the limit function satisfies $u^{\sigma} \geq c_{\sigma,u_0} > 0$.

\noindent\textbf{(iii) Uniqueness of the limit function. } 
We have claimed above that both \(\hat{u}^{h,\sigma}\) and \(\bar{u}^{h,\sigma}\) converge to the same limit \(u^{\sigma}\). Indeed, we observe that
\begin{equation*}
	\begin{split}
	\hat{u}^{h,\sigma}((k+s)h) -\bar{u}^{h,\delta}((k+s)h)
	&= 
	(1-s) u_{k}^{h,\sigma} + s u_{k+1}^{h,\sigma} - u_{k}^{h,\sigma}
	\\
	&=
	s\bigl(u_{k+1}^{h,\sigma} - u_{k}^{h,\sigma}\bigr)
	\\
	&=
	sh\, \partial_x j^{h,\sigma}_{k+1}
	\\
	&=
	s\, h\, \partial_x \bar{j}^{h,\sigma}((k+s)h),
	\quad k \in \N_0,\ s \in [0,1).
	\end{split}
\end{equation*}
This implies
\begin{align*}
	\|\hat{u}^{h,\sigma} -\bar{u}^{h,\sigma}\|_{L_{\frac{\alpha+1}{\alpha}}((0,T];(W^1_{\alpha+1,B}(\Omega)'))}
	& \leq 
	h\, \|\partial_x \bar{j}^{h,\sigma}\|_{L_{\frac{\alpha+1}{\alpha}}((0,T];(W^1_{\alpha+1,B}(\Omega)'))}
	\leq 
	h\, \|\bar{j}^{h,\sigma}\|_{L_{\frac{\alpha+1}{\alpha}}((0,T)\times\Omega)}
	\leq C\, h.
\end{align*}
This proves that the limit functions coincide.
\end{proof}


\begin{proposition}[Uniform convergence of the energy]\label{prop:uniform-convergence-energy}
Fix \(\sigma >0\) and $u_0 \in H^1(\Omega)$ with $E^\sigma[u_0] < \infty$. There is a subsequence (not relabelled) of \((\bar{u}^{h,\sigma})_h\) such that
	\[\bar{u}^{h,\sigma} \longrightarrow u^{\sigma} \quad \text{strongly in } L_{\alpha+1,\loc}\bigl([0,\infty);H^1(\Omega)\bigr)\]
	and
	\[E^\sigma[\bar{u}^{h,\sigma}] \longrightarrow E^\sigma[u^{\sigma}] \quad \text{ uniformly on compact subsets of }[0,\infty).\]
\end{proposition}

The proof of the second part of Proposition \ref{prop:uniform-convergence-energy} relies on the following result from basic calculus which we prove here for the convenience of the reader.

\begin{lemma}\label{lem:monotone-functions}
	Let \(0<T<\infty\) and \(f_k\colon[0,T] \to \R\) be a sequence of non-increasing real-valued functions. Assume that \(f_k(t) \to f(t)\) pointwise for \(t\in [0,T]\), where \(f\colon [0,T]\to \R\) is a continuous function. Then \(f_k \to f\) uniformly.
\end{lemma}

\begin{proof}
	Fix \(\eps >0\). Since \(f\) is continuous on the compact interval \([0,T]\), \(f\) is uniformly continuous. So there is \(\delta >0\) and a subdivision \(0=t_0 < t_1 < \ldots < t_n = T\) with \(t_{i+1}-t_i < \delta\) for every \(i=0,\ldots,n-1\), such that
	\begin{equation*}
	|f(t)-f(s)| < \frac{\eps}{2} \quad \text{for all }t,s \in [t_i,t_{i+1}] \text{ and all } i =0,\ldots,n-1.
	\end{equation*}
	Since \(f_n(t_i) \to f(t_i)\), we find \(N\in \N\) such that
	\[|f_n(t_i) - f(t_i)| \leq \frac{\eps}{2} \quad \text{for all } n \geq N \text{ and } i=0,\ldots,n.\]
	We claim that \(|f_n(t)-f(t)| \leq \eps\) for all \(n \geq N\) and \(t\in [0,T]\). Fix \(t\in [0,T]\), then there is \(i\in \{1,\ldots,n\}\) such that \(t\in [t_i,t_{i+1}]\). Since \(f_n\) is non-increasing, we know that
	\[f_n(t_{i+1})\leq f_n(t) \leq f_n(t_i)\]
	and hence
	\begin{align*}
	|f_n(t)-f(t)| & \leq \max\{|f_n(t_{i+1})-f(t)|,|f_n(t_i)-f(t)|\} \\
	& \leq \max\{|f_n(t_{i+1})-f(t_{i+1})|,|f_n(t_i)-f(t_i)|\}  + \max\{|f(t_{i+1})-f(t)|,|f(t_i)-f(t)|\} \\
	& \leq \eps
	\end{align*}
	for every \(n\geq N\). This proves the lemma.
\end{proof}

Now we turn to the proof of the uniform convergence of the energy stated in Proposition \ref{prop:uniform-convergence-energy}. The compactness result follows from a modification of the Aubin--Lions--Simon lemma to piecewise constant functions which can be found in \cite{dreher_compact_2012}.

\begin{proof}[\textbf{Proof of Proposition \ref{prop:uniform-convergence-energy}}]

\textbf{(i) Strong convergence of $(\bar{u}^{h,\sigma})_h$. } 
Recall that \(\bar{u}^{h,\sigma}\) is a sequence of piecewise constant functions. In order to obtain compactness in \(L_{\alpha+1,\loc}\bigl([0,\infty];H^1(\Omega)\bigr)\), we apply \cite[Theorem 1]{dreher_compact_2012}. Note that \(W^3_{\alpha+1,B}(\Omega)\) embeds compactly in \(H^1(\Omega)\) and that the embedding of \(H^1(\Omega)\) into \(\bigl(W^1_{\alpha+1,B}(\Omega)\bigr)'\) is continuous. Furthermore, by Lemma \ref{lem:uniform_estimates}, we know that \((\bar{u}^{h,\sigma})_{h}\) is uniformly bounded in \(L_{\alpha+1}\bigl([\eps,T];W^3_{\alpha+1,B}(\Omega)\bigr)\) and that
\begin{equation*}
	\frac{1}{h}\|\bar{u}^{h,\sigma}(\cdot +h) - \bar{u}^{h,\sigma}(\cdot)\|_{L_{\frac{\alpha+1}{\alpha}}([0,T];(W^1_{\alpha+1}(\Omega))')} 
	= 
	\|\partial_x \bar{j}^{h,\sigma}\|_{L_{\frac{\alpha+1}{\alpha}}([0,T];(W^1_{\alpha+1}(\Omega))')}
	\leq 
	C \|\bar{j}^{h,\sigma}\|_{L_{\frac{\alpha+1}{\alpha}}([0,T]\times\Omega)} 
	\leq 
	C
\end{equation*}
is also uniformly bounded in \(h\). Hence, there is a subsequence (not relabelled) which converges strongly in \(L_{\alpha+1}\bigl([0,T];H^1(\Omega))\). Taking \(T=T_n =n\) and a diagonal subsequence ensures that 
\begin{equation*}
    \bar{u}^{h,\sigma} \longrightarrow u^{\sigma}
	\quad
	\text{strongly in } 
	L_{\alpha+1,\loc}\bigl([0,\infty);H^1(\Omega)\bigr).
\end{equation*}

\noindent \textbf{(ii) Uniform convergence of the energy. }
We apply Lemma \ref{lem:monotone-functions}. By the strong convergence proved in part (i), we know that there is a subsequence (not relabelled) such that
 \[\bar{u}^{h,\sigma}(t) \longrightarrow u^{\sigma}(t) \quad \text{in } H^1(\Omega) \text{ for almost every } t\in [0,\infty).\]
Since the limit function \(u^{\sigma}\in C\bigl([0,\infty);H^1(\Omega)\bigr)\) is defined for every \(t\in [0,\infty)\), we may assume (after potentially modifying \(\bar{u}^{h,\sigma}\) on a set of measure zero) that \(\bar{u}^{h,\sigma}(t)\) converges to \(u^{\sigma}(t)\) for every \(t\in [0,\infty)\). This guarantees that
\begin{equation*}
    E^\sigma[\bar{u}^{h,\sigma}](t) \longrightarrow E^\sigma[u^\sigma](t)
\end{equation*}
for all $t\in [0,\infty)$.
By continuity of the limit function, we also obtain that \(t\mapsto E^\sigma[u^{\sigma}](t)\) is continuous. Finally, monotonicity of \(t\mapsto E^\sigma[\bar{u}^{h,\sigma}](t)\) in $t$ follows from the energy-dissipation inequality \eqref{eq:energy-dissipation_inequ} for \(\bar{u}^{h,\sigma}\).
Hence, we may apply Lemma \ref{lem:monotone-functions} to obtain uniform convergence on compact subsets.
\end{proof}

As the following proposition shows, we are able to preserve the energy-dissipation inequality \eqref{eq:energy-dissipation_inequ} in the limit $h \to 0$ for every \(0\leq s,t <\infty\). The proof is based on Lemma \ref{lem:uniform_estimates}, Proposition \ref{prop:uniform-convergence-energy}, the uniform convergence $\bar{u}^{h,\sigma} \to u^\sigma$ together with the local Lipschitz continuity of $m$, and lower semicontinuity of the dissipation. 


\begin{proposition}[Energy-dissipation inequality] \label{prop:energy_inequ}
Fix $\sigma > 0$ and let $u_0 \in H^1(\Omega)$ with $E^\sigma[u_0] < \infty$.
Any weak limit point 
\begin{equation*}
    \bigl(u^\sigma,j^\sigma\bigr) 
    \in 
	L_{\alpha+1,\loc}\bigl((0,\infty);W^3_{\alpha+1,B}(\Omega)\bigr) 
	\times 
	L_\frac{\alpha+1}{\alpha}\bigl((0,\infty)\times\Omega\bigr)
	\quad
	\text{with}
	\quad
	\partial_t u^\sigma 
	\in L_{\frac{\alpha+1}{\alpha}}\bigl((0,\infty);\bigl(W^1_{\alpha+1,B}(\Omega)\bigr)'\bigr)
\end{equation*}
with \(u^{\sigma}\in C_b\bigl([0,\infty);H^1(\Omega)\bigr)\) of the family $(\bar{u}^{h,\sigma},\bar{j}^{h,\sigma})_h$ has the following properties.
	\begin{itemize}
		\item[(i)] For all $0 \leq s < t <\infty$, the energy-dissipation inequality
		\begin{equation}\label{eq:energy_inequ}
		\begin{split}
		E^\sigma[u^\sigma](t)
		& + 
		\tfrac{\alpha}{\alpha+1}\int_{s}^{t} \int_\Omega \frac{\left|j^{\sigma}(\tau)\right|^{\frac{\alpha+1}{\alpha}}}{ m\bigl(u^{\sigma}(\tau)\bigr)^{\frac{1}{\alpha}}}\dd x \dd \tau\\
        & +
        \tfrac{1}{\alpha+1}
		\int_{s}^{t} \int_\Omega 
		m\bigl(u^{\sigma}(\tau)\bigr) \left|\partial_x^3 u^{\sigma}(\tau) - G_\sigma^{\prime\prime}(u^{\sigma}(\tau)) \partial_x u^{\sigma}(\tau)\right|^{\alpha+1}\dd x\dd \tau \\
		& \leq  
		E^\sigma[u^\sigma](s)
		\end{split}
		\end{equation}
		is satisfied.
		\item[(ii)] The pair $\bigl(u^{\sigma},j^{\sigma}\bigr)$ solves the continuity equation
		\begin{equation} \label{eq:limit_continuity_equation}
		\begin{cases}
		\partial_t u^{\sigma} + \partial_x j^{\sigma} = 0, & t > 0,\ x \in \Omega,
		\\
		\partial_x u^{\sigma} = j^{\sigma} =  0, & t > 0,\ x \in \partial\Omega,
		\end{cases}
		\end{equation}
		in the sense of distributions, that is, the equation
		\begin{equation*}
		\int_0^{\infty} \langle \partial_t u^{\sigma}, \phi\rangle_{W^1_{\alpha+1}}\dd t
		- \int_0^{\infty} \int_\Omega j^{\sigma} \partial_x \phi\dd x\dd t = 0
		\end{equation*}
		holds true for all $\phi \in L_{\alpha+1}\bigl((0,\infty);W^1_{\alpha+1}(\Omega)\bigr)$ and all $T > 0$.
	\end{itemize} 
\end{proposition}


\begin{proof}
	\textbf{(i) Energy-dissipation inequality. } From the uniform convergence of \(E^\sigma[\bar{u}^{h,\sigma}]\) to \(E^\sigma[u^{\sigma}]\) proved in Proposition \ref{prop:uniform-convergence-energy} and continuity of \(t\mapsto E^\sigma[u^{\sigma}](t)\), we obtain
	\begin{equation} \label{eq:variational_convergence}
	    E^\sigma[\bar{u}^{h,\sigma}](t_h) \longrightarrow E^\sigma[u^{\sigma}](t) \quad \text{ as } h\to 0 \text{ and } t_h \to t
	\end{equation}
	for any $0 \leq t \leq T$. 
	Let now $s_h = \lfloor s/h\rfloor h$ and $t_h = \lceil t/h\rceil h$. Then, in view of \eqref{eq:energy-dissipation_inequ} we know that
	\begin{equation*}
		\begin{split}
		    E^\sigma[\bar{u}^{h,\sigma}](t_h)
			&+
			\tfrac{\alpha}{\alpha+1}\int_{s_h}^{t_h} \int_\Omega \frac{|\bar{j}^{h,\sigma}(\tau)|^{\frac{\alpha+1}{\alpha}}}{ m\bigl(\bar{u}^{h,\sigma}(\tau)\bigr)^{\frac{1}{\alpha}}}\dd x \dd \tau   
			\\
			& 
			+ \tfrac{1}{\alpha+1}
			\int_{s_h}^{t_h} \int_\Omega 
			m\bigl(\bar{u}^{h,\sigma}(\tau)\bigr) |\partial_x^3 \bar{u}^{h,\sigma}(\tau+h) - G_\sigma^{\prime\prime}(\bar{u}^{h,\sigma}(\tau+h)) \partial_x \bar{u}^{h,\sigma}(\tau+h)|^{\alpha+1}\dd x\dd \tau
			\\
			&\leq 
			E^\sigma[\bar{u}^{h,\sigma}](s_h).
		\end{split}
	\end{equation*}
	
	Taking the \(\liminf\) on both sides and using \eqref{eq:variational_convergence} which guarantees convergence of the energy, we obtain
	\begin{equation*}
		\begin{split}
		    E^\sigma[u^{\sigma}](t)
			&+ \liminf\limits_{h\searrow 0}\left[
			\tfrac{\alpha}{\alpha+1}\int_{s_h}^{t_h} \int_\Omega \frac{|\bar{j}^{h,\sigma}(\tau)|^{\frac{\alpha+1}{\alpha}}}{ m\bigl(\bar{u}^{h,\sigma}(\tau)\bigr)^{\frac{1}{\alpha}}}\dd x \dd \tau  \right.
			\\
			& 
			\left.+ \tfrac{1}{\alpha+1}
			\int_{s_h}^{t_h} \int_\Omega 
			m\bigl(\bar{u}^{h,\sigma}(\tau)\bigr) |\partial_x^3 \bar{u}^{h,\sigma}(\tau+h) - G_\sigma^{\prime\prime}(\bar{u}^{h,\sigma}(\tau+h)) \partial_x \bar{u}^{h,\sigma}(\tau+h)|^{\alpha+1}\dd x\dd \tau\right]
			\\
			&\leq 
			E^\sigma[u^{\sigma}](s).
		\end{split}
	\end{equation*}
	It remains to show
	\begin{equation*}
	    \begin{split}
	        &\tfrac{\alpha}{\alpha+1}\int_{s}^{t} \int_\Omega \frac{|j^{\sigma}(\tau)|^{\frac{\alpha+1}{\alpha}}}{ m\bigl(u^{\sigma}(\tau)\bigr)^{\frac{1}{\alpha}}}\dd x \dd \tau
	        +
	        \tfrac{1}{\alpha+1}
			\int_{s}^{t} \int_\Omega 
			m\bigl(u^{\sigma}(\tau)\bigr) |\partial_x^3 u^{\sigma}(\tau) - G_\sigma^{\prime\prime}(u^{\sigma}(\tau)) \partial_x u^{\sigma}(\tau)|^{\alpha+1}\dd x\dd \tau \\
			\leq
			&\liminf\limits_{h\searrow 0}\left[
			\tfrac{\alpha}{\alpha+1}\int_{s_h}^{t_h} \int_\Omega \frac{|\bar{j}^{h,\sigma}(\tau)|^{\frac{\alpha+1}{\alpha}}}{ m\bigl(\bar{u}^{h,\sigma}(\tau)\bigr)^{\frac{1}{\alpha}}}\dd x \dd \tau  \right.
			\\
			& 
			\left.+ \tfrac{1}{\alpha+1}
			\int_{s_h}^{t_h} \int_\Omega 
			m\bigl(\bar{u}^{h,\sigma}(\tau)\bigr) |\partial_x^3 \bar{u}^{h,\sigma}(\tau+h) - G_\sigma^{\prime\prime}(\bar{u}^{h,\sigma}(\tau+h)) \partial_x \bar{u}^{h,\sigma}(\tau+h)|^{\alpha+1}\dd x\dd \tau\right].
	    \end{split}
	\end{equation*}
	Since \(s_h \leq s<t \leq t_h\) for every \(h>0\) and by non-negativity of the integrand, it suffices to prove
	\begin{equation}\label{eq:liminf}
	    \begin{split}
	        &\tfrac{\alpha}{\alpha+1}\int_{s}^{t} \int_\Omega \frac{|j^{\sigma}(\tau)|^{\frac{\alpha+1}{\alpha}}}{ m\bigl(u^{\sigma}(\tau)\bigr)^{\frac{1}{\alpha}}}\dd x \dd \tau
	        +
	        \tfrac{1}{\alpha+1}
			\int_{s}^{t} \int_\Omega 
			m\bigl(u^{\sigma}(\tau)\bigr) |\partial_x^3 u^{\sigma}(\tau) - G_\sigma^{\prime\prime}(u^{\sigma}(\tau)) \partial_x u^{\sigma}(\tau)|^{\alpha+1}\dd x\dd \tau \\
			\leq &  \liminf\limits_{h\searrow 0}\left[
			\tfrac{\alpha}{\alpha+1}\int_{s}^{t} \int_\Omega \frac{|\bar{j}^{h,\sigma}(\tau)|^{\frac{\alpha+1}{\alpha}}}{ m\bigl(\bar{u}^{h,\sigma}(\tau)\bigr)^{\frac{1}{\alpha}}}\dd x \dd \tau  \right.
			\\
			& 
			\left.+ \tfrac{1}{\alpha+1}
			\int_{s}^{t} \int_\Omega 
			m\bigl(\bar{u}^{h,\sigma}(\tau)\bigr) |\partial_x^3 \bar{u}^{h,\sigma}(\tau+h) - G_\sigma^{\prime\prime}(\bar{u}^{h,\sigma}(\tau+h)) \partial_x \bar{u}^{h,\sigma}(\tau+h)|^{\alpha+1}\dd x\dd \tau\right].
	    \end{split}
	\end{equation}
	In Lemma \ref{lem:uniform_estimates} we showed that
	\begin{equation*}
	    \hat{u}^{h,\sigma} \longrightarrow u^\sigma \quad \text{strongly in } C\bigl([0,T]\times\bar{\Omega}\bigr).
	\end{equation*} 
	In virtue of the Arzelà--Ascoli theorem, this implies equicontinuity and thus also the uniform convergence
	\begin{equation*}
    	\bar{u}^{h,\sigma} \longrightarrow u^\sigma \quad\text{in } [0,T]\times\bar{\Omega} \text{ as } h\to 0,
	\end{equation*}
	since \(\bar{u}^{h,\sigma}\) is a piecewise constant approximation of \(\hat{u}^{h,\sigma}\). Using the local Lipschitz continuity of the mobility function $m$, we find that
	\begin{equation*}
    	m(\bar{u}^{h,\sigma}) \longrightarrow m(u^\sigma) \quad \text{uniformly in } [0,T]\times\bar{\Omega} \text{ as } h\to 0.
	\end{equation*}
	Since $u^{\sigma} \geq c_{\sigma,u_0}$ by Proposition \ref{prop:convergences} and $m(s) > 0,\, s > 0$ by assumption, we find that there is a constant $c_{m,\sigma,u_0}>0$ such that
	\begin{equation*}
	    m(u^{\sigma}) > c_{m,\sigma,u_0} > 0, \quad t\geq 0,\, x\in \bar{\Omega}.
	\end{equation*}
	Combining this with
	\begin{equation*}
	\bar{j}^{h,\sigma}\, \xrightharpoonup{\phantom{\text{wii}}} j^\sigma \quad \text{weakly in } L_{\frac{\alpha+1}{\alpha}}\bigl((0,\infty)\times\Omega\bigr)
	\end{equation*}
	by Lemma \ref{lem:uniform_estimates}, this implies by weak lower semicontinuity of the norm that
	\begin{equation*}
	\int_s^t \int_\Omega \frac{|j^\sigma(\tau)|^{\frac{\alpha+1}{\alpha}}}{m(u^\sigma(\tau))^{\frac{1}{\alpha}}}\dd x\dd \tau
	\leq 
	\liminf_{h \to 0} 
	\int_{s}^{t} \int_\Omega \frac{|\bar{j}^{h,\sigma}(\tau)|^{\frac{\alpha+1}{\alpha}}}{ m\bigl(\bar{u}^{h,\sigma}(\tau)\bigr)^{\frac{1}{\alpha}}}\dd x \dd \tau.
	\end{equation*}
	For the second term in \eqref{eq:liminf} we use that 
	\begin{equation*}
	    \partial_x^3 \bar{u}^{h,\sigma}(\cdot+h) - G_\sigma^{\prime\prime}(\bar{u}^{h,\sigma}(\cdot+h)) \partial_x \bar{u}^{h,\sigma}(\cdot+h)\, \xrightharpoonup{\phantom{\text{wii}}} \partial_x^3 u^{\sigma} - G_\sigma^{\prime\prime}(u^{\sigma}) \partial_x u^{\sigma} \quad \text{weakly in } L_{\alpha+1}\bigl((0,\infty)\times\Omega\bigr),
	\end{equation*}
	since $\partial_x^3 \bar{u}^{h,\sigma}(\cdot+h) - G_\sigma^{\prime\prime}(\bar{u}^{h,\sigma}(\cdot+h)) \partial_x \bar{u}^{h,\sigma}(\cdot+h)$ is uniformly bounded in $L_{\alpha+1}\bigl((\eps,\infty)\times\Omega\bigr)$ for every \(\eps >0\) and converges to $\partial_x^3 u^{\sigma} - G_\sigma^{\prime\prime}(u^{\sigma}) \partial_x u^{\sigma}$ in the sense of distributions. Combining this with the uniform convergence of \(m(\bar{u}^{h,\sigma})\) to \(m(u^{\sigma})\) and by weak lower semicontinuity of the norm, we deduce that
	\begin{equation*}
	\begin{split}
	&\int_s^t \int_\Omega  m\bigl(u^\sigma(\tau)\bigr)|\partial_x^3 u_\tau^\delta|^{\alpha+1}\dd x\dd \tau\\
	&\leq 
	\liminf_{h \to 0} 
	\int_{s}^{t} \int_{\Omega}  
	m\bigl(\bar{u}^{h,\sigma}(\tau)\bigr)|\partial_x^3 \bar{u}^{h,\sigma}(\tau+h) - G_\sigma^{\prime\prime}(\bar{u}^{h,\sigma}(\tau+h)) \partial_x \bar{u}^{h,\sigma}(\tau+h)|^{\alpha+1}\dd x\dd \tau.
    \end{split}
	\end{equation*}   
	
\noindent\textbf{(ii) Continuity equation. } That the continuity equation is satisfied for all $\phi \in C^\infty\bigl([0,\infty)\times\bar{\Omega}\bigr)$ with compact support in time in the limit $h\to 0$ follows from the time-discrete continuity equation in Lemma \ref{lem:Discrete_energy-dissipation} (ii) and the weak convergence results of Lemma \ref{prop:convergences}. By density, we may extend this to $\phi \in L_{\alpha+1}\bigl((0,\infty);W^1_{\alpha+1}(\Omega)\bigr)$. This completes the proof.
\end{proof}


\subsection{Energy-dissipation equality}

In this subsection we prove that the accumulation points $(u^\sigma,j^\sigma)_\sigma$ of the minimising-movement scheme constructed in the previous subsection are weak solutions to the modified thin-film equation
\begin{equation} \label{eq:PDE_mod_2}
    \begin{cases}
        \partial_t u^\sigma + \partial_x \bigl(\Psi
        \bigl(\partial_x^3 u^\sigma  - G^{\prime\prime}_\sigma(u^\sigma) \partial_x u^\sigma\bigr)
        \bigr) 
        = 0, & t > 0,\ x \in \Omega,
        \\
        \partial_x u^{\sigma} 
        =
        \partial_x^3 u^\sigma  - G^{\prime\prime}_\sigma(u^\sigma) \partial_x u^\sigma
        = 0, & t > 0,\ x \in \partial\Omega,
    \end{cases}
\end{equation}
and satisfy the energy-dissipation equality
\begin{equation*}
    E^\sigma[u^\sigma](t)
    +
    \int_s^t \int_\Omega m(u^\sigma(\tau)) 
    |\partial_x^3 u^\sigma(\tau) - G^{\prime\prime}_\sigma(u^\sigma(\tau)) \partial_x u^\sigma(\tau)|^{\alpha+1} \dd x \dd \tau
    =
    E^\sigma[u^\sigma](s)
\end{equation*}
for all $0 \leq s < t < \infty$.
In fact, we show that any pair
\begin{equation*}
    (u^\sigma,j^\sigma) 
    \in 
    L_{\alpha+1,\text{loc}}\bigl((0,\infty);W^3_{\alpha+1,B}(\Omega)\bigr) 
    \times L_{\frac{\alpha+1}{\alpha},\text{loc}}\bigl((0,\infty)\times\Omega\bigr)
\end{equation*}
with \(\partial_t u^\sigma
    \in 
    L_{\frac{\alpha+1}{\alpha},\text{loc}}\bigl((0,\infty);\bigl(W^1_{\alpha+1,B}(\Omega)\bigr)'\bigr)\) satisfies the energy-dissipation inequality \eqref{eq:energy_inequ} and the
continuity equation
\begin{equation} \label{eq:continuity_equation_v}
    \begin{cases}
        \partial_t u^{\sigma} + \partial_x j^{\sigma} 
        = 0, 
        & t > 0,\ x \in \Omega,
        \\
        \partial_x u^{\sigma} 
        = 
        j^{\sigma} 
        =
        0, & t > 0,\ x \in \partial\Omega,
    \end{cases}
\end{equation}
then, $u^\sigma$ is already a solution to the modified thin-film equation \eqref{eq:PDE_mod_2}. This is the content of the following proposition.


\begin{proposition}\label{prop:modified_thin-film_v}
If a pair
\begin{equation*}
	(v,k) 
	\in 
    L_{\alpha+1,\loc}\bigl((0,\infty);W^3_{\alpha+1,B}(\Omega)\bigr) 
    \times L_{\frac{\alpha+1}{\alpha},\loc}\bigl((0,\infty)\times\Omega\bigr)
\end{equation*}
with \(
    \partial_t v
    \in 
    L_{\frac{\alpha+1}{\alpha},\loc}\bigl((0,\infty);\bigl(W^1_{\alpha+1,B}(\Omega)\bigr)'\bigr)\) satisfies the continuity equation \eqref{eq:continuity_equation_v} in the sense that
\begin{equation} \label{eq:weak-continuity_equation_v}
	\int_0^\infty \langle \partial_t v_t, \phi\rangle_{W^1_{\alpha+1}}\dd t
	- 
	\int_0^\infty \int_\Omega k_t\cdot \partial_x \phi\dd x\dd t = 0
\end{equation}
for all $\phi \in L_{\alpha+1}\bigl([0,\infty);W^1_{\alpha+1}(\Omega)\bigr)$ compactly supported in time, then 
\begin{equation}\label{eq:rev_energy_inequ}
	E^\sigma[v](t)
	+ 
	\tfrac{\alpha}{\alpha+1}\int_s^t\int_\Omega \frac{|k|^{\frac{\alpha+1}{\alpha}}}{m(v)^{\frac{1}{\alpha}}}\dd x\dd \tau 
	+ 
	\tfrac{1}{\alpha+1}\int_s^t\int_\Omega m(v)|\partial_x^3 v - G_\sigma^{\prime\prime}(v) \partial_x v|^{\alpha+1}\dd x\dd \tau 
	\geq 
	E^\sigma[v](s)
\end{equation}
holds for all $0 \leq s < t < \infty$.
Moreover, provided that $E^\sigma[v_0] < \infty$, equality holds if and only if $v$ is a weak solution to the modified thin-film equation
\begin{equation*} 
	\begin{cases}
	    \partial_t v + \partial_x\bigl(m(v) \Psi
	    \bigl(\partial_x^3 v - G_\sigma^{\prime\prime}(v) \partial_x v\bigr)\bigr) = 0, & t > 0,\ x \in \Omega,
	    \\
	    \partial_x v 
	    = 
	    \partial_x^3 v - G_\sigma^{\prime\prime}(v) \partial_x v
	    =
	    0, & t > 0,\ x \in \partial\Omega,
	\end{cases}
\end{equation*}
	in the sense that $v$ satisfies the equation
\begin{equation} \label{eq:modified_thin-film_weak_v}
	\int_0^\infty \langle \partial_t v,\phi\rangle_{W^1_{\alpha+1}}\dd t 
	- 
	\int_0^\infty \int_{\Omega} 
	m(v) \Psi
	\bigl(\partial_x^3 v - G_\sigma^{\prime\prime}(v) \partial_x v\bigr) \partial_x \phi\dd x\dd t 
	= 0
\end{equation}
for all $\phi \in L_{\alpha+1}\bigl([0,\infty);W^1_{\alpha+1}(\Omega)\bigr)$ with compact support in time.
\end{proposition}

\begin{proof}
	\noindent\textbf{(i) Absolute continuity of $E^{\sigma}$ in time. } We prove that for all $0 \leq s < t < \infty$, we have
	\begin{equation} \label{eq:abs_cont}
	E^{\sigma}[v](t) - E^{\sigma}[v](s)
	=
	-\int_s^t \int_\Omega k(\tau) \bigl(\partial_x^3v(\tau) - G^{\prime\prime}_{\sigma}(v(\tau))\partial_xv(\tau)\bigr) \dd x\dd \tau. 
	\end{equation} 
	As a first step, we show that
	\begin{equation} \label{eq:abs_continuity}
	E^{\sigma}[v](t) - E^{\sigma}[v](s)
	=
	-\int_s^t \langle \partial_\tau v(\tau),\partial_x^2 v(\tau) - G^{\prime}_{\sigma}(v(\tau))\rangle_{W^1_{\alpha+1}}\dd \tau
	\end{equation}
	for all $v \in L_{\alpha+1}\bigl([0,T];W^3_{\alpha+1,B}(\Omega)\bigr)$ with $\partial_t v \in L_{\frac{\alpha+1}{\alpha}}\bigl([0,T];\bigl(W^1_{\alpha+1,B}(\Omega)\bigr)'\bigr)$ and $0 \leq s < t < \infty$. To this end, we mollify in time by introducing, for $\eps > 0$, 
	\begin{equation*}
	v^\eps = \eta_\eps \ast v \in C^\infty\bigl([\eps,\infty);W^3_{\alpha+1,B}(\Omega)\bigr).
	\end{equation*}
	Then, by the fundamental theorem of calculus, $v^\eps$ satisfies 
	\begin{equation}\label{eq:abs_continuity_eps}
	\begin{split}
	   	E^{\sigma}[v^{\eps}](t) - E^{\sigma}[v^{\eps}](s)
	    &= \int_{s}^{t} \partial_{\tau} E^{\sigma}[v^{\eps}](\tau) \dd \tau \\
	    & = \int_{s}^{t} \int_{\Omega} \partial_{\tau} \partial_x v^{\eps}(\tau)\partial_xv^{\eps}(\tau) + \partial_{\tau}v^{\eps} G'(v^{\eps}(\tau)) \dd x \dd t \\
	    &  = -\int_s^t \langle \partial_\tau v^\eps(\tau),\partial_x^2 v^\eps(\tau)- G'(v^{\eps}(\tau))\rangle_{W^1_{\alpha+1}}\dd \tau
	\end{split}
	\end{equation}
	for all $0 < \eps \leq s < t < \infty$. Moreover, we know that, for every $s > 0$,
	\begin{equation*}
	\begin{cases}
	v^\eps \longrightarrow v & \text{strongly in } C\bigl([s,\infty);H^1(\Omega)\bigr),
	\\
	\partial_x^2 v^\eps \longrightarrow \partial_x^2 v & \text{strongly in } L_{\alpha+1}\bigl([s,\infty);W^1_{\alpha+1}(\Omega)\bigr),
	\\
	\partial_t v^\eps \longrightarrow \partial_t v & \text{strongly in } L_{\frac{\alpha+1}{\alpha}}\bigl([s,\infty);\bigl(W^1_{\alpha+1,B}(\Omega)\bigr)'\bigr)
	\end{cases}
	\end{equation*} 
	as $\eps \to 0$. Here, the first convergence stated follows again from the generalised Lions--Magenes theorem, \cite[Remark 3.4]{bernis_existence_1988},  that is
	\begin{equation*}
	    v\in L_{\alpha+1}\bigl([0,T];W^3_{\alpha+1,B}(\Omega)\bigr)
	    \quad \text{with} \quad
	    \partial_t v \in  L_{\frac{\alpha+1}{\alpha}}\bigl([0,T];\bigl(W^1_{\alpha+1,B}(\Omega)\bigr)'\bigr)
	\end{equation*}
 	implies that \(v\) has a continuous representative
 		\[\tilde{v}\in C\bigl([0,T];H^1(\Omega)\bigr)\]
 	with \(\tilde{v} = v\) almost everywhere. Hence, we may take the limit in \eqref{eq:abs_continuity_eps} to obtain \eqref{eq:abs_continuity} for $0 < s < t < \infty$. The case $s=0$ follows by taking the limit $s \searrow 0$.
	
    By testing the continuity equation \eqref{eq:continuity_equation_v} with $\mathbf{1}_{[s,t]} \bigl(\partial_x^2 v(\tau) - G^{\prime}_{\sigma}(v(\tau))\bigr)\in L_{\alpha+1}\bigl([0,\infty);W^1_{\alpha+1}(\Omega)\bigr)$, where $\mathbf{1}_{[s,t]}$ denotes the characteristic function of the interval $[s,t]$, we note that 
	\begin{equation*}
	    \int_s^t \langle \partial_\tau v(\tau),\partial_x^2 v(\tau) - G^{\prime}_{\sigma}(v(\tau))\rangle_{W^1_{\alpha+1}}\dd \tau
	    =
	    \int_s^t \int_{\Omega} k(\tau) \bigl(\partial_x^3 v(\tau) - G^{\prime\prime}_{\sigma}(v(\tau))\partial_xv(\tau)\bigr) \dd x \dd \tau.
	\end{equation*}
	
	\noindent\textbf{(ii) Reverse energy-dissipation inequality. } By an application of Young's inequality, we have
	\begin{equation} \label{eq:Young}
	-k(\tau)\bigl(\partial_x^3 v(\tau)-G^{\prime\prime}_{\sigma}(v(\tau))\partial_xv(\tau)\bigr)
	\geq
	- \tfrac{\alpha}{\alpha+1} \frac{|k(\tau)|^{\frac{\alpha+1}{\alpha}}}{m(v(\tau))^{\frac{1}{\alpha}}} 
	- \tfrac{1}{\alpha+1} m(v(\tau)) |\partial_x^3 v(\tau)-G^{\prime\prime}_{\sigma}(v(\tau))\partial_xv(\tau)|^{\alpha+1}
	\end{equation}
	almost everywhere in $[0,\infty)\times\Omega$. Together with \eqref{eq:abs_cont}, this proves \eqref{eq:rev_energy_inequ}.
	
	\noindent\textbf{(iii) Energy-dissipation equality and the modified thin-film equation. }
	Recall that equality in \eqref{eq:Young} holds if and only if 
	\begin{equation*}
	k(\tau) = \pm m(v(\tau)) \Psi\bigl(\partial_x^3 v(\tau) - G^{\prime\prime}_{\sigma}(v(\tau))\partial_xv(\tau)\bigr)
	\quad \text{a.e. in } [0,\infty)\times\Omega.
	\end{equation*}
	Since the energy $E^\sigma[v]$ may not increase, we obtain that equality in \eqref{eq:rev_energy_inequ} holds for all $0\leq s < t < \infty$ if and only if 
	\begin{equation} \label{eq:flux_k}
	k(\tau) = m(v(\tau)) \Psi\bigl(\partial_x^3 v(\tau) - G^{\prime\prime}_{\sigma}(v(\tau))\partial_xv(\tau)\bigr)
	\quad \text{a.e. in } [0,\infty)\times\Omega.
	\end{equation}
	Inserting this in the continuity equation \eqref{eq:continuity_equation_v} proves \eqref{eq:modified_thin-film_weak_v}. 
	
	To prove that, if $E^\sigma[v_0] < \infty$, solutions to the modified thin-film equation \eqref{eq:modified_thin-film_weak_v} satisfy the energy-dissipation equality, observe that the pair $(v,k)$ satisfies the continuity equation \eqref{eq:continuity_equation_v} with the choice \eqref{eq:flux_k}.
\end{proof}

Since by Proposition \ref{prop:energy_inequ} any accumulation point $(u^\sigma,j^\sigma)$ of the family $(\hat{u}^{h,\sigma},\bar{j}^{h,\sigma})_h$ satisfies the conditions of Proposition \ref{prop:modified_thin-film_v}, we find that $u^\sigma$ is a weak solution to the modified thin-film equation \eqref{eq:PDE_mod_2}. For flow-behaviour exponents \(\alpha\neq 1\), this equations degenerates in the third derivative and hence we cannot claim uniqueness of solutions. For the Newtonian case \(\alpha=1\) though, uniqueness of positive solutions is well-known by standard parabolic theory.


\begin{theorem} \label{thm:well-posedness}
	Given $u_0 \in H^1(\Omega)$ with finite energy \(E^{\sigma}[u_0]<\infty\), there exists
	\begin{equation*}
	    u^\sigma \in C_b\bigl([0,\infty);H^1(\Omega)\bigr) \cap L_{\alpha+1,\loc}\bigl((0,\infty);W^3_{\alpha+1,B}(\Omega)\bigr) \quad \text{with}\quad \partial_t u^{\sigma}\in L_{\frac{\alpha+1}{\alpha}}\bigl([0,\infty);\bigl(W^{1}_{\alpha+1}(\Omega)\bigr)'\bigr)
	\end{equation*} 
	such that a subsequence of $(\hat{u}^{h,\sigma},\bar{j}^{h,\sigma})_h$ converges as follows:
	\begin{equation} \label{eq:convergence_h}
	\begin{cases}
	\hat{u}^{h,\sigma} \xrightharpoonup{\phantom{\text{wii}}} u^\sigma & \text{weakly in } L_{\alpha+1,\loc}\bigl([0,\infty);W^3_{\alpha+1,B}(\Omega)\bigr), 
	\\
	\partial_t\hat{u}^{h,\sigma} \xrightharpoonup{\phantom{\text{wii}}} \partial_t u^\sigma & \text{weakly in }
	L_{\frac{\alpha+1}{\alpha}}\bigl([0,\infty);\bigl(W^1_{\alpha+1,B}(\Omega)\bigr)'\bigr),
	\\
	\hat{u}^{h,\sigma} \longrightarrow u^\sigma & 
	\text{strongly in } C_{\loc}\bigl([0,\infty);C^{\nu}(\bar{\Omega})\bigr) \text{ for every } 0\leq \nu < \frac{1}{2},
	\\
	\bar{j}^{h,\sigma} \xrightharpoonup{\phantom{\text{wii}}} j^\sigma & \text{weakly in } L_{\frac{\alpha+1}{\alpha}}\bigl([0,\infty)\times\Omega\bigr).
	\end{cases}
	\end{equation}
	Furthermore, it holds
	\begin{equation} \label{eq:j_sigma}
	    j^\sigma
	    =
	    m(u^\sigma) \Psi\bigl(\partial_x^3 u^\sigma - G^{\prime\prime}_{\sigma}(u^{\sigma})\partial_x u^{\sigma} \bigr) \quad  \text{a.e. in } [0,\infty)\times\Omega. 
	\end{equation}
	Moreover, $u^\sigma$ is a positive weak solution to the initial-boundary-value problem
	\begin{equation}\label{eq:thin-film-reg}
	\begin{cases}
	\partial_t u^\sigma + \partial_x\bigl(m(u^\sigma) \Psi\bigl(\partial_x^3 u^\sigma - G^{\prime\prime}_{\sigma}(u^{\sigma})\partial_x u^{\sigma} \bigr)\bigr) = 0, & t > 0,\ x \in \Omega,
	\\
	\partial_x u^\sigma= m(u^\sigma) \Psi\bigl(\partial_x^3 u^\sigma - G^{\prime\prime}_{\sigma}(u^{\sigma})\partial_x u^{\sigma} \bigr) = 0, & t > 0,\ x \in \partial\Omega,
	\\
	u^\sigma(0,x) = u_0(x), & x \in \Omega,
	\end{cases}
	\end{equation}
	in the sense that
	\begin{equation}\label{eq:weak-formulation-mod}
	    \int_0^{\infty} \langle \partial_t u^{\sigma}, \phi \rangle_{W^1_{\alpha+1}} \dd t - \int_{0}^{\infty} \int_{\Omega} m(u^{\sigma}) \Psi\bigl(\partial_x^3 u^{\sigma}-G^{\prime\prime}_{\sigma}(u^{\sigma}) \partial_xu^{\sigma}\bigr) \partial_x \phi \dd x \dd t = 0
	\end{equation}
	holds for all \(\phi\in L_{\alpha+1}\bigl([0,\infty);W^1_{\alpha+1}(\Omega)\bigr)\). Furthermore, \(u^{\sigma}\) satisfies the energy-dissipation equality
	\begin{equation} \label{eq:EDI_sigma}
	E^{\sigma}[u^{\sigma}](t) + \int_s^t \int_\Omega  m(u^\sigma(\tau))|\partial_x^3 u^\sigma(\tau) - G^{\prime\prime}_{\sigma}(u^{\sigma}(\tau))\partial_xu^{\sigma}(\tau)|^{\alpha+1}\dd x\dd\tau 
	=  
	E^{\sigma}[u^{\sigma}](s)
	\end{equation}
	for all $0\leq s < t <\infty$. Furthermore, if \(\alpha=1\), i.e. if the fluid is Newtonian, there is exactly one accumulation point \(u^{\sigma}\) of the sequence \((\hat{u}^{h,\sigma})\) and \(u^{\sigma}\) is the unique weak solution to \eqref{eq:thin-film-reg}.
\end{theorem}


\begin{proof}
	\textbf{(i)} The convergence results and the regularity of the limit function $u^\sigma$ have been proved in Lemma \ref{lem:uniform_estimates}.
	Note that the weak convergences in \eqref{eq:convergence_h} are obtained in $[0,\infty)$ by using a diagonal-sequence argument. The strong convergence in \eqref{eq:convergence_h} is obtained only locally since it relies on the Arzelà--Ascoli theorem. However, the limit $u^\sigma$ is contained in $C_b\bigl([0,\infty);H^1(\Omega)\bigr)$ thanks to the energy-dissipation equality and the Lions--Magenes theorem.
	
	\noindent\textbf{(ii)} The energy-dissipation equality is satisfied in view of Proposition \ref{prop:energy_inequ} and Proposition \ref{prop:modified_thin-film_v}.
	
	\noindent\textbf{(iii)} That $u^\sigma$ satisfies the thin-film equation has been shown in Proposition \ref{prop:modified_thin-film_v}. Positivity of $u^{\sigma}$ follows from finite energy $E^{\sigma}[u^{\sigma}](t) \leq E^{\sigma}[u_0]<\infty$ for all $t\geq 0$ and Lemma \ref{lem:positivity_sigma}.
	
	\noindent\textbf{(iv)} Uniqueness of weak solutions --  and thus of the limit point -- follows from standard parabolic theory, see e.g. \cite{pazy_semigroups_1983,amann_nonhomogeneous_1993,lunardi_analytic_2012}, using that $m(u) \geq c_{m,\sigma,u_0}$ for all $(t,x)\in [0,\infty)\times \bar{\Omega}$.
\end{proof}


Finally, we observe that from the energy-dissipation equality \eqref{eq:EDI_sigma} for \(u^{\sigma}\), one obtains a uniform Hölder bound in \(C^{\frac{1}{5\alpha+3},\frac{1}{2}}([0,T]\times\bar{\Omega})\). The proof follows an argument given by \cite[Lemma 4.2]{grun_nonnegativity_2000} or \cite[Lemma 3.1]{otto_lubrication_1998} for the case $\alpha=1$ and $G_\sigma \equiv 0$. By the Arzelà--Ascoli theorem, this guarantees uniform convergence of \(m(u^{\sigma})\) to \(m(u)\).

\begin{lemma} \label{lem:Hoelder}
	Given an initial datum $u_0 \in H^1(\Omega)$ with finite energy $E^\sigma[u_0] < \infty$, let
	\begin{equation*}
	    (u^\sigma,j^\sigma)
	    \in L_{\alpha+1,\loc}\bigl((0,\infty);W^3_{\alpha+1,B}(\Omega)\bigr) 
	    \times
	    L_{\frac{\alpha+1}{\alpha}}\bigl([0,\infty)\times\Omega\bigr)
	    \quad
	    \text{with}
	    \quad
	    \partial_t u^\sigma \in 
	    L_{\frac{\alpha+1}{\alpha}}\bigl([0,\infty);\bigl(W^{1}_{\alpha+1,B}(\Omega)\bigr)'\bigr) 
	\end{equation*}
	satisfy the energy-dissipation equality \eqref{eq:EDI_sigma}.
	Then, for all \(T>0\), we have $u^\sigma \in  C^{\frac{1}{5\alpha+3},\frac{1}{2}}([0,T]\times\bar{\Omega})$ with a bound\vspace{-12pt}
	\begin{equation}\label{eq:Holder_bound}
	    \|u^\sigma\|_{C^{\frac{1}{5\alpha+3},\frac{1}{2}}([0,T]\times\bar{\Omega})}
	    \leq
	    C E^\sigma[u_0]^{\frac{1}{2}},
	\end{equation}
	where the constant $C > 0$ may depend on $T$ and $\alpha$, but not on $\sigma$.
\end{lemma}


\begin{proof}
	We already know from the Lions--Magenes theorem that $u^\sigma \in C_b\bigl([0,\infty);H^1(\Omega)\bigr)$. In view of the Sobolev embedding theorem, we thus find that $u^\sigma \in C_b\bigl([0,\infty);C^{1/2}(\bar{\Omega})\bigr)$.
	
	It remains to prove the H\"older continuity in time. Let $T > 0$ and consider $(U^\sigma,J^\sigma)$ such that $U^\sigma$ is the even extension of $u^\sigma$ and $J^\sigma$ is the odd extension of $j^\sigma$ about $\partial\Omega$. Moreover, let  $\eta_\eps$ be a standard mollifier in space and consider for $x \in \bar{\Omega}$ and $0 \leq s < t \leq T$
	\begin{equation*}
	\begin{split}
	   |u^\sigma(t,x) - u^\sigma(s,x)|
	   &\leq
	   |U^\sigma(t,x) - \eta_\eps \ast U^\sigma(t,x)|
	   +
	   |\eta_\eps \ast U^\sigma(t,x) - \eta_\eps \ast U^\sigma(s,x)|
	   \\
	   &\quad
	   +
	   |\eta_\eps \ast U^\sigma(s,x) - U^\sigma(s,x)|
	   \\
	   &=
	   (I) + (II) + (III).
	\end{split}
	\end{equation*}
	Since $u^\sigma(t,\cdot) \in C^\frac{1}{2}(\bar{\Omega})$, we find that $(I)$ and $(III)$ satisfy
	\begin{equation*}
	    (I) + (III)
	    \leq
	    \eps^{\frac{1}{2}} \left(
	    \left[U^\sigma(t,\cdot)\right]_{C^{1/2}(\bar{\Omega})}
	    +
	    \left[U^\sigma(s,\cdot)\right]_{C^{1/2}(\bar{\Omega})}
	    \right)
	    \leq
	    2 \eps^{\frac{1}{2}} E^\sigma[u_0]^\frac{1}{2}.
	\end{equation*}
	Testing the continuity equation with $\phi(\tau,y) = \textbf{1}_{[s,t]}(\tau) \eta_\eps(y-x)$, the second term may be estimated using the bound 
	\(
	\|\eta_{\eps}'\|_{L^{\alpha+1}(\Omega)} \leq C\eps^{-(2\alpha+1)/(\alpha+1)} \) and Hölder's inequality
	\begin{align*}
	(II) 
	&=
	|\eta_\eps \ast U^\sigma(t,x) - \eta_\eps \ast U^\sigma(s,x)| 
	\\
	&=
	\left|\int_s^t \langle \partial_\tau U^\sigma(\tau,\cdot), \eta_\eps(x-\cdot) \rangle_{W^1_{\alpha+1}} \dd \tau\right|
	\\
	&=
	\left|\int_s^t \int_{\R} \eta^\prime_\eps(x-y) J^\sigma(\tau,y)\dd y\dd \tau\right|
	\\
	&\leq
	C \eps^{\frac{-2\alpha+1}{\alpha+1}} |t - s|^\frac{1}{\alpha+1}
	\|J^\sigma\|_{L_{\frac{\alpha+1}{\alpha}}((0,T)\times\Omega)}
	\\
	&\leq
	C \eps^{\frac{-2\alpha+1}{\alpha+1}} |t - s|^\frac{1}{\alpha+1} \left(\int_0^T \int_\Omega \frac{|j^\sigma|^{\frac{\alpha+1}{\alpha}}}{m(u^\sigma)^{\frac{1}{\alpha}}}\dd x\dd t\right)^{\frac{\alpha}{\alpha+1}} \|m(u^\sigma)\|_{L_\infty((0,T)\times \Omega)}^\frac{1}{\alpha+1}
	\\
	&\leq
	C \eps^{\frac{-2\alpha+1}{\alpha+1}} |t - s|^\frac{1}{\alpha+1}
	E^\sigma[u_0],
	\end{align*}
	where we use that $m$ is locally Lipschitz and $u^\sigma$ is uniformly bounded, so that $\|m(u^\sigma)\|_{L_\infty((0,T)\times\Omega)} \leq C E^\sigma[u_0]^{\frac{1}{2}}$. Combining the estimates for $(I),(II)$ and $(III)$, we obtain
	\begin{align*}
	    |u^\sigma(t,x) - u^\sigma(s,x)| 
	    \leq 
	    2 \eps^{\frac{1}{2}} E^\sigma[u_0]^\frac{1}{2}
	    +
	    C \eps^{\frac{-2\alpha+1}{\alpha+1}} |t - s|^\frac{1}{\alpha+1}
	    E^\sigma[u_0]^\frac{1}{2}
	\end{align*}
	for every \(\eps >0\). Optimising in $\eps$, we may choose $\eps = |t-s|^{\frac{2}{5\alpha+3}}$ and obtain
	\begin{equation*}
	    |u^\sigma(t,x) - u^\sigma(s,x)|
    	\leq
	    2 \eps^{\frac{1}{2}} E^\sigma[u_0]^\frac{1}{2}
	    +
	    C \eps^{\frac{-2\alpha+1}{\alpha+1}} |t - s|^\frac{1}{\alpha+1}
	    E^\sigma[u_0]^\frac{1}{2}
	    \leq
	    C |t - s|^{\frac{1}{5\alpha+3}}E^\sigma[u_0]^\frac{1}{2}.
	\end{equation*}
	This concludes the proof.
\end{proof}

\section{The limit $\sigma \to 0$: weak solutions to the thin-film equation with potential}\label{sec:limit}

Now we investigate the limit \(\sigma \to 0\) and show that, for a non-negative initial datum \(u_0 \geq 0\) with finite energy $E[u_0] < \infty$, every accumulation point \(u\) of the family \(u^{\sigma}\) is a global-in-time non-negative weak solution to the power-law thin-film equation with potential $G$
\begin{align}\label{eq:power-law-equation}\tag{$P$}
	\begin{cases}
		\partial_t u + \partial_x\bigl( m(u) \Psi\bigl(\partial_x^3 u - G^{\prime\prime}(u) \partial_x u\bigr) \bigr) = 0, & t\in (0,\infty),\ x \in \Omega, \\
		\partial_x u = 
		 m(u) \Psi \bigl(\partial_x^3 u - G^{\prime\prime}(u) \partial_x u\bigr)
		= 0, & t\in (0,\infty),\ x \in \partial\Omega, \\
		u(0) = u_0, & x\in \Omega
	\end{cases}
\end{align}
in the sense that
\begin{equation*}
	\int_0^\infty  \langle \partial_t u, \phi\rangle_{W^{1}_{\alpha+1}}\dd t
	-
	\iint_{\{u> 0\}} m(u) \Psi\bigl(\partial_x^3 u - G^{\prime\prime}(u) \partial_x u\bigr)\, \partial_x \phi\dd x\dd t
	=
	0
\end{equation*}
for all \(\phi\in L_{\alpha+1}\bigl([0,\infty);W^1_{\alpha+1}(\Omega)\bigr)\). In the Newtonian case $\alpha=1$ the solution concept agrees with the one in \cite{bernis_higher_1990}. 

The proof relies on uniform (in $\sigma$) a-priori estimates based on the energy-dissipation equality \eqref{eq:EDI_sigma}. Fix an initial condition \(u_0\in H^1(\Omega)\) with \(u_0\geq 0\) and
\begin{equation*}
    E[u_0] = \int_{\Omega} \tfrac{1}{2}|\partial_x u_0|^2\dd x + G(u_0) \dd x < \infty.
\end{equation*}
Since the existence result Theorem \ref{thm:well-posedness} for the modified problem \eqref{eq:PDE_mod} is only valid for positive initial values $u^\sigma_0 \in H^1(\Omega),\, u^\sigma_0 > 0$, we now approximate non-negative $u_0 \geq 0$ by a sequence of positive initial data
\begin{equation} \label{eq:assumptions_u_sigma}
    u^\sigma_0(x) > 0, \quad x \in \Omega,
    \quad
    u^\sigma_0 \longrightarrow u_0
    \quad \text{strongly in } H^1(\Omega) 
    \quad \text{and} \quad 
    G_\sigma(u^\sigma_0) \longrightarrow G(u_0) \quad \text{in } L_1(\Omega).
\end{equation}

The strong $H^1(\Omega)$-convergence $u^\sigma_0 \longrightarrow u_0$ in \eqref{eq:assumptions_u_sigma} guarantees the convergence of the Dirichlet energy. Together with the $L_1(\Omega)$-convergence $G_\sigma(u^\sigma_0) \longrightarrow G(u_0)$, following from the continuity of $G$ in zero, cf. \ref{it:G4}, this yields convergence of the energy functionals
\begin{equation*}
    E^{\sigma}[u_0^{\sigma}] \longrightarrow E[u_0].
\end{equation*}
This in turn is needed in order to preserve the energy-dissipation inequality in the limit $\sigma \to 0$, cf. Theorem \ref{thm:solution_to_tfe} below.


Recall that given such initial data, by Theorem \ref{thm:well-posedness} and for every \(\sigma >0\), there exists a positive function
\begin{equation*}
    u^{\sigma} \in C_b\bigl([0,\infty);H^1(\Omega)\bigr) \cap L_{\alpha+1,\text{loc}}\bigl((0,\infty);W^3_{\alpha+1,B}(\Omega)\bigr) \quad \text{with} \quad \partial_t u^{\sigma} \in L_{\frac{\alpha+1}{\alpha}}\bigl([0,\infty);\bigl(W^1_{\alpha+1,B}(\Omega)\bigr)'\bigr)
\end{equation*}
which is a weak solution to the initial-boundary value problem
\begin{equation*}
    \begin{cases}
	\partial_t u^\sigma + \partial_x\bigl(m(u^\sigma) \Psi\bigl(\partial_x^3 u^\sigma - G^{\prime\prime}_{\sigma}(u^{\sigma})\partial_x u^{\sigma} \bigr)\bigr) = 0, & t > 0,\ x \in \Omega,
	\\
	\partial_x u^\sigma= m(u^\sigma) \Psi\bigl(\partial_x^3 u^\sigma - G^{\prime\prime}_{\sigma}(u^{\sigma})\partial_x u^{\sigma} \bigr) = 0, & t > 0,\ x \in \partial\Omega,
	\\
	u^\sigma(0,x) = u_0^{\sigma}(x), & x \in \Omega,
	\end{cases}
\end{equation*}
and satisfies the energy-dissipation equality
\begin{equation*}
	E^{\sigma}[u^{\sigma}](t) + \int_s^t \int_\Omega  m(u^\sigma(\tau))|\partial_x^3 u^\sigma(\tau) - G^{\prime\prime}_{\sigma}(u^{\sigma}(\tau))\partial_xu^{\sigma}(\tau)|^{\alpha+1}\dd x\dd\tau 
	=  
	E^{\sigma}[u^\sigma](s).
\end{equation*}

Applying the Arzelà--Ascoli theorem and using that \(E^{\sigma}[u_0^{\sigma}]\) is uniformly bounded, Lemma \ref{lem:Hoelder} implies that there exists an accumulation point \(u\in C^{\frac{1}{5\alpha+3},\frac{1}{2}}\bigl([0,\infty)\times\bar{\Omega}\bigr)\) of the sequence \((u^{\sigma})_{\sigma}\) such that a (non-relabelled) subsequence satisfies
\begin{equation}\label{eq:Hoelder-accumulation-point}
    u^{\sigma} \longrightarrow u \quad \text{in } C^{\gamma,\nu}_{\loc}\bigl([0,\infty)\times\bar{\Omega}\bigr)\quad  \text{for every } 0\leq\gamma < \frac{1}{5\alpha+3} \text{ and } 0\leq \nu < \frac{1}{2}.
\end{equation}

We stress again that we actually work with one accumulation point and one converging subsequence $(u^{\sigma_k})_k \subset (u^\sigma)_\sigma$. This is important in the following since uniform bounds will depend on this accumulation point and hence are only valid for elements $u^{\sigma_k}$ with $k$ large enough. However, in order to simplify notation, we suppress the subscript $k$.

For \(\eta\in \R\), we fix the notation
\begin{equation*}
    \{u > \eta\} \coloneq \left\{(t,x) \in [0,\infty) \times \bar{\Omega}\,\ u(t,x) > \eta\right\}.
\end{equation*}

In order to prove that \(u\) is a weak solution to the thin-film equation with potential \eqref{eq:power-law-equation} on the set \(\{u > 0\}\), we need further uniform bounds.

\begin{proposition}\label{prop:uniform-bounds-sigma}
	Let \(u_0\in H^1(\Omega)\) with \(u_0 \geq 0\) in \(\bar{\Omega}\) and \(E[u_0]<\infty\). Let \((u^{\sigma})_{\sigma}\) be the sequence of weak solutions to the modified thin-film equation obtained in Theorem \ref{thm:well-posedness} with initial datum \(u_0^{\sigma}\) as above. Let \(u \in C^{\frac{1}{5\alpha+3},\frac{1}{2}}\bigl([0,\infty)\times\Bar{\Omega}\bigr)\) be any accumulation point of the sequence \(u^{\sigma}\). Then there is \(\sigma_0>0\) small enough such that we have the following uniform bounds for every \(0< \sigma < \sigma_0\):
	\begin{enumerate}
		\item[(i)] \((u^{\sigma})_{\sigma}\) is uniformly bounded in \(L_{\infty}\bigl((0,\infty);H^1(\Omega)\bigr)\);
		\item[(ii)] \(m(u^{\sigma})^{\frac{1}{\alpha+1}}|\partial_x^3 u^{\sigma}- G^{\prime\prime}_{\sigma}(u^{\sigma})\partial_x u^{\sigma}|\) is uniformly bounded in \(L_{\alpha+1}\bigl((0,\infty)\times\Omega\bigr)\);
		\item[(iii)]  \((\partial_x^3 u^{\sigma})_{\sigma}\) is uniformly bounded in \(L_{\alpha+1,\loc}(\{u> 0\})\);
		\item[(iv)] \((\partial_t u^{\sigma})_\sigma\) is uniformly bounded in \(L_{\frac{\alpha+1}{\alpha}}\bigl((0,\infty);\bigl(W^1_{\alpha+1,B}(\Omega)\bigr)'\bigr)\);
		\item[(v)] \((\partial_t\partial_x u^{\sigma})_\sigma\) is uniformly bounded in \(L_{\frac{\alpha+1}{\alpha}}\bigl((0,\infty);\bigl(W^1_{\alpha+1,0}(\Omega)\cap W^2_{\alpha+1}(\Omega)\bigr)'\bigr)\).
	\end{enumerate}
\end{proposition}

\begin{proof}
	Fix \(\eta>0\). Given an accumulation point \(u\) of the sequence \(u^{\sigma}\), from the local uniform convergence we infer that for every \(T>0\) there is \(\sigma_0>0\) such that
	\begin{equation*}
	    u^{\sigma}(t,x)\geq \frac{\eta}{2}, \quad 
	    (t,x)\in \{u>\eta\} \cap \bigl([0,T]\times \bar{\Omega}\bigr)
	\end{equation*}
	and every \(\sigma< \sigma_0\).
	
	Furthermore, \(u^{\sigma}\) satisfies the energy-dissipation identity \eqref{eq:EDI_sigma}. With the choice \(s=0\) in \eqref{eq:EDI_sigma}, we obtain
	\begin{equation}\label{eq:EDI-sigma-0} 
	E^{\sigma}[u^\sigma](t) + \int_0^t \int_\Omega  m(u^\sigma(\tau))|\partial_x^3 u^\sigma(\tau) - G_{\sigma}^{\prime\prime}(u^{\sigma}(\tau))\partial_xu^{\sigma}(\tau)|^{\alpha+1}\dd x\dd\tau 
	=  
	E^{\sigma}[u^{\sigma}_0].
	\end{equation}
	\noindent\textbf{(i)} Recall that since \(G_{\sigma}(u^{\sigma}_0)\) converges uniformly to \(G(u_0)\) and \(u_0^{\sigma}\) converges to \(u_0\) in \(H^1(\Omega)\), we find that \(E^{\sigma}[u_0^{\sigma}]\) is uniformly bounded in \(\sigma\). From \eqref{eq:EDI-sigma-0}, we directly obtain
	\begin{equation*}
	    \int_\Omega \tfrac{1}{2} |\partial_x u^{\sigma}(t)|^2 \dd x \leq E^{\sigma}[u^{\sigma}](t) \leq E^{\sigma}[u_0^{\sigma}] \leq C
	\end{equation*}
	for every \(0< \sigma < \sigma_0\) and every \(t\geq 0\). Again relying on the Poincaré inequality, cf. Remark \ref{rem:cons_mass_poincare}, we conclude that \((u^{\sigma})_{\sigma}\) is uniformly bounded in \(L_{\infty}\bigl((0,\infty);H^1(\Omega)\bigr)\).
	
	\noindent\textbf{(ii)} This follows from \eqref{eq:EDI-sigma-0} via
	\begin{equation*}
	    \int_{0}^\infty \int_\Omega m(u^\sigma(\tau))|\partial_x^3 u^\sigma(\tau) - G_{\sigma}^{\prime\prime}(u^{\sigma}(\tau))\partial_xu^{\sigma}(\tau)|^{\alpha+1}\dd x\dd\tau \leq E^{\sigma}[u_0^{\sigma}] \leq C.
	\end{equation*}
	
	
	
	\noindent\textbf{(iii)} Let \(\eta >0\) as above and \(T>0\). It is enough to prove the uniform bound on the set \(\{u>\eta\} \cap \bigl([0,T]\times\bar{\Omega}\bigr)\). Recall from \eqref{eq:j_sigma} that
	\begin{equation*}
	    j^\sigma
	    =
	    m(u^\sigma) \Psi\bigl(\partial_x^3 u^\sigma - G^{\prime\prime}_{\sigma}(u^{\sigma})\partial_x u^{\sigma} \bigr) \quad  \text{a.e. in } [0,\infty)\times\Omega,
	\end{equation*}
	implying that
	\begin{equation}\label{eq:third-derivative-sigma}
	    \partial_x^3 u^\sigma
	    =
	    \frac{|j^\sigma|^\frac{1-\alpha}{\alpha} j^\sigma}{m(u^\sigma)^\frac{1}{\alpha}}
	    +
	    G^{\prime\prime}_{\sigma}(u^{\sigma})\partial_x u^{\sigma}
	    \quad  \text{a.e. in } [0,\infty)\times\Omega.
	\end{equation}
	Since $G^{\prime\prime}_{\sigma}(u^{\sigma})$ is uniformly bounded on $\{u > \eta\}$ and $\partial_x u^{\sigma} \in L_\infty\bigl((0,\infty);L_2(\Omega)\bigr)$ we have that
	\begin{equation*}
	    G^{\prime\prime}_{\sigma}(u^{\sigma})\partial_x u^{\sigma} \in L_2\bigl(\{u(t) > \eta\}\bigr)
	\end{equation*}
	for almost every $t \geq 0$ with a bound that is independent of $t$. Moreover, due to (ii) and since $m(u^\sigma)^\frac{1}{\alpha}$ is uniformly bounded we have that $j^\sigma \in L_\frac{\alpha+1}{\alpha}\bigl((0,\infty)\times\Omega\bigr)$
	and hence
	\begin{equation*}
	    \frac{|j^\sigma|^\frac{1-\alpha}{\alpha} j^\sigma}{m(u^\sigma)^\frac{1}{\alpha}} \in L_{\alpha+1}\bigl(\{u > \eta\}\bigr)
	\end{equation*}
	with a uniform bound. But then, by \cite{agmon_estimates_1959}, it follows
	\begin{equation*}
	    \|\partial_x^2 u^{\sigma}(t)\|_{L_2(\{u(t)> \eta\})} \leq C\|\partial_x u^{\sigma}(t)\|_{L_2(\Omega)} + C \|j^{\sigma}(t)\|_{L_{\frac{\alpha+1}{\alpha}}(\{u(t)>\eta\})}
	\end{equation*}
	for almost every \(t\in [0,\infty)\) and every \(\sigma < \sigma_0\). By the Sobolev embedding this implies that
	\begin{equation}
	    \|\partial_x u^{\sigma}(t) \|_{L_{\alpha+1}(\{u(t) > \eta\})} \leq C\|\partial_x u^{\sigma}(t)\|_{L_2(\Omega)} + C \|j^{\sigma}(t)\|_{L_{\frac{\alpha+1}{\alpha}}(\{u(t)>\eta\})}
	\end{equation}
	for almost every \(t\in [0,\infty)\) and every \(\sigma < \sigma_0\). But this together with \eqref{eq:third-derivative-sigma} implies that
	\begin{equation*}
	    \|\partial_x^3 u^{\sigma}\|_{L_{\alpha+1}(\{u>\eta\}\cap [0,T]\times\Omega)} \leq C\|j^{\sigma}\|_{L_{\frac{\alpha+1}{\alpha}}((0,T)\times \Omega)} + CTE^{\sigma}[u_0^{\sigma}],
	\end{equation*}
	which is the desired uniform bound.
	
	\noindent \textbf{(iv)} Since \(u^{\sigma}\) is a weak solution to \eqref{eq:thin-film-reg} in the sense of \eqref{eq:weak-formulation-mod}, we have
	\begin{equation*}
    	\int_0^{\infty} \langle \partial_t u^{\sigma}, \varphi \rangle_{W^1_{\alpha+1}} \dd t = \int_{0}^{\infty} \int_{\Omega} m(u^{\sigma}) \Psi\bigl(\partial_x^3 u^{\sigma} - G_{\sigma}^{\prime\prime}(u^{\sigma})\partial_x u^{\sigma}\bigr) \partial_x \varphi \dd x \dd t
	\end{equation*}
	for all \(\varphi\in L_{\alpha+1}\bigl((0,\infty);W^1_{\alpha+1}(\Omega)\bigr)\). Applying the Hölder inequality, using (ii) and that $(u^{\sigma})_{\sigma}$ is uniformly bounded in \(L_{\infty}\bigl([0,\infty)\times \Omega\bigr)\), we conclude
	\begin{align*}
    	&\left|\int_0^{\infty} \langle \partial_t u^{\sigma}, \varphi \rangle_{W^1_{\alpha+1}} \dd t\right|\\
    	& \leq \int_{0}^{T} \int_{\Omega} m(u^{\sigma})|\partial_x^3 u^{\sigma} - G_{\sigma}^{\prime\prime}(u^{\sigma})\partial_x u^{\sigma}|^{\alpha}|\partial_x \varphi| \dd x \dd t \\
    	& \leq C\|m(u^{\sigma})\|_{L_{\infty}((0,\infty)\times\Omega)}^{\frac{1}{\alpha+1}} \left(\int_{0}^{\infty} \int_{\Omega} m(u^{\sigma}) |\partial_x^3 u^{\sigma} - G_{\sigma}^{\prime\prime}(u^{\sigma})\partial_x u^{\sigma}|^{\alpha+1}\dd x \dd t\right)^{\frac{\alpha}{\alpha+1}}\left(\int_{0}^{\infty} |\partial_x\varphi|^{\alpha+1}\dd x \dd t\right)^{\frac{1}{\alpha+1}} \\
	    & \leq C_{u_0}\|\partial_x\varphi\|_{L_{\alpha+1}((0,T)\times \Omega)}.
	\end{align*}
	This  proves that \((\partial_tu^{\sigma})_{\sigma}\) is uniformly bounded in \(L_{\frac{\alpha+1}{\alpha}}\bigl([0,\infty);\bigl(W^1_{\alpha+1,B}(\Omega)\bigr)'\bigr)\).
	
	\noindent\textbf{(v)} This follows similarly as in (iv) by the same duality argument.
\end{proof}


Using the uniform bound, provided in Proposition \ref{prop:uniform-bounds-sigma}, we may extract a subsequence $(u^\sigma)_\sigma$, not relabelled, that converges in the sense of the following proposition.


\begin{proposition}\label{prop:convergence-u}
	Given a non-negative initial datum $u_0 \in H^1(\Omega),\ u_0 \geq 0$, with finite energy $E[u_0]<\infty$, the following holds true. There exists a subsequence of $(u^\sigma)_\sigma$ (not relabelled) and a limit
	\begin{equation*}
	u \in L_{\infty}\bigl((0,\infty);H^1(\Omega)\bigr) \cap C^{\frac{1}{5\alpha+3},\frac{1}{2}}\bigl([0,\infty]\times\bar{\Omega}\bigr)
	\end{equation*}
	with \(\partial_x^3 u \in L_{\alpha+1,\loc}(\{u > 0\})\) and \(\partial_tu\in L_{\frac{\alpha+1}{\alpha}}\bigl((0,\infty);\bigl(W^1_{\alpha+1}(\Omega)\bigr)'\bigr)\)
	such that we have convergence in the following sense:
	\begin{enumerate}
		\item[(i)] \(u^{\sigma}\to u\) strongly in \(C_{\loc}^{\gamma,\nu}\bigl([0,\infty)\times\bar{\Omega}\bigr)\) for all \(0\leq\gamma<\frac{1}{5\alpha+3}\) and \(0\leq \nu < \frac{1}{2}\);
		\item[(ii)] \(\partial_x^3 u^{\sigma} \rightharpoonup \partial_x^3 u\) weakly in \(L_{\alpha+1,\loc}(\{u> 0\})\);
		\item[(iii)] \(\partial_x u^{\sigma} \to \partial_x u\) and \(\partial_x^2 u^{\sigma} \to \partial_x^2 u\) strongly in \(L_{\alpha+1,\loc}(\{u>0\})\);
		\item[(iv)] \(\partial_t u^{\sigma} \rightharpoonup \partial_t u\) weakly in \(L_{\alpha+1}\bigl([0,\infty);\bigl(W^1_{\alpha+1}(\Omega)\bigr)'\bigr)\);
		\item[(v)] \(m(u^{\sigma})\Psi\bigl(\partial_x^3 u^{\sigma} - G_{\sigma}^{\prime\prime}(u^{\sigma})\partial_xu^{\sigma}\bigr) \rightharpoonup \chi\) weakly in  \(L_{\frac{\alpha+1}{\alpha}}\bigl((0,\infty)\times \Omega\bigr)\) for some limit function \(\chi\).
		
	\end{enumerate}
\end{proposition}

\begin{proof}
	The proof is divided into several steps.
	
	\noindent\textbf{(i)} This follows directly from Lemma \eqref{lem:Hoelder} combined with the Arzelà--Ascoli theorem.
	
	\noindent\textbf{(ii)} This follows directly from the uniform bound in Proposition \ref{prop:uniform-bounds-sigma} (iii).
	
	\noindent\textbf{(iii)} Due to the Aubin--Lions--Simon theorem, this is an immediate consequence of Proposition \ref{prop:uniform-bounds-sigma} (iii) and (iv).
	
	\noindent\textbf{(iv)} This follows from the corresponding uniform bound in Proposition \ref{prop:uniform-bounds-sigma} (iv).
	
	\noindent\textbf{(v)} By Proposition \ref{prop:uniform-bounds-sigma} (ii), we know that
	\(m(u^{\sigma})|\flux{\sigma}|^{\alpha-1}(\flux{\sigma})\) is uniformly bounded in \(L_{\frac{\alpha+1}{\alpha}}\bigl((0,\infty)\times\Omega\bigr)\). Hence, there exists a function  \(\chi \in L_{\frac{\alpha+1}{\alpha}}\bigl((0,\infty)\times\Omega\bigr)\) such that
	\begin{equation*}
	    m(u^{\sigma})\Psi\bigl(\flux{\sigma}\bigr) \xrightharpoonup{\phantom{\text{wii}}} \chi \quad \text{weakly in } L_{\frac{\alpha+1}{\alpha}}\bigl((0,\infty)\times\Omega\bigr).
	\end{equation*}
This completes the proof.
\end{proof}

It remains to identify the nonlinear limit flux \(\chi\). The proof relies on a localised version of Minty's trick. In order to apply Minty's trick, a natural choice for the test function would be \(\partial_x^2 u - G^{\prime}(u)\), but with the regularisation we have considered, we fail to give control of the second derivative of $u$ on \(\Omega\). Hence, we need to localise the standard argument to compact subsets of \(\{u>0\}\).
	
Observe that the method of identifying the limit via a local strong convergence result as presented in \cite[Prop. 3.1]{ansini_doubly_2004} is restricted to the case \(\alpha \geq 1\) of Newtonian and shear-thinning fluids, due to the failure of the inequality
\begin{equation*}
    |s - t|^{\alpha+1} \leq C \bigl(\Psi(s) - \Psi(t)\bigr)(s-t)
\end{equation*}
for \(p=\alpha+1<2\), cf. \cite[Lemma 4.4]{dibenedetto_degenerate_1993}. In fact, in the case \(\alpha \geq 1\), we are able to obtain the strong convergence of \(\partial_x^3 u^{\sigma}\) to \(\partial_x^3 u\) locally in \(\{u>0\}\) by the same argument as in \cite{ansini_doubly_2004}. Since we are interested in the general case including also shear-thickening fluids, we rely on Minty's trick.

\begin{proposition}\label{prop:identification_flux}
    Under the assumptions of Proposition we have \ref{prop:convergence-u}
    \begin{equation*}
        m(u^{\sigma})\Psi\bigl(\partial_x^3 u^{\sigma} - G_{\sigma}^{\prime\prime}(u^{\sigma})\partial_xu^{\sigma}\bigr) \rightharpoonup m(u)\Psi\bigl(\partial_x^3 u - G^{\prime\prime}(u)\partial_x u\bigr)\quad \text{weakly in } L_{\frac{\alpha+1}{\alpha}}\bigl((0,\infty)\times \Omega\bigr)
    \end{equation*}
    as \(\sigma \to 0\).
\end{proposition}

\begin{proof}
    We identify \(\chi\) on the set \(\{u>0\}\). Let \(K=I\times V \subset \{u>0\}\) be a compact set, where \(I\) and \(V\) are intervals. Choose closed, bounded intervals \(\tilde{I}\) and \(\tilde{V}\) such that \(\tilde{K} = \tilde{I}\times\tilde{V}\subset \{u>0\}\) and \(K\subset \operatorname{int}(\tilde{K})\). Finally, fix a spatial cut-off \(\eta \in C^{\infty}_c(\tilde{V};[0,1])\) such that \(\eta\equiv 1\) on \(V\).
	
	For \(t_0,t_1\in I\), we define the test functions
	\begin{equation*}
	    \begin{split}
	        \rho^{\sigma} & = \mathbf{1}_{[t_0,t_1]}
	        \bigl(
	        \partial_x[\eta^2 \partial_x u^{\sigma}] - \eta^2 G^{\prime}_\sigma(u^{\sigma})
	        \bigr)
	        \in L_{\alpha+1}\bigl((0,\infty);W^1_{\alpha+1,B}(\Omega)\bigr), \\
	        \rho & = \mathbf{1}_{[t_0,t_1]}
	        \bigl(
	        \partial_x[\eta^2 \partial_x u] - \eta^2 G^{\prime}(u)
	        \bigr)
	        \in L_{\alpha+1}\bigl((0,\infty);W^1_{\alpha+1,B}(\Omega)\bigr).
	   \end{split}
	\end{equation*}
	Define
	    \[E_{\eta}[u](t) = \int_{\Omega} \eta^2 \bigl[\tfrac{1}{2} |\partial_x u|^2 + G(u)\bigr]\dd x\quad \text{and}\quad E^{\sigma}_{\eta}[u](t) = \int_{\Omega} \eta^2 \bigl[\tfrac{1}{2} |\partial_x u|^2 + G_{\sigma}(u)\bigr]\dd x.\]
    Then, for almost every \(t_0,t_1\in I\) it holds
    \begin{equation}\label{eq:Minty-localised-EDE}
        E_{\eta}[u](t_1) - E_{\eta}[u](t_0) + \int_{t_0}^{t_1} \int_{\Omega} \chi \bigl[\partial_x^2(\eta^2 \partial_x u) - \partial_x(\eta^2G'(u))\bigr] \dd x \dd t = 0.
    \end{equation}
    
    The claim follows by using \(\rho\) as test function in \eqref{eq:weak-formulation-mod} and sending \(\sigma \to 0\). Indeed, we find that
    \begin{equation*}
       \int_{0}^{\infty} \langle \partial_t u^{\sigma}, \rho\rangle_{W^{1}_{\alpha+1}} \dd t \longrightarrow \int_{0}^{\infty} \langle \partial_t u, \rho\rangle_{W^{1}_{\alpha+1}} \dd t
    \end{equation*}
    as \(\sigma \to 0\) and
    \begin{equation*}
        \begin{split}
            \int_{0}^{\infty} \langle \partial_t u, \rho\rangle_{W^{1}_{\alpha+1}} \dd t & 
            = 
            -\int_{t_0}^{t_1} \frac{d}{dt} \int_{\Omega} \eta^2(x)\bigl[\tfrac{1}{2} |\partial_x u|^2 + G(u)\bigr]\dd x \dd t 
            = - \bigl(E_{\eta}[u](t_1)- E_{\eta}[u](t_0)\bigr).
        \end{split}
    \end{equation*}
    Further,
    \begin{equation*}
        \int_{0}^{\infty}\int_{\Omega} m(u^{\sigma}) \Psi\bigl(\flux{\sigma}\bigr) \partial_x \rho \dd x \dd t \longrightarrow \int_{t_0}^{t_1}\int_{\Omega} \chi \partial_x \rho \dd x\dd t.
    \end{equation*}
    Using \eqref{eq:weak-formulation-mod} gives \eqref{eq:Minty-localised-EDE}.
    
    As a preparation for Minty's trick, note that
	\begin{equation}\label{eq:Minty-test-function-split}
	\begin{split}
	    \int_{t_0}^{t_1} \int_{\Omega} \chi \bigl[\partial_x^2(\eta^2 \partial_x u) - \partial_x(\eta^2G'(u)\bigr] \dd x \dd t
	    = \int_{t_0}^{t_1} \int_{\Omega} \eta^2 \chi \bigl(\partial_x^3 u - G^{\prime\prime}(u)\partial_x u\bigr) \dd x \dd t + \int_{t_0}^{t_1} \int_{\Omega} \chi \psi \dd x \dd t,
	\end{split}
	\end{equation}
	where we denote
	\begin{equation*}
	    \psi = 4\eta\partial_x \eta \partial_x^2 u - 2 \eta \partial_x \eta G^{\prime}(u) + 2\partial_x(\eta\partial_x \eta) \partial_x u.
	\end{equation*}
	
	We are now in the position to set up Minty's trick. Let \(\phi \in L_{\alpha+1}\bigl((0,\infty);W^3_{\alpha+1,B}(\Omega)\bigr)\). Then, by monotonicity of \(\Psi\) and positivity of \(m\) and \(\eta\), we obtain
	\begin{equation}\label{eq:Minty-monotonicity}
	    \begin{split}
	        0 
	        \leq &
	        \int_{t_0}^{t_1} \int_{\Omega} 
	        \eta^2 m(u^{\sigma})\bigl[
	        \Psi\bigl(\partial_x^3 u^{\sigma} - G^{\prime\prime}_{\sigma}(u^{\sigma})\partial_x u^{\sigma}\bigr) 
	        - 
	        \Psi\bigl(\partial_x^3 \phi - G^{\prime\prime}_{\sigma}(u^{\sigma})\partial_x u^{\sigma}\bigr)\bigr]\cdot
	        \\
	        & \qquad\quad
	        \cdot
	        \bigl[\bigl(\partial_x^3 u^{\sigma}-G^{\prime\prime}_{\sigma}(u^{\sigma})\partial_x u^{\sigma}\bigr) - \bigl(\partial_x^3\phi-G^{\prime\prime}_{\sigma}(u^{\sigma})\partial_xu^{\sigma}\bigr)\bigr] \dd x \dd t.
	    \end{split}
	\end{equation}
	
	We split the integral on the right-hand side into the four terms
	\begin{equation*}
	    \begin{split}
	        & \int_{t_0}^{t_1} \int_{\Omega} 
	        \eta^2 m(u^{\sigma})\bigl[
	        \Psi\bigl(\partial_x^3 u^{\sigma} - G^{\prime\prime}_{\sigma}(u^{\sigma})\partial_x u^{\sigma}\bigr) 
	        - 
	        \Psi\bigl(\partial_x^3 \phi - G^{\prime\prime}_{\sigma}(u^{\sigma})\partial_x u^{\sigma}\bigr)\bigr]\cdot
	        \\
	        & \qquad\quad
	        \cdot
	        \bigl[\bigl(\partial_x^3 u^{\sigma}-G^{\prime\prime}_{\sigma}(u^{\sigma})\partial_x u^{\sigma}\bigr) - \bigl(\partial_x^3\phi-G^{\prime\prime}_{\sigma}(u^{\sigma})\partial_xu^{\sigma}\bigr)\bigr] \dd x \dd t\\
	        = & \int_{t_0}^{t_1} \int_{\Omega} \eta^2 m(u^{\sigma}) \Psi(\partial_x^3 u^{\sigma} - G^{\prime\prime}_{\sigma}(u^{\sigma})\partial_xu^{\sigma})\bigl(\partial_x^3 u^{\sigma}-G^{\prime\prime}_{\sigma}(u^{\sigma})\partial_xu^{\sigma}\bigr) \dd x \dd t \\
	        & - \int_{t_0}^{t_1} \int_{\Omega} \eta^2 m(u^{\sigma})\Psi(\partial_x^3 u^{\sigma} - G^{\prime\prime}_{\sigma}(u^{\sigma})\partial_xu^{\sigma})\bigl(\partial_x^3\phi-G^{\prime\prime}_{\sigma}(u^{\sigma})\partial_xu^{\sigma}\bigr) \dd x \dd t \\
	        & -\int_{t_0}^{t_1} \int_{\Omega} \eta^2 m(u^{\sigma}) \Psi(\partial_x^3 \phi - G^{\prime\prime}_{\sigma}(u^{\sigma})\partial_xu^{\sigma}) \bigl(\partial_x^3 u^{\sigma}-G^{\prime\prime}_{\sigma}(u^{\sigma})\partial_xu^{\sigma}\bigr) \dd x \dd t \\
	        & +\int_{t_0}^{t_1} \int_{\Omega} \eta^2 m(u^{\sigma}) \Psi(\partial_x^3 \phi - G^{\prime\prime}_{\sigma}(u^{\sigma})\partial_xu^{\sigma}) \bigl(\partial_x^3\phi-G^{\prime\prime}_{\sigma}(u^{\sigma})\partial_xu^{\sigma}\bigr) \dd x \dd t.
	    \end{split}
	\end{equation*}
	We study the limit as \(\sigma\to 0\) for each of the four summands. We start with the first summand: therefore, note again that
	    \[\partial_x\rho^{\sigma} = \eta^2 \bigl(\partial_x^3 u^{\sigma} - G^{\prime\prime}_{\sigma}(u^{\sigma})\partial_xu^{\sigma}\bigl) + \psi^{\sigma},\]
	where we denote
	    \[\psi^{\sigma} = 4\eta\partial_x \eta \partial_x^2u^{\sigma} - 2\eta \partial_x \eta G'_{\sigma}(u_{\sigma}) + 2\partial_x(\eta\partial_x\eta) \partial_xu^{\sigma}.\]
	The crucial observation is that since \(\psi^{\sigma}\) only contains lower-order derivatives and due to steps (i) and (iii) combined with the local uniform convergence of \(G_{\sigma}\) to \(G\), we obtain strong convergence
	    \[\psi^{\sigma} \longrightarrow \psi \quad \text{strongly in } L_{\alpha+1,\loc}(I\times\Omega)\]
	as $\sigma \to 0$. Then, using \(\rho^{\sigma}\) as a test function in \eqref{eq:weak-formulation-mod}, we may write
	\begin{equation*}
	    \begin{split}
	        & \int_{t_0}^{t_1} \int_{\Omega} \eta^2 m(u^{\sigma}) \Psi(\partial_x^3 u^{\sigma} - G^{\prime\prime}_{\sigma}(u^{\sigma})\partial_xu^{\sigma})\bigl(\partial_x^3 u^{\sigma}-G^{\prime\prime}_{\sigma}(u^{\sigma})\partial_xu^{\sigma}\bigr) \dd x \dd t \\
	        = & \int_{t_0}^{t_1} \int_{\Omega} m(u^{\sigma}) \Psi(\partial_x^3 u^{\sigma} - G^{\prime\prime}_{\sigma}(u^{\sigma})\partial_xu^{\sigma})\partial_x \rho^{\sigma} \dd x \dd t - \int_{t_0}^{t_1} \int_{\Omega} m(u^{\sigma}) \Psi(\partial_x^3 u^{\sigma} - G^{\prime\prime}_{\sigma}(u^{\sigma})\partial_xu^{\sigma})\psi^{\sigma} \dd x \dd t \\
	        = & \int_{t_0}^{t_1} \langle \partial_t u^{\sigma}, \rho^{\sigma}\rangle_{W^1_{\alpha+1}} \dd t -  \int_{t_0}^{t_1} \int_{\Omega} m(u^{\sigma}) \Psi(\partial_x^3 u^{\sigma} - G^{\prime\prime}_{\sigma}(u^{\sigma})\partial_xu^{\sigma})\psi^{\sigma} \dd x \dd t \\
	        = & -\int_{t_0}^{t_1} \frac{d}{dt} \int_{\Omega}\eta^2 \bigl[\tfrac{1}{2} |\partial_x u^{\sigma}|^2 + G_{\sigma}(u^{\sigma})\bigr] \dd x \dd t - \int_{t_0}^{t_1} \int_{\Omega} m(u^{\sigma}) \Psi(\partial_x^3 u^{\sigma} - G^{\prime\prime}_{\sigma}(u^{\sigma})\partial_xu^{\sigma})\psi^{\sigma} \dd x \dd t \\
	        = & -\bigl(E^{\sigma}_{\eta}[u^{\sigma}](t_1) - E^{\sigma}_{\eta}[u^{\sigma}](t_0)\bigr) - \int_{t_0}^{t_1} \int_{\Omega} m(u^{\sigma}) \Psi(\partial_x^3 u^{\sigma} - G^{\prime\prime}_{\sigma}(u^{\sigma})\partial_xu^{\sigma})\psi^{\sigma} \dd x \dd t.
	    \end{split}
	\end{equation*}
	Since by (i), we have \(u^{\sigma}(t)\to u(t)\) in \(H^1(\Omega)\) for almost every \(t\in I\), we find that
	\begin{equation}\label{eq:Minty-I1-energy}
	    E^{\sigma}_{\eta}[u^{\sigma}](t_1) - E^{\sigma}_{\eta}[u^{\sigma}](t_0) \longrightarrow E_{\eta}[u](t_1) - E_{\eta}[u](t_0)
	\end{equation}
    as \(\sigma \to 0\) for almost every \(t_0,t_1\in I\). Furthermore, by the strong convergence of \(\psi^{\sigma}\to \psi\), we obtain the convergence
	 \begin{equation}\label{eq:Minty-I1-remainder}
	     \lim\limits_{\sigma\to 0} \int_{t_0}^{t_1} \int_{\Omega} m(u^{\sigma}) \Psi(\partial_x^3 u^{\sigma} - G^{\prime\prime}_{\sigma}(u^{\sigma})\partial_xu^{\sigma})\psi^{\sigma} \dd x \dd t = \int_{t_0}^{t_1} \int_{\Omega} \chi \psi \dd x \dd t.
	 \end{equation}
	 Using \eqref{eq:Minty-I1-energy} and \eqref{eq:Minty-I1-remainder}, we find that
	 \begin{equation}\label{eq:Minty-I1}
	    \begin{split}
	     &\lim_{\sigma \to 0} \int_{t_0}^{t_1} \int_{\Omega} \eta^2 m(u^{\sigma}) \Psi(\partial_x^3 u^{\sigma} - G^{\prime\prime}_{\sigma}(u^{\sigma})\partial_xu^{\sigma})\bigl(\partial_x^3 u^{\sigma}-G^{\prime\prime}_{\sigma}(u^{\sigma})\partial_xu^{\sigma}\bigr) \dd x \dd t\\
	    & = -\bigl(E_{\eta}[u](t_1) - E_{\eta}[u](t_0)\bigr) -\int_{t_0}^{t_1} \int_{\Omega} \chi \psi \dd x \dd t =  \int_{t_0}^{t_1} \int_{\Omega} \eta^2 \chi\bigl(\partial_x^3 u - G^{\prime\prime}(u)\partial_x u\bigr) \dd x \dd t
	     \end{split}
	 \end{equation}
	 for almost every \(t_0,t_1\in I\). The last equality follows from \eqref{eq:Minty-localised-EDE} and \eqref{eq:Minty-test-function-split}.
	 
	 For the second term, observe that
	 \begin{equation}\label{eq:G-sigma-partial-strong-convergence}
	     G_{\sigma}^{\prime\prime}(u^{\sigma})\partial_xu^{\sigma} \longrightarrow G^{\prime\prime}(u) \partial_x u\quad \text{strongly in }L_{\alpha+1}\bigl(I\times \tilde{V}\bigr).
	 \end{equation}
	 So, we obtain
	 \begin{equation}\label{eq:Minty-I2}
     \begin{split}
	        &\lim_{\sigma \to 0} \int_{t_0}^{t_1} \int_{\Omega} \eta^2 m(u^{\sigma})\Psi(\partial_x^3 u^{\sigma} - G^{\prime\prime}_{\sigma}(u^{\sigma})\partial_xu^{\sigma})\bigl(\partial_x^3\phi-G^{\prime\prime}_{\sigma}(u^{\sigma})\partial_xu^{\sigma}\bigr) \dd x \dd t \\
	        &=  \int_{t_0}^{t_1} \int_{\Omega} \eta^2 \chi \bigl(\partial_x^3\phi-G^{\prime\prime}(u)\partial_xu\bigr) \dd x \dd t.
     \end{split}
	 \end{equation}
	 
	 For the third and fourth term note that due to the continuity of the Nemytskii operator and the strong convergence in \eqref{eq:G-sigma-partial-strong-convergence}, we have
	 \begin{equation*}
	     \Psi\bigl(\partial_x^3\phi - G_{\sigma}^{\prime\prime}(u^{\sigma})\partial_x u^{\sigma}\bigr) 
	     \longrightarrow 
	     \Psi\bigl(\partial_x^3\phi - G^{\prime\prime}(u) \partial_x u\bigr) 
	     \quad 
	     \text{strongly in }L_{\frac{\alpha+1}{\alpha}}(I\times \tilde{V})
	 \end{equation*}
	 as \(\sigma\to 0\). Applying this in the third integral and using that
	 \begin{equation*}
	     \partial_x^3 u^{\sigma}-G^{\prime\prime}_{\sigma}(u^{\sigma})\partial_xu^{\sigma} \rightharpoonup \partial_x^3 u - G^{\prime\prime}(u)\partial_x u \quad \text{weakly in } L_{\frac{\alpha+1}{\alpha}}(I\times{V}),
	 \end{equation*}
	 we obtain
	 \begin{equation}\label{eq:Minty-I3}
	 \begin{split}
	     &\lim_{\sigma \to 0}  \int_{t_0}^{t_1} \int_{\Omega} \eta^2 m(u^{\sigma}) \Psi\bigl(\partial_x^3 \phi - G^{\prime\prime}_{\sigma}(u^{\sigma})\partial_x u^{\sigma}\bigr) \bigl(\partial_x^3 u^{\sigma}-G^{\prime\prime}_{\sigma}(u^{\sigma})\partial_xu^{\sigma}\bigr) \dd x \dd t \\
	     & = 
	     \int_{t_0}^{t_1} \int_{\Omega} 
	     \eta^2 m(u) \Psi\bigl(\partial_x^3 \phi - G^{\prime\prime}(u)\partial_x u\bigr) \bigl(\partial_x^3 u-G^{\prime\prime}(u)\partial_xu\bigr) \dd x \dd t.
     \end{split}
	 \end{equation}
  
	 The same argument applied to the fourth integral gives
	 \begin{equation}\label{eq:Minty-I4}
	 \begin{split}
	     &\lim_{\sigma \to 0}  \int_{t_0}^{t_1} \int_{\Omega} \eta^2 m(u^{\sigma}) \Psi(\partial_x^3 \phi - G^{\prime\prime}_{\sigma}(u^{\sigma})\partial_xu^{\sigma}) \bigl(\partial_x^3\phi-G^{\prime\prime}_{\sigma}(u^{\sigma})\partial_xu^{\sigma}\bigr) \dd x \dd t \\
	     & = \int_{t_0}^{t_1} \int_{\Omega} \eta^2 m(u) \Psi(\partial_x^3 \phi - G^{\prime\prime}(u)\partial_xu) \bigl(\partial_x^3 \phi-G^{\prime\prime}(u)\partial_xu\bigr) \dd x \dd t.
    \end{split}
	 \end{equation}
	 By \eqref{eq:Minty-I1}, \eqref{eq:Minty-I2}, \eqref{eq:Minty-I3} and \eqref{eq:Minty-I4}, we obtain, when sending \(\sigma\to 0\) in \eqref{eq:Minty-monotonicity},
	 \begin{equation}\label{eq:Minty-main}
	     0 \leq \int_{t_0}^{t_1} \int_{\Omega} \eta^2 \bigl(\chi - m(u) \Psi(\partial_x^3\phi - G^{\prime\prime}(u)\partial_xu\bigr)\bigl(\partial_x^3(u-\phi)\bigr) \dd x \dd t
	 \end{equation}
	 for almost every \(t_0,t_1\in I\). Now, choose \(v\in L_{\alpha+1}\bigl((0,\infty);W^3_{\alpha+1}(\Omega)\bigr)\) and use
	    \[\phi = \tilde{\eta}^2 u + \lambda v, \quad \lambda >0,\]
	  where \(\tilde{\eta}\in C_c^{\infty}(\{u>0\})\) with \(\tilde{\eta}\equiv 1\) on \(\tilde{V}\). Then, from \eqref{eq:Minty-main}, we obtain
	  \begin{equation*}
	      0 \leq - \lambda \int_{t_0}^{t_1} \int_{\Omega} \eta^2 \bigl(\chi - m(u) \Psi(\partial_x^3(u+\lambda v) - G^{\prime\prime}(u)\partial_xu)\bigr)\bigl(\partial_x^3 v\bigr) \dd x \dd t.
	  \end{equation*}
	  Dividing by \(\lambda\), sending \(\lambda \to 0\) and using the continuity of the Nemitski operator once again, we find
	  \begin{equation*}
	      0 \leq - \int_{t_0}^{t_1} \int_{\Omega} \eta^2 \bigl(\chi - m(u) \Psi(\partial_x^3u - G^{\prime\prime}(u)\partial_xu)\bigr)\partial_x^3 v \dd x \dd t.
	  \end{equation*}
	  Using instead \(\phi = u - \lambda v\), \(\lambda >0\), gives
	  \begin{equation*}
	      0 \leq \int_{t_0}^{t_1} \int_{\Omega} \eta^2 \bigl(\chi - m(u) \Psi(\partial_x^3u - G^{\prime\prime}(u)\partial_xu)\bigr)\partial_x^3 v \dd x \dd t.
	  \end{equation*}
	  This finally proves that
	  \begin{equation*}
	      0 = \int_{t_0}^{t_1} \int_{\Omega} \eta^2 \bigl(\chi - m(u) \Psi(\partial_x^3u - G^{\prime\prime}(u)\partial_xu)\bigr)\partial_x^3 v \dd x \dd t
	  \end{equation*}
	 for all \(v\in L_{\alpha+1}\bigl((0,\infty);W^3_{\alpha+1}(\Omega)\bigr)\) and hence
	 \begin{equation*}
	     \chi - m(u) \Psi(\partial_x^3u - G^{\prime\prime}(u)\partial_xu) = 0 \quad \text{almost everywhere in } I \times V.
	 \end{equation*}
	 Since the compact set \(K=I\times V \subset \{u>0\}\) was arbitrary, this proves
	 \begin{equation}
	     \chi = m(u) \Psi(\partial_x^3u - G^{\prime\prime}(u)\partial_xu)\quad \text{almost everywhere in } \{u>0\}.
	 \end{equation}
	 To conclude the weak convergence on the set \((0,\infty)\times \Omega\), note that given \(\phi \in L_{\alpha+1}\bigl((0,\infty)\times \Omega\bigr)\), we have
	 \begin{equation*}
	 \begin{split}
	   & \int_0^{\infty}\int_{\Omega} m(u^{\sigma}) \Psi\bigl(\flux{\sigma}\bigr) \phi \dd x \dd t\\
	   = & \int_{\{u\geq 1/n\}} m(u^{\sigma}) \Psi\bigl(\flux{\sigma}\bigr) \phi \dd x \dd t + \int_{\{u<1/n\}} m(u^{\sigma}) \Psi\bigl(\flux{\sigma}\bigr) \phi \dd x \dd t.
	 \end{split}
	 \end{equation*}
	 Now, by the previous argument
	 \begin{equation*}
	 \begin{split}
	    \int_{\{u\geq 1/n\}} m(u^{\sigma}) \Psi\bigl(\flux{\sigma}\bigr) \phi \dd x \dd t
	     \longrightarrow \int_{\{u\geq 1/n\}} m(u) \Psi\bigl(\flux{}\bigr) \phi \dd x \dd t
	 \end{split}
	 \end{equation*}
	 as \(\sigma \to 0\).
	 Furthermore,
	 \begin{equation*}
	 \begin{split}
	    & \left|\int_{\{u<1/n\}} m(u^{\sigma}) \Psi\bigl(\flux{\sigma}\bigr) \phi \dd x \dd t\right| \\
	    &\leq \|m(u^{\sigma})\|_{L_{\infty}(\{u<1/n\})}^{\frac{1}{\alpha+1}}\|m(u^{\sigma})^{\frac{\alpha}{\alpha+1}}\Psi\bigl(\flux{\sigma}\bigr)\|\|\phi\|_{L_{\alpha+1}((0,\infty)\times \Omega)}
	 \end{split}
	 \end{equation*}
	 is uniformly small, when \(\sigma < \sigma_{0,n}\). Hence, first sending \(\sigma \to 0\) and then \(n\to \infty\), we obtain the weak convergence.
\end{proof}

With the convergence results of Propositions \ref{prop:convergence-u} and \ref{prop:identification_flux} we are now in the position to show that the limit $u$ is a non-negative global-in-time weak solution to the power-law thin-film equation on the positivity set $\{u > 0\}$. %

\begin{theorem} \label{thm:solution_to_tfe}
	Fix a non-negative initial datum $u_0 \in H^1(\Omega),\ u_0 \geq 0$ with finite energy $E[u_0] < \infty$. Assume that $G: \R \to [0,\infty]$ has the properties \ref{it:G1}--\ref{it:G4}.  Then, if \(u\) is an accumulation point of the sequence \((u^{\sigma})_{\sigma}\), as obtained in Proposition \ref{prop:convergence-u}, then
	\begin{equation*}
	    u \in L_{\infty}\bigl([0,\infty);H^1(\Omega)\bigr) \cap C^{\frac{1}{5\alpha+3},\frac{1}{2}}\bigl([0,\infty]\times\bar{\Omega}\bigr)
	\end{equation*}
	with \(\partial_x^3 u \in L_{\alpha+1,\loc}(\{u > 0\})\) and \(\partial_tu\in L_{\alpha+1}\bigl([0,\infty);\bigl(W^1_{\alpha+1}(\Omega)\bigr)'\bigr)\) is a global-in-time non-negative weak solution to the power-law thin-film equation with potential \(G\)
	\begin{equation*} 
	\begin{cases}
	\partial_t u + \partial_x\bigl(m(u) |\partial_x^3 u - G^{\prime\prime}(u) \partial_x u|^{\alpha-1} \bigl(\partial_x^3 u - G^{\prime\prime}(u) \partial_x u\bigr)\bigr) = 0,
	&
	(t,x)\in \{u > 0\},\\
	\partial_x u = 
    m(u) |\partial_x^3 u - G^{\prime\prime}(u) \partial_x u|^{\alpha-1} \bigl(\partial_x^3 u - G^{\prime\prime}(u) \partial_x u\bigr)
    = 0, 
	&
	t > 0,\ x \in \partial\Omega,
	\\
	u(0,x) = u_0(x),
	&
	x \in \Omega
	\end{cases}
	\end{equation*}
	on the set \(\{u > 0\}\)
	in the sense that $u$ satisfies the equation
	\begin{equation*}
	\int_0^\infty  \langle \partial_t u, \phi\rangle_{W^{1}_{\alpha+1}}\dd t
	-
	\iint_{\{u> 0\}} m(u) |\partial_x^3 u-G^{\prime\prime}(u)\partial_xu|^{\alpha-1}\bigl(\partial_x^3 u-G^{\prime\prime}(u)\partial_xu\bigr) \partial_x \phi\dd x\dd t
	=
	0
	\end{equation*}
	for all \(\phi\in L_{\alpha+1}\bigl((0,\infty);W^1_{\alpha+1,B}(\Omega)\bigr)\) and the energy-dissipation inequality
	\begin{equation*}
	E[u](t)
	+
	\int_0^t \int_{\{u_s > 0\}} m(u) |\partial_x^3 u-G^{\prime\prime}(u)\partial_x u|^{\alpha+1}\dd x\dd s
	\leq
	E[u_0],
	\quad t \in [0,\infty).
	\end{equation*}
\end{theorem}

The concept of weak solutions obtained in Theorem \ref{thm:solution_to_tfe} is `very weak'. It is the same concept of weak solutions that is used in \cite{bernis_higher_1990} and, if $G\equiv 0$, it allows for steady-state solutions of the form \(u(x) = [x-b]_+ \cdot [c- x]_+\), where $b, c \in \Omega$.

\begin{proof}[Proof of Theorem \ref{thm:solution_to_tfe}]
	In Theorem \ref{thm:well-posedness} we showed that \(u^{\sigma}\) is a weak solution to the modified thin-film equation \eqref{eq:PDE_mod}, that is \(u^{\sigma}\) satisfies
	\begin{align*}
	\int_0^\infty \langle \partial_t u^{\sigma},\phi\rangle_{W^1_{\alpha+1}}\dd t 
	- 
	\int_0^\infty \int_{\Omega} 
	m(u^{\sigma}) |\partial_x^3 u^{\sigma} - G^{\prime\prime}_{\sigma}(u^{\sigma})|^{\alpha-1}\bigl(\partial_x^3 u^{\sigma} - G^{\prime\prime}_{\sigma}(u^{\sigma})\bigr) \partial_x \phi\dd x\dd t 
	= 0
	\end{align*}
	for all $\phi \in L_{\alpha+1}\bigl([0,\infty);W^1_{\alpha+1}(\Omega)\bigr)$. Using Proposition \ref{prop:convergence-u} (iii) and Proposition \ref{prop:identification_flux}, we obtain
	\begin{align*}
	\int_{0}^{\infty} \langle \partial_t u^{\sigma},\phi\rangle_{W^1_{\alpha+1}} \dd t & \longrightarrow \int_{0}^{\infty} \langle \partial_t u,\phi\rangle_{W^1_{\alpha+1}} \dd t, \\
	\int_0^\infty \int_{\Omega} 
	m(u^{\sigma}) \Psi\bigl(\partial_x^3 u^{\sigma}-G_{\sigma}^{\prime\prime}(u^{\sigma})\partial_x u^{\sigma}\bigr) \partial_x \phi\dd x\dd t & \longrightarrow \iint_{\{u>0\}}
	m(u) \Psi\bigl(\partial_x^3 u-G^{\prime\prime}(u)\partial_x u\bigr) \partial_x \phi\dd x\dd t,
	\end{align*}
	as \(\sigma \to 0\) for all \(\phi\in L_{\alpha+1}\bigl([0,\infty);W^1_{\alpha+1}(\Omega)\bigr)\). This proves that \(u\) is a weak solution to \eqref{eq:power-law-equation} on \(\{u>0\}\).
	
	To obtain the energy-dissipation inequality, note that since \(u^{\sigma}(t)\to u(t)\) uniformly, by \ref{it:Gsigma4} we obtain that \(G_{\sigma}(u^{\sigma})\) converges to \(G(u)\) pointwise on the set \(\{u>0\}\). By Fatou's lemma, we then obtain
    \begin{equation*}
        \int_{\Omega} G(u(t)) \dd x \leq \liminf\limits_{\sigma\to 0} \int_{\Omega} G_{\sigma}(u_{\sigma}(t))\dd x,
        \quad 
        t \geq 0.
    \end{equation*}
	
	Futhermore, from Proposition \ref{prop:uniform-bounds-sigma} (i) it follows that \(\partial_x u^{\sigma}(t) \rightharpoonup \partial_x u(t)\) in \(L_2(\Omega)\).  Hence, by lower semicontinuity and using the energy-dissipation identity\eqref{eq:EDI_sigma} for the modified problem, we obtain
	\begin{align*}
		& E[u](t)
		+
		\int_0^t \int_{\{u(s)> 0\}} m(u) |\partial_x^3 u-G^{\prime\prime}(u)\partial_xu|^{\alpha+1}\dd x\dd s \\
		&\leq \liminf\limits_{\sigma \to 0}\left[E^{\sigma}[u^{\sigma}](t) + \int_0^t \int_\Omega  m(u^{\sigma})|\partial_x^3 u^{\sigma} - G_{\sigma}^{\prime\prime}(u^{\sigma})\partial_x u^{\sigma}|^{\alpha+1}\dd x\dd s\right] \\
		& = \liminf\limits_{\sigma\to 0} E^{\sigma}[u_0^{\sigma}] \\
		& = E[u_0],	\quad t \in [0,\infty).
	\end{align*}
	This completes the proof.
\end{proof}


Finally, we characterise positive steady states of \eqref{eq:PDE} by positive constants.

\begin{lemma}
Let $G\colon \R \to [0,\infty]$ satisfy \ref{it:G1}--\ref{it:G4}. Then, a positive function $0 < u \in W^3_{\alpha+1}(\Omega)$ is a steady-state solution to \eqref{eq:PDE} if and only if $u \equiv \bar{u}$ is a positive constant.
\end{lemma}


\begin{proof}
\noindent\textbf{(i)} If  $u \equiv \bar{u}$ is a positive constant, then clearly $u$ is a solution \eqref{eq:PDE} with $\partial_t u = 0$.

\noindent\textbf{(ii)} Let now $0 < u \in W^3_{\alpha+1}(\Omega)$ be a positive steady state of \eqref{eq:PDE}. Then 
\begin{equation*}
    u = \argmin_{\substack{v \in H^1(\Omega), \\ \bar{v}=\bar{u}}} E[v].
\end{equation*}
Since $G$ is convex, by Jensen's inequality we find that
\begin{equation*}
    \int_{\Omega} G(\bar{u}) \dd x
    \leq
    \int_{\Omega} G(u) \dd x.
\end{equation*}
This implies in particular that $E[\bar{u}] \leq E[u]$. Since $E[\cdot]$ is strictly convex on the set $\{v \in H^1(\Omega);\ \bar{v}=\bar{u}\}$, the minimiser is unique. Hence $u = \bar{u}$.
\end{proof}

\section{Guaranteed liftoff} \label{sec:lifting}

In this section, we study the solutions to the power-law thin-film equation
\begin{equation}\label{eq:power-law-tfe-lifting}
\begin{cases}
    \partial_t u + \partial_x\bigl(u^{n} |\partial_x^3 u|^{\alpha-1} \partial_x^3 u \bigr) = 0, & t>0,\ x \in \Omega, 
    \\
    \partial_x u = u^{n} |\partial_x^3 u|^{\alpha-1} \partial_x^3 u = 0, & t>0,\ x \in \partial\Omega, \\
    u(0,x) = u_0, & x\in \Omega,
\end{cases}
\end{equation}
obtained in Section \ref{sec:limit}, i.e. for \(G\equiv 0\) and \(m(u)=u^n\). We show that, under certain conditions, solutions emerging from positive initial data with low energy lift up after fixed time. This means that they are close to a constant after a fixed time. In particular, the initial value may be arbitrarily close to zero at one point and will lift up after a fixed time that is independent of the closeness to zero.

The following results are stated for the unit interval $\Omega = (0,1)$ in order to simplify calculations. The same results hold true for any interval $\Omega = (0,L)$ by rescaling.

\begin{theorem}\label{theorem: liftoff}
	Let $\Omega = (0,1)$. Let $n,\alpha>0$ with $2(\alpha+1)>n$. For a mass \(M>0\), define $v(x)= \frac{3M}{2}(1 - x^2)\in H^1(\Omega)$, with
    \begin{equation*}
        M = \bar{v} = \int_{\Omega} v(x) \dd x \quad \text{and} \quad E_0 \coloneq \int_\Omega |v'(x)|^2 \dd x = \frac{9M^2}{2}. 
    \end{equation*}
    Then there exists $t_0=t_0(M,n,\alpha)>0$ such that
	\begin{itemize}
		\item [(i)] whenever $u_0\in H^1(\Omega)$ with $\bar{u}_0 = M$ and $\int_\Omega|u_0'|^2\dd x < E_0$, then any solution $u$ to \eqref{eq:power-law-tfe-lifting} in the sense of Theorem \ref{thm:solution_to_tfe} with initial value $u_0$ satisfies
            $$\min\limits_{x\in \bar{\Omega}}\, u(t,x)\geq \frac{M}{2}, \quad t\geq t_0;$$
		\item [(ii)] whenever $u_0\in H^1(\Omega)$ with $\bar{u}_0 = M$ and $\int_\Omega |u_0'|^2\dd x = E_0$, then 
        $$u_0=v\quad \text{or} \quad u_0 = v(1-\cdot)$$
        and there exists a solution $u$ to \eqref{eq:power-law-tfe-lifting} in the sense of Theorem \ref{thm:solution_to_tfe} with initial value $u_0$ satisfying 
            $$\min\limits_{x\in \bar{\Omega}}\, u(t,x)\geq \frac{M}{2},\quad t\geq t_0.$$
	\end{itemize}
\end{theorem}

\begin{center}
    \begin{figure}[H]
        \includegraphics[scale=1]{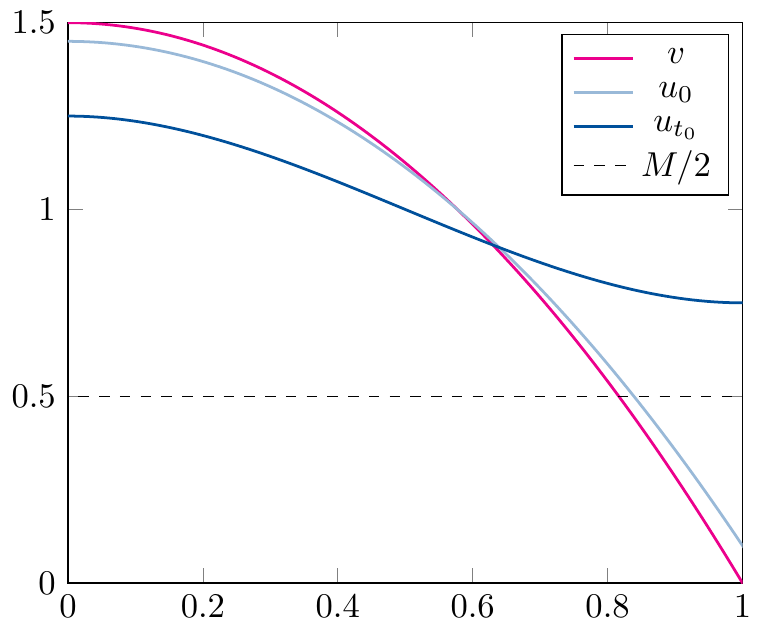}
        \caption{The function $v\in H^1((0,1))$ has a root at $x=1$. Any initial value $u_0$ with less energy produces an almost constant solution after time $t_0>0$ independent of $u_0$. }
        \label{fig: liftoff}
    \end{figure}
\end{center}

\begin{remark}
Theorem \ref{theorem: liftoff} holds true for any interval $\Omega=(0,L)$ by replacing $v$ with $v(\cdot/L)/L$, and $E_0$ with $2E[v(\cdot/L)/L] = \frac{9M^2}{2L^3}$.
A similar result holds for the circle $\R/\Z$, in which case $v(x)= \frac{3M}{2}\max_{z\in \Z}(1-4(x-z)^2)$.
\end{remark}

We note that the function $v$ has a root at $x=1$. We have thus constructed a solution to \eqref{eq:power-law-tfe-lifting} that starts with a root but loses that root immediately. The function $v$ is significant as the lowest-energy function that touches zero, due to the following lemma:

\begin{lemma}\label{lemma: energy vs min}
	Let \(\Omega =(0,1)\). Let $u\in H^1(\Omega)$ with $\bar{u} = M$ and $\min_{x\in\bar{\Omega}} u = \delta \geq 0$. Then
	\begin{equation*}
		\int_\Omega |u'|^2\dd x \geq \int_\Omega |(1-\tfrac{\delta}{M})v'|^2\dd x.
	\end{equation*}
	In other words, $v_\delta = \delta+(1-\frac{\delta}{M})v$ is the minimiser of the Dirichlet energy among all functions $u\in H^1(\Omega)$ with $\min_x u = \delta$ and $\bar{u} = M$.
\end{lemma}

\begin{proof}
The Euler--Lagrange equations tell us that $u''=C$ in $\{u>\delta\}$ and $u'= 0$ on $\partial \Omega \cap \{u>\delta\}$. This means that $u$ consists of at most countably many parabolae with constant $u''$. Among these we pick $v$ as the function with the lowest Dirichlet energy.
\end{proof}

The following lemma tells us that the dissipation $\int_\Omega u^n|\partial_x^3 u|^{\alpha+1}\dd x$ is bounded from below by a power of the minimum of $u$.

\begin{lemma}\label{lemma: dissipation lower bound}
Let \(\Omega=(0,1)\) and  $M, n, \alpha>0$. Then there is a constant $C=C(M,n,\alpha)>0$ such that for any $u\in W^3_{\alpha+1,B}(\Omega)$ with \(\bar{u} = M\) and $\min_{x\in\bar{\Omega}} u = \delta \in (0,M/2)$ we have
\begin{equation*}
\int_\Omega u^n |\partial_x^3u|^{\alpha+1}\dd x \geq \frac{C}{ \log^{\alpha+1}(M/\delta)} \min\{\delta^{n-1 - 2\alpha},1\}.
\end{equation*}
Further, for any $\delta>0$, there exists $u_\delta\in W^3_{\alpha+1,B}(\Omega)$ with these properties such that
\begin{equation*}
    \int_\Omega u_\delta^n |\partial_x^3u_\delta|^{\alpha+1}\dd x \leq C \min\{\delta^{n-1-2\alpha},1\}.
\end{equation*}

\end{lemma}

Before we prove this lemma, we show the following auxiliary result.

\begin{lemma}\label{lemma: point}
Let $u\in C^2(\bar{\Omega})$ with $\min_{x\in\bar{\Omega}} u = \delta > 0,\ \max_{x\in\bar{\Omega}} u = D$ and $u'= 0$ on $\partial \Omega$. Then there exists a point $x_0\in \Omega$ such that
\begin{equation*}
    |u'(x_0)| \geq \frac{D - \delta}{2}
    \quad \text{and} \quad
    u(x_0) u''(x_0) \geq \frac{(D-\delta)^2}{4 \log(D/\delta)}.
\end{equation*}
\end{lemma}

\begin{proof}
Let $x_1\in \Omega$ with $|u'(x_1)| \geq D - \delta$. Without loss of generality assume that $u'(x_1) > 0$. By the Neumann boundary conditions, there exists $x_{-1} < x_1$ such that $u'(x_{-1}) = u'(x_1)/2$ and $u' \geq u'(x_1)/2$ in $(x_{-1},x_1)$.
Define the Hamiltonian 
\begin{equation*}
    H(x) = \tfrac12 |u'(x)|^2 + \log u(x). 
\end{equation*}
Then by the choice of $x_{-1}, x_1$ we have
\begin{equation}\label{eq: H difference}
    H(x_1) - H(x_{-1}) 
    = 
    \tfrac38 |u'(x_1)|^2 + \log\bigl(u(x_1)/u(x_{-1})\bigr).
\end{equation}
On the other hand, we define
\begin{equation*}
    \beta = \max_{x\in [x_{-1},x_1]} u(x)u''(x),
\end{equation*}
so that we have, using that $u'>0$ in \((x_{-1},x_1)\),
\begin{equation}\label{eq: H FTC}
    H(x_1) - H(x_{-1}) 
    = 
    \int_{x_{-1}}^{x_1} u'\, u'' + \frac{u'}{u}\dd x 
    \leq 
    (\beta+1) \int_{x_{-1}}^{x_1} \frac{u'}{u}\dd x = (\beta+1)\log\bigl(u(x_1)/u(x_{-1})\bigr).
\end{equation}
Rearranging \eqref{eq: H difference} and \eqref{eq: H FTC} yields
\begin{equation*}
    \beta \geq \frac{3|u'(x_1)|^2}{8 \log\bigl(u(x_1)/u(x_{-1})\bigr)} \geq \frac{(D-\delta)^2}{4\log(D/\delta)},
\end{equation*}
which proves the claim.
\end{proof}

\begin{proof}[Proof of Lemma \ref{lemma: dissipation lower bound}]
\noindent\textbf{Lower bound. } Let without loss of generality $M=1$. If $\min_{x\in\bar{\Omega}} u(x) = \delta \leq \frac12$, then $\max_{x\in\bar{\Omega}} u(x) \geq 1$. It follows from Lemma \ref{lemma: point} that there is $x_0\in \Omega$ with 
\begin{equation*}
|u'(x_0)| \geq \frac{1}{4} \quad \text{and}\quad  u(x_0)u''(x_0) \geq \frac{1}{16\log(1/\delta)}.
\end{equation*}
We assume without loss of generality that $u'(x_0) > 0$. Then by the Neumann boundary values there exists a point $x_1 \in (x_0,1)$ such that $u''>-1/2$ in $(x_0,x_1)$ and $u''(x_1) = -1/2$. Then by Taylor expansion
\begin{equation}\label{eq: u growth}
    u(x) \geq u(x_0) + u'(x_0)(x-x_0) + \frac{u''(\xi)}{2}(x- x_0)^2 \geq u(x_0) + \frac{x-x_0}{4},
    \quad
    x \in [x_0,x_1].
\end{equation}
By H\"older's inequality we then have
\begin{equation}\label{eq: Cauchy Schwarz dissipation}
    (u''(x_0)+\frac12)^{\alpha+1} = \left(\int_{x_0}^{x_1} -\partial_x^3u \dd x \right)^{\alpha+1} 
    \leq 
    \left(\int_{x_0}^{x_1} u^n|\partial_x^3u|^{\alpha+1}\dd x\right) 
    \left(\int_{x_0}^{x_1} u^{-\frac{n}{\alpha}}\dd x\right)^{\alpha}.
\end{equation}
Thanks to the lower bound \eqref{eq: u growth} on $u$, we have
\begin{equation*}
\int_{x_0}^{x_1} u(x)^{-\frac{n}{\alpha}}\dd x \leq \int_{x_0}^{x_1} \left(u(x_0) + \frac{x-x_0}{4} \right)^{-\frac{n}{\alpha}}\dd x \leq C(n,\alpha)\begin{cases}
1 & \alpha>n\\
\log\left(\frac{u(x_0)+1}{u(x_0)}\right) & \alpha=n\\
u(x_0)^{1-\frac{n}{\alpha}} & \alpha<n
\end{cases}.
\end{equation*}

We treat the three cases separately:

\noindent\textbf{Case 1: $\alpha>n$.} Rearranging \eqref{eq: Cauchy Schwarz dissipation} yields
\begin{equation*}
    \int_{x_0}^{x_1} u^n|\partial_x^3u|^{\alpha+1}\dd x 
    \geq  
    C(n,\alpha) (1/2)^{\alpha+1}.
\end{equation*}

\noindent\textbf{Case 2: $\alpha=n$.} Rearranging \eqref{eq: Cauchy Schwarz dissipation} yields
\begin{equation*}
    \int_{x_0}^{x_1} u^n|\partial_x^3u|^{\alpha+1}\dd x 
    \geq  
    C(n,\alpha) (1/2)^{\alpha+1}\frac{1}{\log^\alpha\left(\frac{u(x_0)+1}{u(x_0)}\right)} \geq \frac{C(n,\alpha)}{\log^\alpha(1/\delta)}.
\end{equation*}

\noindent\textbf{Case 3: $\alpha<n$.} Rearranging \eqref{eq: Cauchy Schwarz dissipation} yields
\begin{equation*}
	\begin{split}
    \int_{x_0}^{x_1} u^n|\partial_x^3u|^{\alpha+1}\dd x 
    &\geq  
    C(n,\alpha) \left(u(x_0)^{\left(1-\frac{n}{\alpha}\right)\alpha} u''(x_0)^{\alpha+1} + u(x_0)^{n-\alpha}\right)
    \\
    &= 
    C(n,\alpha) \left(\bigl(u(x_0)u''(x_0)\bigr)^{\alpha+1} u(x_0)^{n-1-2\alpha} + u(x_0)^{n-\alpha}\right)
    \\
    & \geq  \frac{C(n,\alpha)}{\log^{\alpha+1}(1/\delta)} u(x_0)^{n-1-2\alpha} + C(n,\alpha)u(x_0)^{n-\alpha}.
	\end{split}
\end{equation*}
If the exponent $n-1-2\alpha$ is positive, we simply estimate $u(x_0)^{n-1-2\alpha}\geq \delta^{n-1-2\alpha}$. If the exponent is negative, we have either $u(x_0)\leq 1$, in which case $u(x_0)^{n-1-2\alpha}\geq 1$, or $u(x_0)>1$, in which case $u(x_0)^{n-\alpha} \geq 1$. Either way, the lower bound follows.

\noindent\textbf{Upper bound. } To show the upper bound, we construct for $l\in (0,1/2)$ the function $w_l\geq 0$ which solves 
\begin{equation*}
    w_l(1/2) = 0,
    \quad
    w_l'(1/2) = 0,
    \quad \text{and} \quad
    w_l''(x) = -1 + \frac{1}{l}\left(1-\frac{|x-1/2|}{l}\right)_+,
\end{equation*}
so that $w_l'(0) = w_l'(1) = 0$. We note that $|\partial_x^3w| = \frac{1}{l^2} \textbf{1}_{(1/2 - l, 1/2+ l)}$.
Setting now $\beta_l = \int_\Omega w_l\dd x$ we have that $\beta_l\colon [0,1/2] \to (0,\infty)$ is continuous and thus bounded from below and above.

We define $u_\delta = \delta + \frac{M-\delta}{\beta_l}w_l$ for $l=\delta$ or $l=1/2$. In the first case, since $u_\delta \leq C\delta$ in $(1/2-\delta,1/2+\delta)$, we have
\begin{equation*}
    \int_\Omega u_\delta^n |\partial_x^3u_\delta|^{\alpha+1}\dd x \leq (C\delta)^n \delta^{-1-2\alpha}.
\end{equation*}
In the second case, we have
\begin{equation*}
    \int_\Omega u_\delta^n |\partial_x^3u_\delta|^{\alpha+1}\dd x \leq (C+\delta)^n.
\end{equation*}
This concludes the proof.
\end{proof}

We are now in a position to prove Theorem \ref{theorem: liftoff}

\begin{proof}[Proof of Theorem \ref{theorem: liftoff}]


\noindent\textbf{Part (i). } Let \(u_0\in H^1((0,1))\) with \(\bar{u}_0=M\) and
\begin{equation*}
    \int_{(0,1)}|\partial_x u_0|^2 \dd x < E_0 .
\end{equation*}
Let \(u\in L_{\infty}\bigl([0,\infty);H^1(\Omega)\bigr)\)
with \(\partial_x^3 u \in L_{\alpha+1,\loc}(\{u > 0\})\) and \(\partial_tu\in L_{\alpha+1}\bigl([0,\infty);\bigl(W^1_{\alpha+1}(\Omega)\bigr)'\bigr)\) a solution to \eqref{eq:power-law-tfe-lifting} in the sense of Theorem \ref{thm:solution_to_tfe}. In particular, \(u\) satisfies the energy-dissipation inequality
\begin{equation}\label{eq:EDI-lifting}
    E[u](t) + \int_{0}^{t}\int_{\{u_s>0\}} u^n|\partial_x^3 u|^{\alpha+1} \dd x \dd s \leq E[u_0] < E[v] = \frac{E_0}{2}.
\end{equation}
Combining this with the result of Lemma \ref{lemma: energy vs min}, we obtain that
\begin{equation} \label{eq: energy vs min}
	\min_{x \in \bar{\Omega}} u(t)
	\geq 
	C(M,n,\alpha)\bigl(E[v]-E[u](t)\bigr) > 0.
\end{equation}
Hence, since \(u\) is positive, we obtain the inequality
\begin{equation}\label{eq:EDE-lifting-full}
    E[v] - E[u](t) > E[u_0] - E[u](t) \geq + \int_{0}^{t} \int_{\Omega} u^n|\partial_x^3 u|^{\alpha+1} \dd x \dd t, \quad t>0.
\end{equation}
By Lemma \ref{lemma: dissipation lower bound}, we have in addition
\begin{equation}\label{eq: diss bound}
	\int_0^T \int_{\Omega} u^n|\partial_x^3 u|^{\alpha+1} \dd x\dd t 
	\geq 
	C(M,n,\alpha) f_{n,\alpha}\bigl(\min_{x \in \bar{\Omega}} u(t,x)\bigr),
\end{equation}
with $f_{n,\alpha}(\delta) = \frac{\min\{1,\delta^{n-1-2\alpha}\}}{\log^{\alpha+1}(1/\delta)}$. Observe that $f_{n,\alpha}(\delta) \geq C(n,\alpha)\delta^{1-\eps}$ for e.g. \(\eps = (\alpha+1)-\frac{n}{2}>0\) by the assumption \(2(\alpha+1)>n\). Combining \eqref{eq:EDE-lifting-full}, \eqref{eq: energy vs min} and \eqref{eq: diss bound} yields the differential inequality
\begin{equation*}
    \begin{split}
	\frac{d}{dt}\bigl(E[v] - E[u](t)\bigr) 
	& \geq 
	C(M,n,\alpha) f_{n,\alpha}\bigl(C(M,n,\alpha)(E[v] - E[u](t))\bigr) \\
    &\geq 
    C(M,n,\alpha)\bigl(E[v]-E[u](t)\bigr)^{1-\eps},
    \end{split}
\end{equation*}
and by Gronwall's Lemma
\begin{equation*}
	E[v]-E[u](t) 
	\geq 
    \left(
		(E[v]-E[u_0])^\eps + C(M,n,\alpha)t
	\right)^{1/\eps}.
\end{equation*}
Thus if $t\geq t_0$ for some $t_0>0$, depending only on $M,n,\alpha$, then $E[u](t)$ is small enough such that 
\begin{equation*}
    \min_{x \in \bar{\Omega}} u(t,x)\geq M/2.
\end{equation*}
\medskip

\noindent\textbf{Part (ii). } To show (ii), approximate $u_0$ by $ u_0^\delta = u_0+\delta\geq \delta$ and take $u$ as a limit of the solutions $u^\delta$ to \eqref{eq:power-law-tfe-lifting}. The limit function has small energy at all times $t\geq t_0$ and is thus bounded away from zero.
\end{proof}

\begin{remark}
    Using the lifting property and under the assumptions of Theorem \ref{theorem: liftoff}, we obtain convergence to the steady state \(\bar{u} = M\) and the following convergence rates, cf. \cite{jansen_long-time_2022}:
    \begin{itemize}
        \item algebraic convergence for shear-thinning fluids \(\alpha >1\);
        \item exponential convergence for Newtonian fluids \(\alpha =1\);
        \item convergence in finite time for shear-thickening fluids \(\alpha <1\).
    \end{itemize}
\end{remark}


\appendix

\section{Benamou--Brenier action with superlinear mobility}\label{app:BenamouBrenier}

In this section, we define for two non-negative functions $u_0, u_1\in L_1(\Omega)$, $u_0,u_1 \geq 0$, with $\bar{u}_0 = \bar{u}_1$, the Benamou--Brenier action functional depending on a continuous mobility function $m:[0,\infty) \to [0,\infty)$,
\begin{equation}\label{eq:Wasserstein-metric}
\begin{split}
    d_m^{\frac{\alpha+1}{\alpha}}(u_0, u_1) &\coloneq \inf\biggl\{ \int_0^1 \int_\Omega \frac{|j|^{\frac{\alpha+1}{\alpha}}}{m(u)^{\frac{1}{\alpha}}}\dd x\dd t;\, \partial_t u + \div j = 0,\,
    j\cdot n = 0\text{ on } \partial\Omega,
    \\ 
    &  \qquad \qquad u\geq 0,\, u(0,x) = u_0(x),\, u(1,x) = u_1(x)\biggr\}.
\end{split}
\end{equation}

If the mobility \(m(u)= u^n\) is superlinear, we show that, since this functional lacks convexity simultaneously in \((u,j)\), \(d_m \equiv 0\). So, the natural distance for the setting of gradient flows in metric spaces degenerates, and we have to resort to a different approach to obtain a gradient-flow scheme.

\begin{proposition}\label{prop:metric-degenerate}
	Let \(d\geq 1\), $\ \Omega \subset \R^d$ be open, bounded, and connected, with Lipschitz boundary. Let \(u_0,u_1\in L_1(\Omega)\) with \(\bar{u}_0 =\bar{u}_1\). Assume that \(m \colon [0,\infty) \to [0,\infty)\) satisfies 
		\[\lim_{u\to \infty} \frac{m(u)}{u} = \infty \quad \text{and} \quad m^{-1}(0) = \{0\}.\]
	Then
		\[d_m^{\frac{\alpha+1}{\alpha}}(u_0,u_1) = 0\]
	for all pairs $u_0,u_1\in L^1(\Omega)$ with $\bar{u}_0 = \bar{u}_1$ and $\inf_\Omega u_0, \inf_\Omega u_1 > 0$, where \(d_m\) is the metric introduced in \eqref{eq:Wasserstein-metric}.
\end{proposition}

\begin{remark}
	The proof uses a construction that first concentrates mass locally to achieve a very high density, allowing it to be transported very effectively over large distances, and then dissipates the mass to meet the terminal values. A similar effect occurs in \cite{santambrogio_optimal_2007}.
	
	The opposite effect occurs in dynamic entropic optimal transport and certain mean field games, where at intermediate times the mass is more spread out than at the end points, see e.g. \cite{benamou_variational_2017}.
\end{remark}

Recall the Benamou--Brenier formula for the Wasserstein distance, cf. \cite{benamou_computational_2000},
\begin{equation*}
\begin{split}
    W^{\frac{\alpha+1}{\alpha}}_{\frac{\alpha+1}{\alpha}}(u_0,u_1)  & = \inf\biggl\{ \int_0^1 \int_\Omega \frac{|j|^{\frac{\alpha+1}{\alpha}}}{u^{\frac{1}{\alpha}}}\dd x\dd t;\, \partial_t u + \div j = 0,\,
    j\cdot n = 0\text{ on } \partial\Omega,
    \\ 
    &  \qquad \qquad u\geq 0,\, u(0,x) = u_0(x),\, u(1,x) = u_1(x)\biggr\}.
\end{split}
\end{equation*}

\begin{proof}
	Define $\delta = \frac12 \min(\essinf_{x\in\Omega} u_0, \essinf_{x\in \Omega} u_1)$. For $\eta > 0$ define the grid
	\begin{equation*}
			Z_\eta = \{z\in \eta \Z^d \cap \Omega;\, \dist(x,\partial \Omega) \geq \eta\}.
	\end{equation*}
	Also define $l_\eta>0$ as the longest length of a shortest curve in $\bar{\Omega}$ connecting any point $x\in \Omega$ with some $z\in Z_\eta$, that is
	\begin{equation*}
		l_\eta = \sup_{x\in \Omega} \inf_{z\in Z_\eta} \inf\{L(\gamma);\,\gamma \subset \Omega\text{ is a }C^1\text{-curve connecting }x\text{ and }z\}. 
	\end{equation*}
	By a compactness argument we then have $\lim_{\eta \to 0} l_\eta = 0$.
	
	Define $\tilde u_0 = u_0 - \delta \geq \delta$, $\tilde u_1 = u_1 - \delta \geq \delta$. Then there are measures $\mu_{0,\eta}, \mu_{1,\eta}\in \mathcal{M}_+(Z_\eta)$ of the form
	\begin{equation*}
		\mu_{0,\eta} = \sum_{z\in Z_\eta} \alpha_z \delta_z, \quad \mu_{1,\eta} = \sum_{z\in Z_{\eta}} \beta_z \delta_z
	\end{equation*}
	that satisfy
	\begin{equation*}
		W_{\frac{\alpha+1}{\alpha}}^{\frac{\alpha+1}{\alpha}}(\mu_{0,\eta},\tilde u_0) \leq l_\eta^{\frac{\alpha+1}{\alpha}} \int_\Omega \tilde u_0\dd x \quad \text{and} \quad  W_{\frac{\alpha+1}{\alpha}}^{\frac{\alpha+1}{\alpha}}(\mu_{1,\eta},\tilde u_1) \leq l_\eta^{\frac{\alpha+1}{\alpha}} \int_\Omega \tilde u_1\dd x.
	\end{equation*}
	
	Now choose a coupling $\Gamma\in \mathcal{M}_+(Z_\eta \times Z_\eta)$ of $\mu_{0,\eta}$ and $\mu_{1,\eta}$, e.g. the product measure \(\mu_{0,\eta}\otimes\mu_{1,\eta}\). Define $c_\eta = \min_{z,z'\in Z_\eta\,:\, \Gamma(z,z')>0} \Gamma(z,z')>0$ since the infimum ranges only over finitely many points.
	
	Also for any $z,z'\in Z_\eta$ there exists a $C^1$-curve $\gamma_{z,z'}:[0,1] \to\Omega$ connecting $z$ and $z'$. Let $L_\eta>0$ be the maximal length of such a curve and $d_\eta>0$ the minimal distance from any point on any such curve to $\partial \Omega$.

	Finally, we define for $M>0$
	\begin{equation}
	u_{0,\eta,M} = \delta + \sum_{z\in Z_\eta} \frac{M^d\alpha_z}{\omega_d \eta^d}\mathbf{1}_{B(z,\frac{\eta}{M})} \quad \text{and} \quad u_{0,\eta,M} = \delta + \sum_{z\in Z_\eta} \frac{M^d\beta_z}{\omega_d \eta^d}\mathbf{1}_{B(z,\frac{\eta}{M})}.
	\end{equation}
	
	We now wish to estimate $d_m^{\frac{\alpha+1}{\alpha}}(u_0,u_1)$. Since \(d_m\) satisfies the triangle inequality, it is enough to estimate \(d_m^{\frac{\alpha+1}{\alpha}}(u_0,u_{0,\eta,M})\), \(d_m^{\frac{\alpha+1}{\alpha}}(u_{0,\eta,M},u_{1,\eta,M})\) and \(d_m^{\frac{\alpha+1}{\alpha}}(u_{1,\eta,M},u_1)\). We start with the first and last term since they can be treated analogously. To do so, we note that for any pair $(u,j)$ we have
	\begin{equation}\label{eq:ch6-mobility comparison}
	\int_0^1 \int_\Omega \frac{|j(t,x)|^{\frac{\alpha+1}{\alpha}}}{m(u(t,x))^{\frac{1}{\alpha}}}\dd x\dd t \leq \left[\esssup_{(t,x)\in \supp j} \frac{u(t,x)^{\frac{1}{\alpha}}}{m(u(t,x))^{\frac{1}{\alpha}}} \right] \int_0^1 \int_\Omega \frac{|j(t,x)|^{\frac{\alpha+1}{\alpha}}}{u(t,x)^{\frac{1}{\alpha}}}\dd x\dd t.
	\end{equation}
	
	Observe that by construction it holds
		\[W_{\frac{\alpha+1}{\alpha}} (\mu_{0,\eta},u_{0,\eta,M}-\delta) \leq \frac{\eta}{M}\int_{\Omega}\tilde{u}_0\dd x.\]
	Combining this with the triangle inequality, we obtain
	\begin{align*}
	W_{\frac{\alpha+1}{\alpha}}^{\frac{\alpha+1}{\alpha}}(\tilde u_0, u_{0,\eta, M} -\delta) & \leq  \left(W_{\frac{\alpha+1}{\alpha}}(\tilde u_0, \mu_{0,\eta}) + W_{\frac{\alpha+1}{\alpha}}(\mu_{0,\eta}, u_{0,\eta, M} -\delta) \right)^{\frac{\alpha+1}{\alpha}} \\ & \leq (l_\eta + \frac{\eta}{M})^{\frac{\alpha+1}{\alpha}} \int_\Omega \tilde u_0 \dd x.
	\end{align*}
	Hence, for any \(\eps >0\), there is a distributional solution $j\in L_{\frac{\alpha+1}{\alpha}}([0,1]\times \Omega;\R^d)$ to the continuity equation
	\begin{equation*}
	\begin{cases}
		\partial_t \tilde{u} + \div j = 0,& t>0,\ x\in \Omega,\\
		j\cdot n = 0,&t >0, \ x\in \partial\Omega,\\
		\tilde u(0,x) = \tilde{u}_0(x), & x\in \Omega,\\
		\tilde u(1,x) = u_{0,\eta,M}(x) - \delta, & x\in \Omega,
	\end{cases}
	\end{equation*}
	with
	\begin{equation}
		\int_0^1 \int_\Omega \frac{|j|^{\frac{\alpha+1}{\alpha}}}{\tilde u^{\frac{1}{\alpha}}}\dd x\dd t \leq W^{\frac{\alpha+1}{\alpha}}_{\frac{\alpha+1}{\alpha}}(\tilde{u}_0,u_{0,\eta,M}-\delta) + \eps \leq (l_\eta + \frac{\eta}{M})^{\frac{\alpha+1}{\alpha}} \int_\Omega \tilde u_0 \dd x + \eps.
	\end{equation}
	
	We see that $u(t,x) = \tilde u(t,x) + \delta$ and $j$ together solve the continuity equation with initial and terminal values $u(0,x) = u_0(x)$ and $u(1,x) = u_{0,\eta,M}(x)$. Moreover, by \eqref{eq:ch6-mobility comparison}, we have 
	\begin{equation}
	\int_0^1 \int_\Omega \frac{|j(t,x)|^{\frac{\alpha+1}{\alpha}}}{m(u(t,x))^{\frac{1}{\alpha}}}\dd x\dd t \leq \sup_{s \geq \delta} \frac{s^{\frac{1}{\alpha}}}{m(s)^{\frac{1}{\alpha}}} \left[\left(l_\eta + \frac{\eta}{M}\right)^{\frac{\alpha+1}{\alpha}} \int_\Omega \tilde u_0 \dd x + \eps \right],
	\end{equation}
	so that
	\begin{equation}
	\limsup_{\eta \to 0} \sup_{M\geq 1} d_m^{\frac{\alpha+1}{\alpha}}(u_0,u_{0,\eta,M}) \leq  \eps,
	\end{equation}
	for every \(\eps >0\) arbitrary. We conclude that 
		\[\limsup_{\eta \to 0} \sup_{M\geq 1} d_m^{\frac{\alpha+1}{\alpha}}(u_0,u_{0,\eta,M}) = 0,\]
	and likewise
		\[\limsup_{\eta \to 0} \sup_{M\geq 1} d_m^{\frac{\alpha+1}{\alpha}}(u_1,u_{1,\eta,M}) = 0.\]
	
	Next, we estimate $d_m^{\frac{\alpha+1}{\alpha}}(u_{0,\eta,M},u_{1,\eta,M})$. To this end, we define for $M\geq \frac{\eta}{d_{\eta}}$
	\begin{equation*}
	u(t,x) = \delta + \sum_{z,z'} \Gamma(z,z') \frac{M^d}{\omega_d \eta^d} \mathbf{1}_{B(\gamma_{z,z'}(t),\frac{\eta}{M})}(x),\quad\text{and}\quad
		j(t,x) = \sum_{z,z'} \Gamma(z,z') \frac{M^d}{\omega_d \eta^d} \mathbf{1}_{B(\gamma_{z,z'}(t),\frac{\eta}{M})}(x) \dot \gamma_{z,z'}(t).
	\end{equation*}
	This is clearly a curve connecting $u_{0,\eta,M}$ and $u_{1,\eta,M}$, and
	\begin{equation*}
		\essinf_{(x,t)\in \supp j} u(t,x) \geq c_\eta \frac{M^d}{\omega_d \eta^d} \longrightarrow \infty
	\end{equation*}
	as $M\to \infty$, for every fixed $\eta>0$. By the superlinear growth condition on $m$, it follows that
	\begin{equation*}
		\esssup_{(x,t)\in \supp j} \frac{u(t,x)^{\frac{1}{\alpha}}}{m(u(t,x))^{\frac{1}{\alpha}}} \longrightarrow 0
	\end{equation*}
	as $M\to \infty$, for every fixed $\eta>0$. By \eqref{eq:ch6-mobility comparison}, we have that
	\begin{equation*}
		d_m^{\frac{\alpha+1}{\alpha}}(u_{0,\eta,M}, u_{1,\eta,M}) \leq \left[\esssup_{(x,t)\in \supp j} \frac{u(t,x)^{\frac{1}{\alpha}}}{m(u(t,x))^{\frac{1}{\alpha}}}\right] L_\eta^{\frac{\alpha+1}{\alpha}} \int_{\Omega} \tilde u_0(x)\dd x  \longrightarrow 0
	\end{equation*}
	as $M\to \infty$, for every fixed $\eta>0$.
	
	Finally, we concatenate the curves connecting $u_0$ with $u_{0,\eta,M}$, $u_{0,\eta,M}$ with $u_{1,\eta,M}$ and $u_{1,\eta,M}$ with $u_1$, and estimate
	\begin{align*}
		& d_m^{\frac{\alpha+1}{\alpha}}(u_0,u_1) \\
		&\leq \limsup_{\eta \to 0} \lim_{M\to \infty} C\left(d_m^{\frac{\alpha+1}{\alpha}}(u_0,u_{0,\eta,M}) + d_m^{\frac{\alpha+1}{\alpha}}(u_{0,\eta,M}, u_{1,\eta,M}) + d_m^{\frac{\alpha+1}{\alpha}}(u_{1,\eta,M},u_1) \right) \\
	 & = 0.
	\end{align*}
	This concludes the proof.
\end{proof}

\begin{remark}
	If $u_0$ and $u_1$ are smooth, the connecting curve $(u,j)$ can be chosen smooth in space-time via mollification and the dominated convergence theorem.
\end{remark}



\printbibliography

\end{document}